%% file: main.tex
\pdfoutput=1
\documentclass[12pt]{amsart}
\input{sty-categorical-gln-generous.tex}
\DeclareRobustCommand{\gobblefive}[5]{}
\newcommand*{\SkipTocEntry}{\addtocontents{toc}{\gobblefive}}%
\makeatletter
\def\subsection{\@startsection{subsection}{2}%
  \z@{.5\linespacing\@plus.7\linespacing}{.3\linespacing}%
  {\normalfont\bfseries}}
\makeatother
\usepackage{fullpage}
\usepackage[normalem]{ulem}
\author{Konrad Zou}
\title{Categorical local Langlands for \texorpdfstring{\(\mathrm{GL}_{\MakeLowercase{n}}\)}{GLn} for parameters of Langlands-Shahidi type with integral coefficients}
\begin{document}	
	\begin{abstract}
		We prove the categorical form of Fargues' geometrization conjecture for \(\mathrm{GL}_n\) along \(L\)-parameters of Langlands-Shahidi type for rational, torsion, and integral coefficients.
		Additionally, we prove that in this case the categorical equivalence is \(t\)-exact, which yields new torsion vanishing results in the cohomology of unitary Shimura	 varieties.
	\end{abstract}
	\maketitle
	\tableofcontents
	\include{introduction}
        \include{langlands-shahidi-type}
	\include{reduction-to-local-langlands-in-families}
        \include{irreducible}
	\include{explicit-hecke-operators}

	\include{proof-of-categorical-conjecture}
	\include{t-exactness}
	\include{applications}
	\printbibliography%
\end{document}

%% file: sty-categorical-gln-generous.tex
\usepackage{comment}
\usepackage{amsmath,amssymb,amsthm,mathtools}
\usepackage{tikz-cd}
\usepackage{tikz}
\usepackage[utf8]{inputenc}
\usepackage[british]{babel}
\usepackage{lmodern}
\usepackage[T1]{fontenc}
\usepackage[babel=true]{microtype}
\usepackage[mathscr]{euscript}
\usepackage{hyperref}
\usepackage[capitalise]{cleveref}
\usepackage[backend=biber, bibencoding=utf8,style=alphabetic,maxnames=99,maxalphanames=99]{biblatex}
\usepackage{relsize}

\DefineBibliographyStrings{english}{%
	bibliography = {References},
}

\binoppenalty=\maxdimen
\relpenalty=\maxdimen

\theoremstyle{plain}
\newtheorem{thm}{Theorem}[section]
\newtheorem{lem}[thm]{Lemma}
\newtheorem{cor}[thm]{Corollary}

\newtheorem{conj}[thm]{Conjecture}

\theoremstyle{definition}
\newtheorem{defn}[thm]{Definition}
\newtheorem{notn}[thm]{Notation}

\newtheorem{rem}[thm]{Remark}

\theoremstyle{plain} %
\newcommand{\thistheoremname}{}
\newtheorem{genericthm}[thm]{\thistheoremname}
\newenvironment{custom}[1]
  {\renewcommand{\thistheoremname}{#1}%
   \begin{genericthm}}
  {\end{genericthm}}
\theoremstyle{remark}

\DeclareMathOperator{\Hom}{Hom}

\newcommand*{\from}{\colon}
\newcommand*{\defword}[1]{\emph{#1}}

\newcommand*{\defined}{\coloneqq}

\mathchardef\mhyphen="2D
\newcommand*{\xto}[1]{\xrightarrow{#1}}
\DeclareMathOperator*{\colim}{colim}

\DeclareMathOperator{\id}{id}
\DeclareMathOperator{\im}{im}
\DeclareMathOperator{\rk}{rk}

\makeatletter
\def\mydefbb#1{\expandafter\def\csname bb#1\endcsname{\mathbb{#1}}}
\def\mydefallbb#1{\ifx#1\mydefallbb\else\mydefbb#1\expandafter\mydefallbb\fi}
\mydefallbb ABCDEFGHIJKLMNOPQRSTUVWXYZ\mydefallbb
\makeatother
\newcommand*{\injto}{\hookrightarrow}

\newcommand*{\surjto}{\twoheadrightarrow}

\newcommand*{\Mod}{\mathrm{Mod}}

\newcommand*{\rhom}{\mathrm{RHom}}

\DeclareMathOperator{\spec}{Spec}
\DeclareMathOperator{\res}{Res}

\DeclareMathOperator{\rep}{Rep}

\DeclareMathOperator{\bun}{Bun}
\newcommand*{\lheck}{\overleftarrow{h}}
\newcommand*{\rheck}{\overrightarrow{h}}
\newcommand*{\hck}{\mathrm{Hck}}

\newcommand*{\Div}{\mathrm{Div}}
\DeclareMathOperator{\spd}{\mathrm{Spd}}
\newcommand*{\supp}{\mathrm{supp}}

\newcommand*{\oo}{\mathcal{O}}

\newcommand*{\localhck}{\mathcal{H}\mathrm{ck}}

\DeclareMathOperator{\locsys}{Par}

\DeclareMathOperator{\cind}{c-ind}
\DeclareMathOperator{\perf}{Perf}
\DeclareMathOperator{\qcoh}{QCoh}
\DeclareMathOperator{\coh}{Coh}
\newcommand*{\nilp}{{\mathrm{Nilp}}}

\newcommand{\op}{{\mathrm{op}}}

\DeclareMathOperator{\End}{End}

\DeclareMathOperator{\prl}{{\mathcal{P}r^L}}
\DeclareMathOperator{\D}{\mathcal{D}}

\newcommand{\qc}{\mathrm{qc}}

\DeclareMathOperator{\cofib}{cofib}

\DeclareMathOperator{\ind}{Ind}
\DeclareMathOperator{\nind}{nInd}
\DeclareMathOperator{\nres}{nRes}

\usepackage{scalerel, stackengine}
\stackMath
\newcommand\widecheck[1]{%
\savestack{\tmpbox}{\stretchto{%
  \scaleto{%
    \scalerel*[\widthof{\ensuremath{#1}}]{\kern-.6pt\bigwedge\kern-.6pt}%
    {\rule[-\textheight/2]{1ex}{\textheight}}%
  }{\textheight}%
}{0.5ex}}%
\stackon[1pt]{#1}{\scalebox{-1}{\tmpbox}}%
}

\DeclareMathOperator{\ad}{ad}

\DeclareMathOperator{\speccenter}{\mathcal{Z}^{spec}}
\DeclareMathOperator{\berncenter}{\mathfrak{Z}}
\newcommand{\coarse}{{\mathrm{coarse}}}
\newcommand{\irred}{{\mathrm{irred}}}

\newcommand{\ren}{{\mathrm{ren}}}

\newcommand{\GL}{\mathrm{GL}}

\newcommand{\std}{{\mu_{\mathrm{std}}}}
\newcommand{\stdd}{{\mu_{\mathrm{std}^\vee}}}
\newcommand{\alg}{{\mathrm{alg}}}

\DeclareMathOperator{\numin}{\nu_{\min}}
\DeclareMathOperator{\numax}{\nu_{\max}}
\DeclareMathOperator{\Lie}{Lie}
\newcommand{\lift}{{\mathrm{lift}}}
\DeclareMathOperator*{\Ind}{Ind}
\DeclareMathOperator*{\sht}{Sht}
\newcommand{\func}[4]{i_{#2}^{\ren!}\oo(#3,#4)*i_{#1*}^\ren}
\DeclareMathOperator{\rkmin}{\rk_{\min}}
\DeclareMathOperator{\degmin}{\deg_{\min}}
\newcommand{\nmin}{n_{\min}}

\newcommand{\LSt}{{\mathrm{LSt}}}

%% file: introduction.tex
\section{Introduction}
Let \(E\) be a non-archimedean local field of characteristic 0 of residue characteristic \(p\) with residue field \(\bbF_{q}\) and let \(G=\GL_n\).
To this one can attach the Langlands dual group, which is \(\widehat{G}=\GL_n\).

The local Langlands correspondence analyzes the set of irreducible \(\bbC\)-valued continuous representations of \(G(E)\) by constructing a map
\begin{equation*}
    \begin{tikzcd}[row sep=tiny]
        \{\text{Irreducible representations of }G(E)\} \arrow[rr] &  & {\{\text{Weil-Deligne representations }(\rho,N)\}} \\
        \pi \arrow[rr, maps to]                                   &  & {\varphi_\pi=(\rho,N).}                           
        \end{tikzcd}
\end{equation*}
Here \(N\in\Lie(\widehat{G})\) is a nilpotent endomorphism and \(\rho\from W_E\to\widehat{G}(\bbC)\) is a \(n\)-dimensional representation of \(W_E\) where \(N\) and \(\rho\) satisfy the relation \(\ad(\rho(\sigma))N=q^{|\sigma|}N\).
The map \(|\cdot|\from W_E\to\bbZ\) is the valuation of the corresponding element of \(E^\times\) under local class field theory.
This Weil-Deligne representation \(\varphi_\pi\) is called the \(L\)-parameter of \(\pi\).
In our case this map is in fact a bijection.

These mapping should have some properties, for example it should be bijective and preserve various analytical invariants like $L$- and $\epsilon$-factors.
This falls outside the scope of this work, instead let us focus on the question how one would construct such a map.

Fix an isomorphism $\iota\from\mathbb{C}\cong\overline{\mathbb{Q}_\ell}$ with \(\ell\neq p\).
Firstly, using such an isomorphism, by Grothendieck's quasi-unipotence theorem the datum of a Weil-Deligne representation is equivalent to a continuous morphism $\varphi\from W_E\to\GL_n(\overline{\mathbb{Q}_\ell})$.
The datum of the element $N$ is taken by $\log\varphi(s)$ where $s$ is some generator of the tame inertia.

This suggest to realize this correspondence in the \'etale cohomology of some space.
It turns out that the correct space is \(\mathrm{LT}_{\infty}\), the Lubin-Tate tower at infinite level.
It carries an action of \(G(E)\times D^\times\times W_E\) where \(D^\times\) is the division algebra of invariant \(1/n\).
Indeed, in \cite{harris-taylor} they compute that for supercuspidal \(\pi\) we have
\begin{equation*}
    R\Gamma_c(\mathrm{LT}_\infty,\overline{\mathbb{Q}_\ell})\otimes_{\mathcal{H}(G(E))}^L\pi=\mathrm{JL}(\pi)\boxtimes\varphi_\pi[1-n](\frac{1-n}{2}).
\end{equation*}
Here \(\mathcal{H}(G(E))\) is the Hecke algebra for \(G(E)\) and \(\mathrm{JL}(\pi)\) is the so called Jacquet-Langlands transfer of \(\pi\).

Using Bernstein-Zelevinsky's theory of segments we can extend this correspondence to all irreducible representations \(\pi\).
The idea is to use that the local Langlands correspondence is compatible with parabolic induction and use the fact that any irreducible representation \(\pi\) occurs as some irreducible consituent of \(\Ind_{\GL_{n_1}(E)\times\dots\GL_{n_r}(E)}^{\GL_n(E)}\pi_1\boxtimes\dots\boxtimes\pi_r\), where each \(\pi_i\) is a supercuspidal representation of \(\GL_{n_i}(E)\).
The simplest case occurs when this parabolic induction is already irreducible.
There is a certain genericity assumption one can impose on \(L\)-parameters called ``Langlands-Shahidi'' type, which guarantees that the corresponding representation is an irreducible parabolic induction.
We will recall more details of this condition later on in the introduction.
We record for future reference in this introduction that the cohomology is concentrated in middle degree, as \(\dim(\mathrm{LT}_{\infty})=n-1\).

In \cite{geometrization}, Laurent Fargues and Peter Scholze have reformulated this circle of ideas as a geometric Langlands conjecture for the Fargues-Fontaine curve.
For this, they consider the stack \(\bun_G\) of \(G\)-bundles on the Fargues-Fontaine curve and look at the category \(\D(\bun_G)\) of lisse sheaves on \(\bun_G\).
This category should be equivalent to \(\Ind\coh_{\nilp}^{\qc}(\locsys_{\widehat{G}})\), where \(\locsys_{\widehat{G}}\) is the stack of \(L\)-parameters for \(\widehat{G}\).
Given a \(\bbZ_{\ell}\)-algebra \(\Lambda\) there is a scheme \(Z^1(W_E,\widehat{G})\) whose \(\Lambda\)-valued points are really continuous cocycles with values in \(\widehat{G}(\Lambda)\).
We can then form the quotient stack under the \(\widehat{G}\)-action to obtain \(\locsys_{\widehat{G}}\).
Instead one can also form the GIT quotient to form a scheme \(\locsys_{\widehat{G}}^\coarse\).
This is an infinite union of affine schemes and its closed points parametrize semi-simple \(L\)-parameters and in general it parametrizes pseudo-representations.
There is a natural map \(\locsys_{\widehat{G}}\to\locsys_{\widehat{G}}^\coarse\).
On closed points it maps an \(L\)-parameter to its simplification, and in general it maps representations to its corresponding pseudo-representations.

There is a canonical action of \(\Ind\perf(\locsys_{\widehat{G}})\) on \(\Ind\coh_\nilp^\qc(\locsys_{\widehat{G}})\), by Ind-extending the natural action of \(\perf(\locsys_{\widehat{G}})\) given by the tensor product.
They also construct an action of \(\Ind\perf(\locsys_{\widehat{G}})\) on \(\D(\bun_G)\), the so called ``spectral action''.
For this they prove that for each \(V\in\rep(\widehat{G})\) one can attach a functor \(T_V\from\D(\bun_G)\to\D(\bun_G)^{BW_E}\).
Taking \(V=\mathrm{std}\), one recovers the computation of \(R\Gamma_c(\mathrm{LT}_{\infty},\overline{\bbQ_{\ell}})\) via the isomorphism
\begin{equation}\label{eq: stalk hecke operator}
    i_{\oo(1/n)}^*T_{\mathrm{std}}i_{\oo^n!}\pi[n-1](\frac{n-1}{2})=R\Gamma_c(\mathrm{LT}_{\infty},\overline{\bbQ_{\ell}}).
\end{equation}
Here \(\oo^n\) denotes the trivial rank \(n\) vector bundle and \(\oo(1/n)\) the semistable vector bundle of slope \(1/n\) occuring in the Fargues-Fontaine classifcation of vector bundles on the Fargues-Fontaine curve.
Varying over all \(V\) and also allowing representations for \(\widehat{G}^I\) for \(I\) a finite set  they show this is indeed equivalent to giving an action of \(\Ind\perf(\locsys_{\widehat{G}})\) on \(\D(\bun_G)\).
The details can be found in \cite[Chapter IX]{geometrization}.
With this setup, the main conjecture of \cite{geometrization} is roughly (ignoring the question of coefficients, which can be taken to be a \(\bbZ_{\ell}[\sqrt{q},\mu_{p^\infty}]\)-algebra\footnote{The inclusion of $\mu_{p^\infty}$ is not strictly necessary, however the set of Whittaker data we will consider later will be empty otherwise.} \(\Lambda\) where \(\ell\neq p\)).
\begin{conj}\label{conj: full categorical conjecture}
    Choose a Whittaker datum \((U,\psi)\).
    There is an \(\Ind\perf(\locsys_{\widehat{G}})\)-linear equivalence 
    \begin{equation*}
        \bbL_G\from\Ind\coh^\qc_\nilp(\locsys_{\widehat{G}})\simeq\D(\bun_G)
    \end{equation*}
    preserving compact objects, mapping \(\oo\) to \(\mathcal{W}\), where \(\mathcal{W}\in\D(\bun_G)\) is a certain object constructed from the Whittaker datum\footnote{In the geometric Langlands literature this functor would be called $\bbL_G^L$. To simplify the notation in this paper, we drop the ``$L$'' superscript.}.
\end{conj}
Philosophically, this claims that two sheaves of categories over \(\locsys_{\widehat{G}}\) are equivalent as sheaves of categories over \(\locsys_{\widehat{G}}\), see \cite{sheaves-of-categories} for a way to make this mathematically precise.
While the full conjecture appears to be out of reach, one might hope to be able to prove a version of the conjecture when restricted to an open substack \(V\subset\locsys_{\widehat{G}}\).

For technical reasons it is easier to restrict to an open subscheme \(V\subset\locsys_{\widehat{G}}^\coarse\).
This should be formulated via base changing along the symmetric monoidal pullback functor \(\Ind\perf(\locsys_{\widehat{G}}^\coarse)\to\Ind\perf(V)\).
Let us write 
\[\D^V(\bun_G)\defined\D(\bun_G)\otimes_{\Ind\perf(\locsys_{\widehat{G}}^\coarse)}\Ind\perf(V).\]

Thus, the conjecture would look like this:
\begin{conj}\label{conj: categorical conjecture open substack}
    Choose a Whittaker datum \((U,\psi)\).
    There is an \(\Ind\perf(p^{-1}(V))\)-linear equivalence 
    \begin{equation*}
        \bbL_G^V\from\Ind\coh^\qc_\nilp(p^{-1}(V))\simeq\D^V(\bun_G)
    \end{equation*}
    preserving compact objects, mapping \(\oo\) to \(\mathcal{W}^V\), where \(p\from\locsys_{G}\to\locsys_{\widehat{G}}^\coarse\) is the canonical map.
\end{conj}
Here \(\mathcal{W}^V\) is the image of \(\mathcal{W}\) under the functor \(\D(\bun_G)\to\D^V(\bun_G)\).
Since \(V\) is an open subscheme, this functor is a Verdier quotient, so we have a fully faithful right adjoint \(\D^V(\bun_G)\subset\D(\bun_G)\).
Note that, if $V\subset V'$ are open subschemes of $\locsys_{\widehat{G}}^\coarse$, \cref{conj: categorical conjecture open substack} for $V$ is implied by \cref{conj: categorical conjecture open substack} for $V'$ and if $V=\locsys_{\widehat{G}}^\coarse$ then \cref{conj: categorical conjecture open substack} is precisely \cref{conj: full categorical conjecture}.
Of course, now the task is to find an open subscheme \(V\subset\locsys_{\widehat{G}}^\coarse\) that is as big as possible, such that we can prove the conjecture.

We now restrict our attention to \(\Lambda\) a field.
 Note that, \(\D(\bun_G)\) has a semi-orthogonal decomposition into \(\D(G_b(E))\) for \(b\in B(G)\), and thus \(\D^V(\bun_G)\) also admits a semi-orthogonal decomposition into \(\D(G_b(E))\otimes_{\Ind\perf(\locsys_{\widehat{G}}^\coarse)}\Ind\perf(V)\) for \(V\) an open subscheme of \(\locsys_{\widehat{G}}^\coarse\)\footnote{Since it is a full subcategory the orthogonality relations are preserved. 
As $i_b^*$ is $\ind\perf(\locsys_{\widehat{G}}^\coarse)$-linear, it maps $\D(\bun_G)\otimes_{\Ind\perf(\locsys_{\widehat{G}}^\coarse)}\otimes\Ind\perf(V)$ into $\D(G_b(E))\otimes_{\Ind\perf(\locsys_{\widehat{G}}^\coarse)}\Ind\perf(V)$. From this, joint conservativity of $i_b^*$ follows.}.

It has been discovered by Linus Hamann, that for \(\LSt\subset\locsys_{\widehat{G}}^\coarse\) the locus of parameters of Langlands-Shahidi type, that the semi-orthogonal decomposition should split.
This would reduce the conjecture to representation theory of \(G_b(E)\) for various \(b\) which is more tractable than \(\D(\bun_G)\).
In our situation, parameters of Langlands-Shahidi type are the semisimple \(L\)-parameters \(\varphi\), such that when one writes them as \(\varphi\cong\varphi_1\oplus\dots\oplus\varphi_r\) for \(\varphi_i\) irreducible, we have \(\varphi_i\ncong\varphi_j\) and \(\varphi_i\ncong\varphi_j(1)\) for \(i\neq j\).
This generalizes the condition of decomposed generic in \cite{caraiani-scholze} beyond the case \(r=n\) and we recall that the representations corresponding to these \(L\)-parameters are irreducible parabolic inductions.
We obtain one of our main results:
\begin{thm}[\cref{thm: categorical equivalence}]
    Let \(\LSt\) be the locus of \(\locsys_{\widehat{G}}\) consisting of \(L\)-parameters of Langlands-Shahidi type.
    There is an \(\Ind\perf(\LSt)\)-linear equivalence 
    \begin{equation*}
        \bbL_G^\LSt\from\Ind\perf^\qc(\LSt)\simeq\D^\LSt(\bun_G)
    \end{equation*}
    preserving compact objects, mapping \(\oo\) to \(\mathcal{W}^\LSt\).
\end{thm}
We will compute that \(\nilp|_\LSt=\{0\}\) (\cref{lem: nilpotent singular support Langlands-Shahidi}), so by \cite[Theorem VIII.2.9]{geometrization} \(\Ind\perf^\qc=\Ind\coh^\qc_{\nilp}\), so this verifies \cref{conj: categorical conjecture open substack} on \(\LSt\).
On \(\bun_G\) there is a perverse \(t\)-structure by setting \(\dim(b)=-\langle2\rho_G,\nu_b\rangle\), see \cite[Proposition 1.2.1.]{beijing-notes}.
This equips \(\D^\LSt(\bun_G)\) with a perverse \(t\)-structure by defining the \(t\)-structure on \(\Ind\perf(\locsys_{\widehat{G}}^\coarse)\) and \(\Ind\perf(\LSt)\) by declaring the connective part to be generated by connective perfect complexes.
One can also ask if the equivalence in \cref{conj: categorical conjecture open substack} is \(t\)-exact.
We do not expect this to be true for \(V=\locsys_{\widehat{G}}^\coarse\) (unless \(G\) is a torus, where it is \cite[Theorem 6.4.5.]{categorical-fargues-tori}). However, we obtain:
\begin{thm}[\cref{thm: t-exact}]\label{thm: introduction t-exact}
    Let \(\Lambda\) be a field.
    Then \(\bbL_G^\LSt\) of the previous theorem is \(t\)-exact.
\end{thm}
Together with \cref{eq: stalk hecke operator}, this yields that the cohomology of the Lubin-Tate tower is concentrated in middle degree.
In characteristic 0 we use this fact in the proof, so the argument is circular.
This result seems to be new in positive characteristic however and the argument uses the full power of the categorical equivalence to apply a reduction argument from characteristic 0 to positive characteristic.
In \cite{beijing-notes}, David Hansen proposes many conjectures for Langlands-Shahidi type parameters (which he calls generous), and these theorems resolve them all, see \cref{sec: applications}.
Furthermore, using \cite[Section 6]{torsionvanishing} we see that the \(t\)-exactness for fields of characteristic \(\ell\) implies torsion vanishing results for PEL type Shimura varieties in type A.
Applications of \cref{thm: introduction t-exact} will be collected in \cref{sec: applications}.
\SkipTocEntry
\subsection*{Structure of the proof}
Let \(p\from\locsys_{\widehat{G}}\to\locsys_{\widehat{G}}^\coarse\) be the natural map.
We fix a Levi \(L=\GL_{n_1}\times\dots\times\GL_{n_r}\) and consider the locus \(\LSt\subset\locsys_{\widehat{G}}^\coarse\) of parameters of Langlands-Shahidi type that factor through \(\widehat{L}\) but no smaller Levi.

It suffices to prove \cref{conj: categorical conjecture open substack} one parameter at a time by using \cite[Observation 21.4.4.]{stack-of-restricted-variation}, which also applies in our context. 
We recall this in \cref{sec: reduction}.
Therefore, we want to show that 
\begin{equation*}
    \bbL_G^\varphi\from\Ind\coh_{\nilp}^\qc(p^{-1}(\LSt)\times_{\LSt}\LSt^{\wedge}_{\varphi})\to\D(\bun_G)\otimes_{\Ind\perf(\LSt)}\Ind\perf(\LSt^{\wedge}_{\varphi})
\end{equation*}
is an equivalence, where \(\varphi\) is an \(L\)-parameter valued in an algebraically closed field \(k\) which lives over \(\Lambda\).
There are now many observations that make the situation easier.

First, note that, \(\coh_\nilp^\qc(p^{-1}(\LSt)\times_{\LSt}\LSt_\varphi^{\wedge})\simeq\perf^\qc(p^{-1}(\LSt)\times_{\LSt}\LSt_{\varphi}^{,\wedge})\), thus \(\bbL_G^\varphi\) is pinned down by the condition that \(\bbL_G\) was \(\Ind\perf(p^{-1}(\LSt))\)-linear.
Moreover, when \(\varphi\in\LSt\), then \(\mathrm{Stab}_{\widehat{G}}(\varphi)=Z(\widehat{L})\) and we have 
\begin{equation*}
    p^{-1}(\LSt)\times_{\LSt}\LSt^{\wedge}_{\varphi}\cong BZ(\widehat{L})\times \LSt_{\varphi}^{\wedge}.
\end{equation*}
Such an isomorphism is not canonical, but if one fixes a lift \(\varphi_{\widehat{L}}\from W_E\to\widehat{L}(k)\) of $\varphi$ along the inclusion \(\widehat{L}\subset\widehat{G}\), then this determines a canonical isomorphism.
Hence, we have 
\begin{equation*}
    \Ind\coh_{\nilp}^\qc(p^{-1}(\LSt)\times_{\LSt}\LSt^{\wedge}_{\varphi})\simeq\prod_{\chi\in X^*(Z(\widehat{L}))}\Ind\perf^\qc(\LSt_{\varphi}^{\wedge}).
\end{equation*}
Thus, there are two types of objects to consider to entirely understand 
\begin{equation*}
    \Ind\coh_{\nilp}^\qc(p^{-1}(\LSt)\times_{\LSt}\LSt^{\wedge}_{\varphi}):
\end{equation*}
\begin{enumerate}
    \item The sheaves \(\oo(\chi)\), coming from \(\chi\in X^*(Z(\widehat{L}))\),
    \item Sheaves \(A\in\Ind\perf^\qc(\LSt_{\varphi}^{\wedge})\).
\end{enumerate}
Let us write \(\D(\bun_G)^\wedge_{\varphi}\defined\D(\bun_G)\otimes_{\Ind\perf(\LSt)}\Ind\perf(\LSt^{\wedge}_{\varphi})\).
Assume we have a similar decomposition
\begin{equation*}
    \D(\bun_G)^\wedge_{\varphi}\simeq\prod_{\xi\in X^*(Z(\widehat{L}))}(\D(\bun_G)^\wedge_{\varphi})_{\xi}.
\end{equation*}
Proving the equivalence would then consider of two parts:
\begin{enumerate}
    \item Understand the action of \(\oo(\chi)\), in particular checking that \(\oo(\chi)*-\) does restrict to a functor \((\D(\bun_G)^\wedge_{\varphi})_{\xi}\to(\D(\bun_G)^\wedge_{\varphi})_{\xi\chi}\),
    \item Show that \(\Ind\perf^\qc(\LSt_{\varphi}^{                            \wedge})\simeq(\D(\bun_G)^\wedge_{\varphi})_{0}\).
\end{enumerate}
The second point is simple, it turns out that there is a commutative diagram of functors
\begin{equation*}
    \begin{tikzcd}
        \Ind\perf^\qc(\LSt_{\varphi}^{\wedge}) \arrow[r] \arrow[d, hook] & (\D(\bun_G)^\wedge_\varphi)_0 \arrow[d, "{\Hom(\mathcal{W},-)}", hook] \\
        \Mod_{\speccenter(G)} \arrow[r]                                      & \Mod_{\End(\mathcal{W})}                                              
        \end{tikzcd}
\end{equation*}
where the vertical arrows are fully faithful.
The bottom arrow is an equivalence via the isomorphism of rings \(\speccenter(G)\cong\End(\mathcal{W})\), which is an easy consequence of local Langlands in families and unicity of the Whittaker model, see \cref{lem: generic local langlands in families}.

We are left with understanding the origin of the categories \((\D(\bun_G)^\wedge_{\varphi})_{\xi}\) and the action of \(\oo(\chi)\) on them.
Since \(\varphi\) is of Langlands-Shahidi type, we expect a decomposition 
\begin{equation*}
    \D(\bun_G)^\wedge_{\varphi}\simeq\prod_{b\in B(G)}\D(\bun_G^b)^\wedge_{\varphi}
\end{equation*}
with \(\D(\bun_G^b)^\wedge_{\varphi}\) defined similarly to \(\D(\bun_G)^\wedge_{\varphi}\).
Any object in this category admits a filtration by Schur-irreducible objects with \(L\)-parameter \(\varphi\).
Thus, by consideration of \(L\)-parameters we can see that \(\D(\bun_G^b)^\wedge_{\varphi}=0\) if \(b\notin\im(B(L)_{\mathrm{basic}}\to B(G))\).
There is a canonical map \(X^*(Z(\widehat{L}))\to\im(B(L)_{\mathrm{basic}}\to B(G))\) coming from the isomorphism \(X^*(Z(\widehat{L}))\cong B(L)_{\mathrm{basic}}\). However, this map is not injective.
There is an injective map \(X^*(Z(\widehat{L}))\injto B(G)\times W_G\), where \(W_G\) is the Weyl group of \(G\).
Also, note while \(\speccenter(G)\) acts on \(\D(\bun_G^b)\), the action factors through an action of \(\speccenter(G_b)\) on this category.
Together with this observation, one might guess that the correct decomposition is 
\begin{equation*}
    \D(\bun_G)^\wedge_{\varphi}\simeq\prod_{\xi\in X^*(Z(\widehat{L}))}\D(\bun_G^{b_{\xi,L}})^\wedge_{\eta\circ\varphi_{\widehat{L}}^{w_{\xi}}}.
\end{equation*}
where \((b_{\xi,L},w_{\xi})\) is the image of \(\xi\) under \(X^*(Z(\widehat{L}))\injto B(G)\times W_G\), \(\varphi_{\widehat{L}}^{w_{\xi}}\) is the conjugation of \(\varphi_{\widehat{L}}\) with \(w_\xi\) and \(\eta\) is just \(\widehat{L}^{w_\xi}\subset\widehat{G_{b_{\xi,L}}}\).
Of course, one needs to check that such an \(\eta\) exists.
One should also compare this also with the construction of the \(B(G)\)-parametrized local Langlands correspondence of \cite[Theorem 1.1.]{bm-o-parametrization}, when restricted to parameters of Langlands-Shahidi type.
It follows from the semi-orthogonal decomposition on $\D(\bun_G)$, that the factors generate \(\D(\bun_G)^\wedge_{\varphi}\), so we only need to check the vanishing of hom-spaces.

It turns out that the vanishing of hom-spaces and also a complete description of the action of \(\oo(\chi)\) on the factors follows from the following theorem:
\begin{thm}[special case of \cref{cor: explict hecke * stalk}]\label{thm: putative theorem}
    The functor \(i_b^{\ren*}\oo(\chi)*i_{1!}^\ren\) vanishes unless \(b=b_{\chi,L}\), and in this case, it restricts to a functor 
    \begin{equation*}
        i_{b_{\chi,L}}^{\ren*}\oo(\chi)*i_{1!}^\ren\from\D(G(E))^\wedge_{\varphi}\to\D(G_{b_{\xi,L}}(E))^\wedge_{\eta\circ\varphi_{\widehat{L}}^{w_{\chi}}}.
    \end{equation*}
\end{thm}
On first pass the reader should ignore the renormalizations on the functors, they only differ by a shift and twist from the usual functors.
Assuming this theorem, we obtain that
\begin{equation*}
    \D(\bun_G)^\wedge_{\varphi}\simeq\prod_{\chi\in X^*(Z(\widehat{L}))}\oo(\chi)*\D(\bun_G^1)^\wedge_{\varphi}
\end{equation*}
from which the vanishing of hom-spaces can be deduced, as well as a description of \(\oo(\chi)\) on the factor categories.
Later we observe that this is simple if \(\varphi\) is irreducible, in this case, the decomposition comes from the connected components of \(\bun_G\) and the computation of \(\oo(\chi)\) follows from \cite[Lemma 5.3.3.]{categorical-fargues-tori}.
Here \(t\)-exactness is obtained in \cite[Theorem 1.8.]{t-exact} in characteristic 0, and one can deduce from this the same \(t\)-exactness in characteristic \(\ell\) by lifting to characteristic 0 (following a similar strategy as in \cite[Lemma 3.7.]{hansenjohansson}).
This will be done in \cref{sec: t-exactness}.

\SkipTocEntry
\subsection*{Spectral action at Langlands-Shahidi type parameters}
We are left with understanding the stalks of the Hecke operators \(\oo(\chi)\).
Recall how the Hecke action is constructed.
We have the following diagram of \(v\)-stacks
\begin{equation*}
    \begin{tikzcd}
        & \hck_G \arrow[ld, "\lheck"'] \arrow[r, "q"] \arrow[rd, "\rheck"] & \localhck_G            \\
 \bun_G &                                                                  & \bun_G\times\Div^1
 \end{tikzcd}
\end{equation*}
(the precise construction is not relevant for us). Note that geometric Satake provides a functor \(\mathcal{S}\from\rep(\widehat{G})\to\D(\localhck_G)\), so we obtain functors 
\begin{equation*}
    T_VA\defined\rheck_{\natural}(q^*\mathcal{S}(V)\otimes\lheck^*A)
\end{equation*}
for \(V\in\rep(\widehat{G})\) and \(A\in\D(\bun_G)\), which one can show lands in \[\D(\bun_G)^{BW_E}\subset\D(\bun_G\times\Div^1).\footnote{\text{In fact, in this case, the categories are equivalent, however there is a version of this diagram where} \((\Div^1)^I\) \text{appears for \(I\) a finite set, and then the inclusion is strict.}} \]
This relates to the spectral action as follows:

Let \(f\from\locsys_{\widehat{G}}\to B\widehat{G}\) be the canonical projection.
By identifying \(\rep(\widehat{G})\) as a full subcategory of \(\perf(B\widehat{G})\), for \(V\in\rep(\widehat{G})\) we get \(f^*V\in\perf(\locsys_{\widehat{G}})^{BW_E}\) and thus \(f^*V*A\in\D(\bun_G)^{BW_E}\), and the spectral action is defined to satisfy \(f^*V*A\cong T_VA\).

Fix the isomorphism \(\bbG_m^r\cong Z(\widehat{L})\) given by embedding via diagonal matrices. %
This allows us to identify \(X^*(Z(\widehat{L}))\) with \(\bbZ\)-valued functions on \(\{1,\dots,r\}\), which we equip with the lexicographic ordering.
One Hecke operator that is easy to compute is \(T_{\det}\), this corresponds to \(\oo(\chi_{\det})\) where \(\chi_{\det}(i)=n_i\) for all \(i\).
Thus we may assume \(\chi(i)\geq 0\) for all \(i\) in the proof of \cref{thm: putative theorem}, which allows for an induction on \(|\chi|\).
For this induction, we need to understand \(i_{b'}^{\ren*}\oo(\chi_c)*i_{b!}^\ren\from\D(\bun_G^b)^\wedge_{\eta\circ\varphi_{\widehat{L}}}\to\D(\bun_G^{b'})\), where \(\chi_c\) is the indicator function on the element \(c\).

Let us write \(\varphi_{\widehat{L}}=\varphi_1\times\dots\times\varphi_r\).
The isomorphism \(p^{-1}(\LSt)\times_{\LSt}\LSt^{\wedge}_{\varphi}\cong BZ(\widehat{L})\times \LSt^{\wedge}_{\varphi}\) shows us that \(\oo(\chi)\) occurs as a direct summand of \(T_{\mathrm{std}}\).
Here, another advantage of the Langlands-Shahidi type condition comes in.
It implies that \(\rhom_{W_E}(\varphi_i,\varphi_j)=0\) for \(i\neq j\).
Therefore, keeping the Galois action in mind, the functor \(\oo(\chi)\) can be recovered from \(T_{\mathrm{std}}\) by taking the \(\varphi_c\)-isotypic component, at least up to tensoring with \(\rhom_{W_E}(\varphi_c,\varphi_c)\).
Thus we need to understand \(i_{b'}^{\ren*}T_{\mathrm{std}}i_{b!}\).
When \((b,b',\mathrm{std})\) are Hodge-Newton reducible, this can be computed (up to some twists and shifts), from \(i_{\theta'}^{\ren*}T_{V}i_{\theta!}\from\D(\bun_M^\theta)^\wedge_{\eta\circ\varphi_{\widehat{L}}^w}\to\D(\bun_M)\), where \(\theta,\theta'\in B(M)\) are reductions of \(b\) and \(b'\), and \(w\in W_G\) is some element such that \(\widehat{L}^w\subset\widehat{M_{\theta}}\), and the highest weight of \(V\) is a reduction to \(\widehat{M}\) of the highest weight \(\mu\) of \(\mathrm{std}\).

We also note, that for \(b,b'\) such that there is no modification between the attached vector bundles \(\mathcal{E}_b\) and \(\mathcal{E}_{b'}\) bounded by \(\mu\), \(i_{b'}^{\ren*}T_{\mathrm{std}}i_{b!}\) vanishes automatically.
The main observation due to \cite[Section 6]{nguyen} is, that by carefully writing \(\chi=\prod_{i\leq|\chi|}\chi_{c_i}\), for some sequence \(c_i\in\{1,\dots,r\}\) and an analysis which types of modifications can occur, that we will only ever consider situations where \(i_{b'}^{\ren*}T_{\mathrm{std}}i_{b!}\) vanishes due to the non-existence of a suitable modification or where \((b,b',\mathrm{std})\) are Hodge-Newton reducible.
This allows us to reduce to the case where \(\varphi\) is irreducible in the end, which we have already resolved at the end of the previous section.
We roughly follow the steps \cite[Section 6]{nguyen}, but our setup is slightly different, and therefore we execute the inductive argument in \cref{sec: computation Hecke operators} following the ideas of Nguyen.
\begin{rem}
    In this paper, we will assume that we are working with a non-archimedean local field of characteristic 0.
    The reason for this is mainly that local Langlands in families are only available in this case.
    However, the spectral action constructs a map \(\speccenter(G)\to\End(\mathcal{W})\) without any restriction on the characteristic.
    Using the theory of close local fields we expect to remove the assumption that $E$ have characteristic 0.
    This will be done in a future update.
\end{rem}

\SkipTocEntry
\subsection*{Notations and conventions}
By ``category'' we will mean ``\(\infty\)-category'' and all functors are implicitly derived.
However, the fiber product of schemes and Artin stacks will be underived.
Since we work with \(\GL_n\), the notation simplifies a bit, otherwise (co)invariants for the Galois group would appear.
\begin{itemize}
    \item Let \(E\) be a non-archimedean local field of characteristic 0 and residue characteristic \(p\) and residue field \(\bbF_q\), we fix a prime \(\ell\neq p\), and we write \(W_E\) for its Weil group.

    \item Let \(\Lambda\) be a \(\bbZ_\ell\)-algebra unless otherwise mentioned.
    
    \item Let \(G=\GL_n\), \(L\subset G\) is a Levi, with corresponding parabolic \(P'\) and unipotent \(N\), so that \(LN=P'\). 
    We will also have a Levi \(M\) of the form \(\GL_{m_1}\times\GL_{m_2}\) with parabolic \(P\).
    
    \item We fix the upper triagular Borel \(B\subset G\), its unipotent radical \(U\) and the torus \(T\subset B\).
    
    \item Write \(r_{\ad}^N\) for \(\widehat{L}\) acting on \(\Lie(\widehat{N})\) via the adjoint action, for \(\theta\in X^*(Z(\widehat{L}))\), let \(r_{\ad}^{N,\theta}\) denote the maximal subrepresentation of \(r_{\ad}^N\) such that \(Z(\widehat{L})\) acts via \(\theta\).
    
    \item For an Artin \(v\)-stack \(X\) let \(\D(X,\Lambda)\) denote the category of lisse sheaves as defined in \cite[Definition VII.6.1.]{geometrization}, sometimes we will just write \(\D(X)\).
    
    \item For a connected reductive group \(H/E\), let \(\D(H(E),\Lambda)\) denote the derived category of smooth \(H(E)\)-representations, sometimes we will just write \(\D(H(E))\).
    
    \item We write \(B(H)\) for the Kottwitz set of \(H\) and \(H_b\) for the automorphisms of the associated isocrystal for \(b\). We write \(\kappa(b)\) for the Kottwitz point and \(\nu_b\) for the Newton point. However, we will twist the Kottwitz point by \(-1\) compared to the usual conventions using isocrystals, as this is more compatible with bundles on the Fargues-Fontaine curve. For example, there is no opposite in \cite[Proposition III.3.6. (ii)]{geometrization} or \cite[Lemma 5.3.3.]{categorical-fargues-tori} and for example, explicitly we have \(\kappa(\oo(1))=1\).
    
    \item For \(b\in B(H)\), let \(i_b\from\bun_H^b\to\bun_H\) be the canonical locally closed immersion and \(s_b\from [*/H_b(E)]\to\bun_H^b\) be the splitting of \(\bun_H^b\to [*/H_b(E)]\) as constructed in \cite[Proposition III.5.3.]{geometrization}.
    
    \item For \(?\in\{\natural,!,*\}\) we define \(i_{b?}^\ren\from\D(H_b(E),\Lambda)\to\D(\bun_H,\Lambda)\) via \begin{equation*}
        i_{b?}^\ren A=i_{b?}(s_*(A\otimes\delta_b^{-1/2}))[-\langle 2\rho_G,\nu_b\rangle]
    \end{equation*}
    and similarly for \(?\in\{*,!\}\) we define \(i_b^{\ren?}A=\delta_b^{1/2}\otimes s_b^*i_b^{?}A[\langle2\rho_G,\nu_b\rangle]\), where \(\delta_b\) is a modulus character defined as in \cite[Definition 3.14.]{imaihamann}.%
    
    \item We write \(\locsys_{\widehat{H}}\) for the stack of \(L\)-parameters, denoted \(Z^1(W_E,\widehat{H})/\widehat{H}\) in \cite{geometrization}. 
    
    \item We write \(\speccenter(H)\defined\oo(\locsys_{\widehat{H}})\) and write \(\locsys_{\widehat{H}}^\coarse\) for coarse moduli space, denoted \(Z^1(W_E,\widehat{H})/\!/\widehat{H}\) in \cite{geometrization}.
    Given a compact open pro-\(p\) subgroup \(P^e\subset W_E\) write \(\locsys_{\widehat{H},P^e}\) and \(\locsys_{\widehat{H},P^e}^{\coarse}\) for the corresponding open substacks.
    
    \item For an open subscheme \(V\subset\locsys_{\widehat{G}}^\coarse\) we write
    \begin{equation*}
        \D^V(\bun_G,\Lambda)\defined\D(\bun_G,\Lambda)\otimes_{\Ind\perf(\locsys_{\widehat{G}}^\coarse)}\Ind\perf(V)
    \end{equation*}
    and write \(A\mapsto A^V\) for the functor \(\D(\bun_G,\Lambda)\to\D^V(\bun_G,\Lambda)\).
    For \(V=\locsys_{\widehat{G},P^e}^\coarse\) we will instead also write \(\D^{P^e}(\bun_G,\Lambda)\defined\D^V(\bun_G)\) and write \(A\mapsto A^{P^e}\).
    \item We will use cohomological conventions, and our shift functor is \(M[1]\defined\cofib(M\to 0)\). In particular we have \(H^{-1}(M[1])=H^0(M)\).
    \item Given an Artin stack \(\mathfrak{X}\) and a closed stack \(\mathfrak{Z}\subset\mathfrak{X}\) we will talk about the formal completion \(\mathfrak{X}^\wedge_{\mathfrak{Z}}\).
    By this we mean the functor on the category of affine schemes given by maps to \(\mathfrak{X}\) that set-theoretically factor though \(|\mathfrak{Z}|\).
    Recall that \(\perf(\mathfrak{X}^\wedge_{\mathfrak{Z}})\) consists of those perfect complexes set-theoretically supported on \(|\mathfrak{Z}|\).
\end{itemize}
\SkipTocEntry
\subsection*{Acknowledgements}
We thank Peter Scholze for many helpful discussions and for reading the preliminary versions of this manuscript, vastly improving its structure. We would also like to thank David Hansen, Kieu Hieu Nguyen and Annie Littler reading a preliminary draft of this paper. We thank David Hansen and Christian Johansson for sharing a preliminary draft of \cite{hansenjohansson}.
I am thankful to the Max Planck Institute
for Mathematics for their hospitality. The author was supported by the DFG Leibniz Preis through Peter
Scholze.

%% file: langlands-shahidi-type.tex
\section{Parameters of Langlands-Shahidi type for \texorpdfstring{\(\GL_n\)}{GLn}}
All schemes and stacks in this section are considered to live over $\bbZ_\ell$.
\begin{defn}
    We define the substack \(\locsys_{\widehat{G}}^\irred\) as the complement over the images of \(\locsys_{\widehat{P}}\to\locsys_{\widehat{G}}\) with \(P\subset G\) running over the proper parabolics.
    This is open as there are only finitely many parabolics to consider and \(\locsys_{\widehat{P}}\to\locsys_{\widehat{G}}\) is proper.
    A similar definition can be made for any Levi \(L\subset G\).
    We call a parameter \(\varphi\) irreducible if the corresponding point on \(\locsys_{\widehat{G}}\) lies in \(\locsys_{\widehat{G}}^\irred\).
\end{defn}
\begin{rem}
    An irreducible parameter is indeed irreducible as a representation of \(W_E\), in particular it is semisimple.
\end{rem}
\begin{lem}\label{lem: etale morphism algebraic stack}
    Let \(f\from\mathfrak{X}\to\mathfrak{Y}\) be a morphism locally of finite presentation between algebraic stacks such that the cotangent complex \(L_f\) vanishes.
    Then \(f\) is \'etale in the sense of \cite[Tag 0CIL]{stacks-project}.
\end{lem}
\begin{proof}
    Since \(L_f\) vanishes, the cotangent complex of \(\Delta_f\) vanishes too.
    This implies that \(\Delta_f\) is unramified, so the morphism is DM, so base changing to an DM-stack we can assume that both \(\mathfrak{X}\) and \(\mathfrak{Y}\) are DM-stacks, but then the claim is easy.
\end{proof}
\begin{defn}
    We define a substack \(\locsys_{\widehat{G}}^{\widehat{L}\mathrm{\mhyphen LSt}}\) in the following way.
    Consider the morphism \(i\from\locsys_{\widehat{L}}^\irred\to\locsys_{\widehat{G}}\).
    Let \(U\) be the locus where the cotangent complex \(L_i\) vanishes. This is open since the support of perfect complexes is always closed. %
    We get a map \(U\to\locsys_{\widehat{G}}\), this is \'etale in the sense of \cite[Tag 0CIL]{stacks-project} by \cref{lem: etale morphism algebraic stack}, thus the image is an open substack we call \(\locsys_{\widehat{G}}^{\widehat{L}\mathrm{\mhyphen LSt}}\).
    We say a parameter \(\varphi\) is Langlands-Shahidi type with cuspidal support \(\widehat{L}\), if the corresponding point in \(\locsys_{\widehat{G}}\) lies in \(\locsys_{\widehat{G}}^{\widehat{L}\mathrm{\mhyphen LSt}}\).
\end{defn}
\begin{rem}
    Since parameters of Langlands-Shahidi type are the image of semisimple parameters, they are themselves semisimple.
    In particular the property can be checked on the coarse moduli space and by computation of the cotangent complex we see that this agrees with the notion \cite[Definition 6.2.]{torsionvanishing}.
\end{rem}
\begin{defn}
    We say that a semisimple parameter \(\varphi\) valued in a \(\bbZ_\ell\)-field has cuspidal support \((\widehat{L},\varphi_{\widehat{L}})\), if \(\varphi_{\widehat{L}}\) is an irreducible parameter for \(\widehat{L}\) and \(\varphi\) arises from \(\varphi_{\widehat{L}}\) by composing with \(\widehat{L}\subset\widehat{G}\).
    We sometimes say that \(\varphi\) has cuspidal support \(\widehat{L}\) if we would rather not stress the existence of \(\varphi_{\widehat{L}}\).
\end{defn}
\begin{lem}\label{lem: langlands-shahidi type explicit}
    Let \(\varphi\) be a semisimple parameter with cuspidal support \(\GL_{n_1}\times\dots\times\GL_{n_r}\) valued in an algebraically closed \(\bbZ_\ell\)-field \(k\), giving rise to a decomposition \(\varphi\cong\varphi_1\oplus\dots\oplus\varphi_r\) with each \(\varphi_i\) irreducible and \(\dim(\varphi_i)=m_i\).
    Then \(\varphi\) is Langlands-Shahidi type with cuspidal support \(\widehat{L}\) if and only if \(\varphi_i\ncong\varphi_j\) for \(i\neq j\) and \(\varphi_i\ncong\varphi_j(1)\) for \(i\neq j\).
    In this notation \(\Hom(\varphi_i,\varphi_j)=0\) if \(i\neq j\).
\end{lem}
\begin{proof}
    Recall that the complex that needs to vanish is \(R\Gamma(W_E,r_{\ad}^N\varphi)\oplus R\Gamma(W_E,r_{\ad}^N\varphi^\vee)\).
    This happens if for all \(\theta\in X^*(Z(\widehat{L}))\) the complex \(R\Gamma(W_E,r_{\ad}^{N,\theta}\varphi)\oplus R\Gamma(W_E,r_{\ad}^{N,\theta}\varphi^\vee)\) vanishes.
    Since the Euler characteristic vanishes and since the cohomological dimension of \(W_E\) is 2 this is equivalent to \(H^0(R\Gamma(W_E,r_{\ad}^{N,\theta}\varphi))\oplus H^0(R\Gamma(W_E,r_{\ad}^{N,\theta}\varphi^\vee))=0\) and similarly \(H^2\) vanishing for all \(\theta\in X^*(Z(\widehat{L}))\).
    The claim about \(H^0\) is equivalent to \(H^0(R\Gamma(W_E,\varphi_i\otimes\varphi_j^\vee))\) vanishing for \(i\neq j\).
    This is equivalent to \(\varphi_i\ncong\varphi_j\).
    The claim about \(H^2\) by local Tate duality is equivalent to \(H^0(R\Gamma(W_E,\varphi_i\otimes\varphi_j(1)^\vee))\) vanishing for \(i\neq j\).
    This is equivalent to \(\varphi_i\ncong\varphi_j(1)\).
\end{proof}
\begin{lem}\label{lem: nilpotent singular support irreducible}
    Let \(\varphi\) be an irreducible parameter valued in an algebraically closed \(\bbZ_\ell\)-field \(k\)
    Then \(x_{\varphi}^*\mathrm{Sing}_{\locsys_{\widehat{G}}/\bbZ_l}\cap\mathcal{N}^*_{\widehat{G}}\otimes_{\bbZ_\ell}k=\{0\}\).
\end{lem}
\begin{proof}
    We have to check that the only element in the nilpotent cone that is stable under the \(W_E\)-action given by \(\varphi\) and twisted by \(\bbZ_\ell(1)\) is \(\{0\}\).
    Assume we have a \(v\in\mathcal{N}^*\) stabilized by \(W_E\).
    Let \(P_v\) denote the parabolic associated with \(v\), see \cite[Proposition 5.9.]{nilpotent-orbit-in-rep-theory}. 
    Note this requires that \(\ell\) is good for \(\widehat{G}\), but since \(G=\GL_n\), any prime is good for \(\widehat{G}\).
    \cite[Remark 5.10.]{nilpotent-orbit-in-rep-theory} states that the stabilizer of the linear subspace spanned by \(v\) is contained in \(P_v\), so that \(\varphi\) factors through \(P_v\).
    By irreducibility, \(P_v=\widehat{G}\), which implies that \(v=0\).
\end{proof}
\begin{lem}\label{lem: nilpotent singular support Langlands-Shahidi}
    Let \(\varphi\) be a parameter of Langlands-Shahidi type valued in an algebraically closed \(\bbZ_\ell\)-field \(k\).
    Then \(x_{\varphi}^*\mathrm{Sing}_{\locsys_{\widehat{G}}/\bbZ_\ell}\cap\mathcal{N}^*_{\widehat{G}}\otimes_{\bbZ_\ell}k=\{0\}\).
\end{lem}
\begin{proof}
    By definition the morphism \(\locsys_{\widehat{L}}^\irred\times_{\locsys_{\widehat{G}}}\locsys_{\widehat{G}}^{\widehat{L}\mathrm{\mhyphen LSt}}\to\locsys_{\widehat{G}}^{\widehat{L}\mathrm{\mhyphen LSt}}\) is \'etale (since the cotangent complex vanishes, and it is clearly locally of finite presentation), we reduce to \cref{lem: nilpotent singular support irreducible}.
\end{proof}
\begin{lem}\label{lem: residual gerbes langlands-shahidi}
    The stabilizer of a parameter of Langlands-Shahidi type with cuspidal support \(\widehat{L}\) is \(Z(\widehat{L})\).
\end{lem}
\begin{proof}
    This follows from \cref{lem: langlands-shahidi type explicit}.
\end{proof}
\begin{cor}\label{cor: formal completion langlands-shahidi}
    Let \(\varphi\) be a parameter of Langlands-Shahidi type valued in a \(\bbZ_{\ell}\)-field \(k\) with cuspidal support \((\widehat{L},\varphi_{\widehat{L}})\).
    Then \((k\times_{\bbZ_{\ell}}\locsys_{\widehat{L}})^\wedge_{\varphi_{\widehat{L}}}\cong(k\times_{\bbZ_{\ell}}\locsys_{\widehat{G}})^\wedge_\varphi\).
\end{cor}
\begin{proof}
    By definition the morphism \(\locsys_{\widehat{L}}^\irred\times_{\locsys_{\widehat{G}}}\locsys_{\widehat{G}}^{\widehat{L}\mathrm{\mhyphen LSt}}\to\locsys_{\widehat{G}}^{\widehat{L}\mathrm{\mhyphen LSt}}\) is \'etale, so it suffices to check that the residual gerbes are isomorphic, this is \cref{lem: residual gerbes langlands-shahidi}.
\end{proof}
\begin{lem}
    The fiber of \(\locsys_{\widehat{G}}\to\locsys_{\widehat{G}}^\coarse\) at a parameter \(\varphi\) of Langlands-Shahidi type is \(BS_{\varphi}\).
\end{lem}
\begin{proof}
    In general, this is \cite[Remark 6.5.]{torsionvanishing}, however in this case, it is also an easy consequence of \cref{cor: formal completion langlands-shahidi}.
\end{proof}
\begin{cor}\label{cor: localize at langlands-shahidi fiber product}
    We have \((k\times_{\bbZ_{\ell}}\locsys_{\widehat{G}})^\wedge_\varphi\cong\locsys_{\widehat{G}}\times_{\locsys_{\widehat{G}}^\coarse}(k\times_{\bbZ_{\ell}}\locsys_{\widehat{G}}^\coarse)^{\wedge}_{\varphi}\).
\end{cor}
\begin{lem}\label{lem: locsys irred is gerbe}
    The map \(\locsys_{\widehat{G}}^\irred\to\locsys_{\widehat{G}}^{\irred,\coarse}\) is a gerbe banded by \(Z(\widehat{G})\).
\end{lem}
\begin{proof}
    This map identifies with \(Z^1(W_E,\widehat{G})/\widehat{G}\to Z^1(W_E,\widehat{G})/\widehat{G}_{\ad}\) restricted to the locus of irreducible parameters, which is clearly a gerbe banded by \(Z(\widehat{G})\).
\end{proof}
\begin{lem}\label{lem: isotypic decomposition Langlands-Shahidi type}
    Let \(\varphi\) be a parameter of Langlands-Shahidi type with cuspidal support \((\widehat{L},\varphi_{\widehat{L}})\).
    The isomorphism \((k\times_{\bbZ_{\ell}}\locsys_{\widehat{L}})^\wedge_{\varphi_{\widehat{L}}}\cong(k\times_{\bbZ_{\ell}}\locsys_{\widehat{G}})^\wedge_\varphi\) of \cref{cor: formal completion langlands-shahidi} and \cref{lem: locsys irred is gerbe} provides a map \(f\from(k\times_{\bbZ_{\ell}}\locsys_{\widehat{G}})^\wedge_\varphi\to BS_{\varphi}\), this allows us to define line bundles \(\oo(\chi,\varphi)\defined f^*\chi\) for \(\chi\in X^*(S_{\varphi})\).
    Then \(T_V\cong \bigoplus_{\chi\in\mathrm{Irr}(S_{\varphi})}\Hom_{S_{\varphi}}(\chi,V)\boxtimes \oo(\chi,\varphi)*-\) equivariant for the \(W_E\)-action on both sides.
\end{lem}
\begin{proof}
    This is just the isotypic decomposition of \(V\) restricted to \(S_{\varphi}\).
\end{proof}
\begin{cor}\label{cor: hom galois representation}
    Consider the situation in \cref{lem: isotypic decomposition Langlands-Shahidi type}.
    Let \(\varphi=\varphi_1\times\dots\times\varphi_r\) be the decomposition of \(\varphi\) into irreducible summands coming from the fact that \(\varphi\) is an irreducible parameter for \(\widehat{L}\).
    We have a natural isomorphism \(\rhom_{W_E}(\varphi_c,T_\std(-))\cong K\otimes\oo(\chi_c,\varphi)*-\), where \(K\defined\rhom_{W_E}(\varphi_c,\varphi_c)\) is a perfect complex.
\end{cor}
\begin{proof}
    This is a direct consequence of \cref{lem: isotypic decomposition Langlands-Shahidi type} and \cref{lem: langlands-shahidi type explicit}.
\end{proof}
\begin{lem}\label{lem: support sheaf with L-parameter of Langlands-Shahidi type}
    Let \(A\in\D(\bun_G,k)^\wedge_{\varphi}\) for an algebraically closed \(\Lambda\)-field \(k\), with \(\varphi\) a parameter of Langlands-Shahidi type with cuspidal support \(\widehat{L}\).
    Then \(\supp(A)\subset\im(B(L)_{\mathrm{basic}}\to B(G))\).
\end{lem}
\begin{proof}
    For \(L=\widehat{G}\), this is the supercuspidal case, this follows from compatibility with passage from \(\widehat{G_b}\) to \(\widehat{G}\).

    In general, this follows from the Bernstein decomposition of \(\D(G_b(E))\) and compatibility with the Fargues-Scholze parameters with parabolic induction and the supercuspidal case, where the Bernstein decomposition for modular coefficients is the main result of \cite{secherre-stevens}.
\end{proof}
\begin{lem}\label{lem: Langlands-Shahidi type generic}
    Let \(\pi\in\D(G(E),k)\) be an irreducible representation over an algebraically closed \(\bbZ_\ell\)-field \(k\) such that its parameter \(\varphi\) is of Langlands-Shahidi type.
    Then \(\pi\) is generic.
    In particular, given a Whittaker datum \((U,\psi)\) we have \(\Hom(\cind_{U(E)}^{G(E)}\psi,\pi)\neq 0\).
\end{lem}
\begin{proof}
    This follows from \cite[Th\'{e}or\`{e}me 9.10.]{lift-irreducible}.
    We remark that usually we call a representation \(\pi\) generic for some fixed Whittaker datum, if \(\Hom(\pi,\ind_{U(E)}^{G(E)}\psi)\neq 0\).
    Following the argument of \cite[Corollary A.0.2.]{beijing-notes} we see that if \(\Hom(\pi,\ind_{U(E)}^{G(E)}\psi)\neq 0\), then \(\Hom(\cind_{U(E)}^{G(E)}\psi^{-1},\pi^\vee)\neq 0\).
    Taking contragradients is an involution on the set of irreducible representations whose parameter is of Langlands-Shahidi type, and for \(G=\GL_n\) any two Whittaker data are \(G(E)\)-conjugate, which gives us the final claim.
\end{proof}

%% file: reduction-to-local-langlands-in-families.tex
\section{Reduction to generic local Langlands in families}\label{sec: reduction}
The material here regarding abstract category theory follows the discussion in the proof of \cite[Observation 21.4.4.]{stack-of-restricted-variation}.

\begin{rem}\label{rem: remove sqrt q}
    The results in this section except \cref{lem: generic local langlands in families} work unconditionally for arbitrary quasi-split groups $G$, and the last one is conjecturally true.
    If there is a lift of the cocharacter $\widehat{\rho}_{\ad}$ to $\widehat{\rho}'\from\bbG_m\to\widehat{G}$ one can drop all the $\sqrt{q}$.
    This is the case for example for $G=\GL_n$.
    In particular in this case the spectral action is already defined over \(\bbZ_\ell\).
\end{rem}

\begin{defn}
    Let \(Y\) be an eventually coconnective noetherian derived scheme.
    Then \(\qcoh(Y)\) is generated by \((i_y)_*\oo_{\spec(k)}\) running through the geometric points \(y\from\spec(k)\to Y\).
\end{defn}
\begin{proof}
    By eventual coconnectivity we may assume that \(Y\) is classical.
    In this case, it follows by noetherian induction.
\end{proof}
\begin{defn}
    Let \(\mathcal{C}\) be a presentable category with action of \(\qcoh(Y)\) for \(Y\) an eventually coconnective noetherian derived scheme.
    Let \(y\from\spec(k)\to Y\) be a point.
    Then define \(\mathcal{C}_{y}\defined\Mod_k\otimes_{\qcoh(Y)}\mathcal{C}\), which we call the \defword{fiber} of \(\mathcal{C}\) at \(y\).
    Here the tensor product is the Lurie tensor product of presentable categories.
\end{defn}
We recall the following:
\begin{lem}\label{lem: fppf descent sheaves of categories}
    Let \(Y\) be a qcqs derived scheme.
    Fix a presentable \(\qcoh(Y)\)-linear category \(\mathcal{C}\).\footnote{By this we mean an object of \(\Mod_{\qcoh(Y)}(\prl)\). Recall that these objects are automatically stable categories.}
    The assignment 
    \begin{equation*}
        X\mapsto\qcoh(X)\otimes_{\qcoh(Y)}\mathcal{C}
    \end{equation*}
    on the category of qcqs derived schemes over \(Y\) is an fppf sheaf.
    If \(\mathcal{C}\) is dualizable in \(\prl\) then this also holds for all derived schemes.
\end{lem}
\begin{proof}
    This follows from \cite[Corollary 1.5.4.]{sheaves-of-categories} combined with \cite[Theorem 2.1.1.]{sheaves-of-categories}.
    If \(\mathcal{C}\) is dualizable in \(\prl\), then by rigidity of \(\qcoh(Y)\) it is dualizable in \(\Mod_{\qcoh(Y)}(\prl)\) and we can apply \cite[Lemma 1.3.5.]{sheaves-of-categories}.
\end{proof}
\begin{lem}\label{lem: reduction to algebraic closure}
    Let \(y\from\spec(k)\to Y\) be a point, and \(y'\from\spec(k')\to Y\) be a point obtained by precomposing \(y\) with a morphism \(\spec(k')\to\spec(k)\).
    Then \(\mathcal{C}_y\to\mathcal{C}_{y'}\) is conservative.
\end{lem}
\begin{proof}
    This follows from the fact that \(F\from\Mod_k\to\Mod_{k'}\) is a 2-retraction in the sense that it has a right adjoint \(G\) and the map \(\id\to GF\) is a retract.
    This property is preserved under \(\mathcal{C}\otimes_{\Mod_k}-\) and 2-retractions are conservative.
\end{proof}
\begin{lem}
    Let \(Y\) be an eventually coconnective noetherian derived scheme.
    Let \(\mathcal{C}\) be a \(\qcoh(Y)\)-linear presentable category. 
    The functors \(\mathcal{C}\to\mathcal{C}_{y}\) running through geometric points \(y\from \spec(k)\to Y\) are conservative.
\end{lem}
\begin{proof}
    Since \(\mathcal{C}\) is a \(\qcoh(Y)\)-linear presentable category, it is a stable category.
    This means we only need to check that the functors  \(\mathcal{C}\to\mathcal{C}_{y}\) can detect the zero object.
    Note that \(\oo_{Y}\) can be written as a colimit of objects of the form \(y_*\oo_{\spec(k)}=y_*y^*\oo_{Y}\), in particular \(c\in\mathcal{C}\) can be written as a colimit of the form \((y_*\otimes\id) \circ(y^*\otimes \id) c\), so the claim follows.
\end{proof}
\begin{lem}\label{lem: reduction to individual fibers}
    Let \(Y\) be an eventually coconnective noetherian derived scheme.
    Assume we have an arrow \(f\from\mathcal{C}\to\mathcal{C}'\) of \(\qcoh(Y)\)-linear presentable categories that admits a colimit preserving right adjoint (for example \(\mathcal{C}\) and \(\mathcal{C}'\) are compactly generated and \(f\) preserves compact objects) such that \(f_y\from\mathcal{C}_y\to\mathcal{C}'_y\) is an equivalence for all geometric points \(y\) of \(Y\).
    Then \(f\) is an equivalence.
\end{lem}
\begin{proof}
    By rigidity of \(\qcoh(Y)\) the right adjoint is also \(\qcoh(Y)\)-linear.
    We now need to check that the unit and counit are equivalences.
    This can be done on fibers, so we are done.
\end{proof}
For later purposes and for independent interest, lets us dicuss the relationship between the fiber in the categorical sense and the fiber in geometry.
\begin{lem}\label{lem: categorical fiber is fiber}
    Let \(f\from\locsys_{\widehat{G}}\to\spec(\bbZ_\ell[\sqrt{q}])\) be the structure map.
    Consider a \(k\)-point \(y\from \spec(k)\to \spec(\bbZ_\ell[\sqrt{q}])\).
    Then we have an equivalence of categories
    \begin{equation*}
        \ind\perf(\locsys_{\widehat{G}})_y\simeq\ind\perf(\locsys_{\widehat{G}}\times_{\spec(\bbZ_{\ell}[\sqrt{q}])}\spec(k)).
    \end{equation*}
\end{lem}
\begin{proof}
    Let us first start with the case where \(*=\spec(\bbQ_{\ell}[\sqrt{q}])\) or \(*=\spec(\bbF_{\ell}[\sqrt{q}])\) and \(y\from *\to\spec(\bbZ_{\ell}[\sqrt{q}])\).
    Consider the following cartesian square:
    \begin{equation*}
        \begin{tikzcd}
            \locsys_{\widehat{G}}\times_{\spec(\bbZ_\ell[\sqrt{q}])} * \arrow[d, "\bar{y}"] \arrow[r, "\bar{f}"] & * \arrow[d, "y"] \\
            \locsys_{\widehat{G}} \arrow[r, "f"]                                      & \spec(\bbZ_{\ell}[\sqrt{q}]).               
            \end{tikzcd}
    \end{equation*}
    Since \(\bar{y}_*\from\ind\perf(\locsys_{\widehat{G}}\times_{\spec(\bbZ_\ell[\sqrt{q}])} *)\to\ind\perf(\locsys_{\widehat{G}})\) admits a right adjoint (since \(\bar{y}^*\) preserves compact objects), we can apply the Barr-Beck theorem.
    Therefore we may identify \(\Mod_{\bar{y}_*\oo}(\ind\perf(\locsys_{\widehat{G}}))\) as the full subcategory of \(\ind\perf(\locsys_{\widehat{G}}\times_{\spec(\bbZ_\ell[\sqrt{q}])} *)\) generated by \(\bar{y}^*\) under colimits.
    On the other hand \(\ind\perf(*)=\Mod_{y_*\oo}(\ind\perf(\spec(\bbZ_{\ell}[\sqrt{q}])))\).
    Thus, using \cite[Theorem 4.8.4.6.]{ha} we compute
    \begin{equation*}
        \ind\perf(\locsys_{\widehat{G}})\otimes_{\ind\perf(\spec(\bbZ_{\ell}[\sqrt{q}]))}\ind\perf(*)\simeq\Mod_{f^*y_*\oo}(\ind\perf(\locsys_{\widehat{G}})).
    \end{equation*}
    If \(*=\spec(\bbF_\ell[\sqrt{q}])\), then \(y_*\oo\in\perf(\spec(\bbZ_{\ell}[\sqrt{q}]))\), therefore we may compute \(f^*y_*\oo\) using quasi-coherent sheaves instead of using ind-perfect complexes.
    In this case, we have base-change, so we conclude that \(f^*y_*\oo=\bar{y}_*\oo\).
    If \(*=\spec(\bbQ_{\ell}[\sqrt{q}])\), then \(y_*\oo=\oo[\ell^{-1}]\), so that \(f^*y_*\oo=\oo_{\locsys_{\widehat{G}}}[\ell^{-1}]\).
    We need to check \(\oo_{\locsys_{\widehat{G}}}[\ell^{-1}]\cong \bar{y}_*\oo\), for this it suffices to check that for any perfect complex \(V\in\perf(\locsys_{\widehat{G}})\) we have \(R\Gamma(\bar{y}^*V)\cong R\Gamma(V)[\ell^{-1}]\).
    By \cite[Theorem VIII.5.1.]{geometrization}, we can assume that \(V\) is pulled back from \(\perf(\spec(\bbZ_{\ell}[\sqrt{q}]))/\widehat{G}\), and even that it comes from a \(\widehat{G}\)-representation on a projective \(\bbZ_{\ell}[\sqrt{q}]\)-module.
    Both sides are computed via \(\widehat{G}\)-cohomology, and by semisimplicity in characteristic 0 both sides are concentrated in degree 0 and isomorphic, concluding the computation.

    We are left with checking that \(\ind\perf(\locsys_{\widehat{G}}\times_{\spec(\bbZ_\ell[\sqrt{q}])} *)\) generated by \(\bar{y}^*\) under colimits.
    If \(*=\spec(\bbQ_{\ell}[\sqrt{q}])\), then both sides are generated by retracts from pull backs of objects in \(\perf(\spec(\bbZ_{\ell}[\sqrt{q}]))\), since any \(\widehat{G}\)-representation on a finite dimensional \(\bbQ_{\ell}[\sqrt{q}]\)-module admits an integral lattice.
    The situation is similar when \(*=\spec(\bbF_{\ell}[\sqrt{q}])\).
    In this case not every irreducible representation lifts to \(\bbZ_{\ell}[\sqrt{q}]\), but by highest weight theory it is sufficient to only consider the Weyl representations and they lift.

    The for the case of a general field extension of \(\bbQ_{\ell}[\sqrt{q}]\) or \(\bbF_\ell[\sqrt{q}]\) observe the proof of \cite[Theorem VIII.5.1.]{geometrization} works for any \(\bbZ_{\ell}[\sqrt{q}]\) field, so we can repeat the argument above.
\end{proof}
In general, it is difficult to understand the fibers \(\mathcal{C}_y\).
It would be nice to have a full subcategory instead:
\begin{defn}\label{def: formal completion}
    Let \(Z\subset Y\) be a closed subscheme of a eventually coconnective noetherian derived scheme.
    We write \(\mathcal{C}^\wedge_{Z}\defined\qcoh(Y^\wedge_{Z})\otimes_{\qcoh(Y)}\mathcal{C}\) and call it the \defword{formal completion} of \(\mathcal{C}\) at \(Z\).
    If \(Z=\{y\}\) is a point, we will also refer to it as the \defword{formal fiber}.
\end{defn}
\begin{lem}\label{lem: reduction to formal fibers}
    Assume that we have a morphism \(Y\to S\) between eventually coconnective noetherian derived scheme.
    (If \(\mathcal{C}\) is a category of sheaves, one should think of \(\qcoh(S)\) as the category of its coefficients).
    Then in \cref{lem: reduction to individual fibers} we can equivalently check that for all geometric points \(s\from\spec(k)\to S\), and all \(k\)-valued points \(y\from\spec(k)\to Y\times_S\spec(k)\) we have that \((f_s)^\wedge_y\from(\mathcal{C}_s)^\wedge_y\to(\mathcal{C}'_s)^\wedge_y\) is an equivalence of categories.
\end{lem}
\begin{proof}
    In this case, any \(\qcoh(Y)\)-linear presentable category \(\mathcal{C}\) also becomes \(\qcoh(S)\)-linear.
    Let \(y\from\spec(k)\to Y\) be a geometric point of \(Y\), and \(s\from\spec(k)\to Y\to S\) the induced point on \(S\).
    In this case, 
    \begin{align*}
        \mathcal{C}_y&=\Mod_k\otimes_{\qcoh(Y)}\mathcal{C}\\
        &=\Mod_k\otimes_{\qcoh(Y\times_S\spec(k))}\Mod_k\otimes_{\qcoh(S)}\mathcal{C}\\
        &=\Mod_k\otimes_{\qcoh(Y\times_S\spec(k))}\mathcal{C}_{s}\\
        &=\Mod_k\otimes_{\qcoh((Y\times_S\spec(k))^\wedge_y)}\qcoh((Y\times_S\spec(k))^\wedge_y)\otimes_{\qcoh(Y\times_S\spec(k))}\mathcal{C}_{s}\\
        &=\Mod_{k}\otimes_{\qcoh((Y\times_S\spec(k))^\wedge_y)}(\mathcal{C}_s)^\wedge_y,
    \end{align*}
    by observing that \(\qcoh(Y\times_S\spec(k))=\qcoh(Y)\otimes_{\qcoh(S)}\Mod_k\).
    It follows that if all \((f_s)^\wedge_y\) are equivalences of categories, so are all \(f_y\) for all geometric points of \(Y\), so we reduce to \cref{lem: reduction to individual fibers}.
\end{proof}
Consider a \(k\)-valued semisimple parameter where \(k\) is a \(\bbZ_\ell[\sqrt{q}]\)-field.
We can now define \(\D(\bun_G,k)^\wedge_\varphi\) via \cref{def: formal completion}, using the action of \(\qcoh(k\times_{\bbZ_{\ell}[\sqrt{q}]}\locsys_{\widehat{G}}^\coarse)\) on \(\D(\bun_G,k)\).
\begin{lem}\label{lem: reduction to individual parameter}
    Let \(\Lambda\) be a noetherian \(\bbZ_\ell\)-algebra.
    Let \(V\subset\locsys_{\widehat{G}}^\coarse\) be an open subscheme.
    Assume there is a \(\perf(V)\)-linear functor 
    \begin{equation*}
        \bbL_{G}^V\from\coh^\qc_{\nilp}(V)\to\D^V(\bun_G,\Lambda)^\omega
    \end{equation*}
    such that for any algebraically closed field \(k\) over \(\Lambda\) and any semisimple \(L\)-parameter \(\varphi\from W_E\to\widehat{G}(k)\)  the induced functor 
    \begin{equation*}
        \bbL_{G}^\varphi\from\coh^\qc_{\nilp}((\locsys_{\widehat{G},k})^\wedge_{\varphi})\to(\D(\bun_G,k)^\wedge_{\varphi})^\omega
    \end{equation*}
    is an equivalence, then \(\bbL_G^V\) is an equivalence.
\end{lem}
\begin{proof}
    This follows from \cref{lem: reduction to formal fibers} and \cref{lem: categorical fiber is fiber}.
    We take \(S=\spec(\bbZ_\ell[\sqrt{q}])\) and \(Y\) to be a connected component of \(\locsys_{\widehat{G}}^\coarse\) and then varying over all connected components.
    We implicitly use that \(\D(\bun_G,\Lambda)\otimes_{\Mod_\Lambda}\Mod_k=\D(\bun_G,k)\).
    By Barr-Beck it is easy to see that there is a fully faithful inclusion \(\D(\bun_G,\Lambda)\otimes_{\Mod_\Lambda}\Mod_k\injto\D(\bun_G,k)\), and to check it is essentially surjective it suffices to check that the compact generators \(i_{b\natural}\cind_K^{G_b(E)}k\) running over \(b\in|\bun_G|\) and compact open subgroups \(K\subset G_b(E)\) lie in the essential image.
    This is true, since
    \begin{equation*}
        \begin{tikzcd}
            {\D(G_b(E),\Lambda)\otimes_{\Mod_\Lambda}\Mod_k} \arrow[rr, "i_{b\natural}\otimes_{\Mod_\Lambda}\Mod_k"] \arrow[d] &  & {\D(\bun_G,\Lambda)\otimes_{\Mod_\Lambda}\Mod_k} \arrow[d] \\
            {\D(G_b(E),k)} \arrow[rr, "i_{b\natural}"]                                                                   &  & {\D(\bun_G,k)}                                      
            \end{tikzcd}
    \end{equation*}
    commutes, where the vertical arrows are the natural arrows appearing in the Barr-Beck theorem.
    We conclude by observing that \(\D(G_b(E),\Lambda)\otimes_{\Mod_\Lambda}\Mod_k\simeq\D(G_b(E),k)\) since for example on Hecke algebras we have \(\mathcal{H}(G_b(E),\Lambda)\otimes_{\Lambda}k\cong\mathcal{H}(G_b(E),k)\).
\end{proof}
\begin{lem}
    The functor \(\mathcal{C}^\wedge_Z\to\mathcal{C}\) is fully faithful.
\end{lem}
\begin{proof}
    Since \(\qcoh(Y^\wedge_{Z})\to\qcoh(Y)\) is fully faithful with colimit preserving right adjoint the claim follows.
\end{proof}
\begin{lem}\label{lem: formal completion stack}
    Let \(\mathfrak{Y}\) be a derived stack over a qcqs derived scheme \(Y\), let \(f\from\mathfrak{Y}\to Y\) be the structure map.
    Let \(Z\subset Y\) be a cocompact closed subscheme.
    Then the subcategory \(\ind(\perf(\mathfrak{Y}))_Z^\wedge\subset\ind(\perf(\mathfrak{Y}))\) spanned by those ind-perfect sheaves that are set-theoretically supported on \(f^{-1}(Z)\).
    In particular \(\ind(\perf(\mathfrak{Y}))_Z^\wedge=\ind(\perf(\mathfrak{Y}^\wedge_{f^{-1}(Z)}))\).
\end{lem}
\begin{proof}
    We may localise around \(Z\) to an open affine using \cref{lem: fppf descent sheaves of categories}, and thus we can assume that \(Y\) is affine.
    We have a pair of colimit preserving adjoint functors 
    \begin{equation*}
        j_!\from\qcoh(Y^\wedge_y)\rightleftarrows\qcoh(Y):j^*
    \end{equation*}
    and we have that \(j_!j^*\cong -\otimes j_!\oo\).
    It follows that for an arbitrary \(\qcoh(Y)\)-linear presentable stable category \(\mathcal{C}\) that \(\mathcal{C}^\wedge_y\) is the full subcategory spanned by the objects such that \(M\cong M\otimes j_!\oo\).
    Writing \(Y=\spec(R)\) and \(Z=V(f_1,\dots,f_r)\) we can see that \(j_!\oo=\colim_n R/^L(f_1^n,\dots,f_r^n)\).
    Taking \(\mathcal{C}=\ind(\perf(\mathfrak{Y}))\) we thus obtain that each object in \(\ind(\perf(\mathfrak{Y}))^\wedge_Z\) is set-theoretically supported on \(f^{-1}(Z)\) and conversely perfect complexes set-theoretically supported on \(f^{-1}(Z)\) lie in this category and generate it.
\end{proof}
\begin{cor}
    Let \(k\) be a \(\bbZ_\ell[\sqrt{q}]\)-field and let \(\varphi\) be a semisimple \(k\)-valued parameter.
    We identifiy this with a closed point of \(k\times_{\bbZ_\ell[\sqrt{q}]}\locsys_{\widehat{G}}^\coarse\) and we write \(p\from\locsys_{\widehat{G}}\to\locsys_{\widehat{G}}^\coarse\) for the canonical map.
    The formal fiber \(\D(\bun_G,k)^\wedge_\varphi\) admits a canonical action of perfect complexes on the formal completion \(\perf((k\times_{\bbZ_\ell[\sqrt{q}]}\locsys_{\widehat{G}})^{\wedge}_{p^{-1}(\varphi)})\).
\end{cor}
In the affine case, we obtain a very explicit description, that also works outside the noetherian eventually coconnective case.
\begin{lem}\label{lem: explicit description formal completion affine case}
    Let \(\mathcal{C}\) be an \(\perf(R)\)-linear small stable idempotent complete category.
    Let \(Z\subset\spec(R)\) be a cocompact closed subset.
    Then \(\mathcal{C}\otimes_{\perf(\spec(R))}\perf(\spec(R)^\wedge_Z)\) is the full subcategory of \(\mathcal{C}\) spanned by those objects \(c\in\mathcal{C}\) such that \(R\to\End(c)\) makes \(\End(c)\) a quasi-coherent sheaf on \(\spec(R)\) set-theoretically supported on \(Z\).
\end{lem}
\begin{proof}
    Let \(X\defined\spec(R)\).
    From the Verdier sequence 
    \begin{equation*}
        \perf(X^\wedge_Z)\to\perf(X)\to\perf(X\setminus Z),
    \end{equation*}
    we obtain a Verdier sequence 
    \begin{equation*}
        \perf(X^\wedge_Z)\otimes_{\perf(X)}\mathcal{C}\to\perf(X)\otimes_{\perf(X)}\mathcal{C}\to\perf(X\setminus Z)\otimes_{\perf(X)}\mathcal{C}.
    \end{equation*}
    Choosing some ideal \(I=(f_1,\dots,f_r)\subset\pi_0R\) such that \(Z=V(I)\), we may describe \(\perf(X^\wedge_Z)\) as the full subcategory of \(\perf(X)\) where for all \(f_i\) there exists some power \(f_i^n\) that acts by 0 (depending on the object).
    Thus the same description holds for 
    \[\perf(X^\wedge_Z)\otimes_{\perf(X)}\mathcal{C}\to\perf(X)\otimes_{\perf(X)}\mathcal{C}=\mathcal{C}.\]
    This is precisely the condition that \(\End(c)\) becomes a quasi-coherent sheaf set-theoretically supported on \(Z\).
\end{proof}
\begin{rem}
    Assume we have a compactly generated \(\Mod_R\)-linear stable presentable category \(\mathcal{C}\), such that the induced action of \(\perf(R)\) preserves the subcategory \(\mathcal{C}^\omega\) of compact objects.
    Then \(\mathcal{C}\otimes_{\qcoh(\spec(R))}\qcoh(\spec(R)^\wedge_Z)\) is compactly generated by 
    \[\mathcal{C}^\omega\otimes_{\perf(\spec(R))}\perf(\spec(R)^\wedge_Z),\]
    thus by \cref{lem: explicit description formal completion affine case} we can get an explicit description of the compact objects.
\end{rem}
\begin{lem}\label{lem: compact objects in localization compact and ULA}
    Let \(\varphi\) be a semisimple parameter valued in a \(\bbZ_{\ell}[\sqrt{q}]\)-field \(k\).
    The category \(\D(\bun_G,k)^\wedge_{\varphi}\) is compactly generated.
    The compact objects are compact and ULA in \(\D(\bun_G,k)\).
\end{lem}
\begin{proof}
    For formal reasons \(\D(\bun_G,k)^\wedge_{\varphi}\) is compactly generated, and the compact objects are compact in \(\D(\bun_G,k)\).
    These compact objects are compact in \(\D(\bun_G)\) for formal reasons. 
    Let \(\pi\in\D(\bun_G,k)^\wedge_{\varphi}\) be compact and we claim that it is ULA.
    For this let \(\pi\in\D(G_b(E),k)^\wedge_{\varphi}\) be a compact object.
    We know that then \(H^i(\pi)\neq 0\) only for finitely many \(i\) and each \(H^i(\pi^K)\) is finitely generated over some Hecke algebra \(\mathcal{H}(G_b(E),K)\).
    We know by \cite[Corollary 3.5.]{finiteness-p-adic} and \cite[Theorem 5.2.1.]{categorical-fargues-tori} that the \(H^i(\pi^K)\) are finitely generated over \(\speccenter(G_b)\).
    By \cite[Corollary 2.4.]{finiteness-p-adic} and \cite[Theorem IX.7.2.]{geometrization} it is also finite over \(\speccenter(G)\).
    By definition, considering \(H^i(\pi^K)\) as a sheaf on \(\spec(\speccenter(G))\) is set-theoretically supported \(\varphi\), together with the finiteness we see that as a \(k\)-module finitely generated.
    We conclude that \(\pi^K\) is a perfect complex of \(k\)-modules
    and thus \(\pi\) is ULA.
\end{proof}
\begin{cor}\label{cor: localizing D(Bun_G) at a parameter}
    Let \(\varphi\) be a semisimple parameter valued in an algebraically closed \(\bbZ_{\ell}[\sqrt{q}]\)-field \(k\).
    The category \(\D(\bun_G,k)^\wedge_{\varphi}\) is the full subcategory of \(\D(\bun_G,k)\) spanned by those objects such that any irreducible subquotient of the cohomology sheaves has \(L\)-parameter given by \(\varphi\).
\end{cor}
\begin{proof}
    Since \(\D(\bun_G,k)^\wedge_{\varphi}\) is compactly generated, so it suffices to check this claim on compact objects.
    By \cref{lem: compact objects in localization compact and ULA} we may find a filtration such that each quotient lives on a single stratum \(\bun_G^b\), is concentrated in a single cohomological degree and is irreducible as a \(G_b(E)\)-representation.
    For such an irreducible \(G_b(E)\)-representation \(\pi\), the condition that \(\speccenter(G)\to\End(\pi)\) is set-theoretically supported on \(\varphi\) is equivalent to \(\pi\) having \(L\)-parameter given by \(\varphi\).
    This implies that the conditions are necessary.
    They are also sufficient, as \((\D(\bun_G,k)^\wedge_{\varphi})^\omega\) is closed under extensions, so we may reduce complexes concentrated in a single degree.
\end{proof}
We can now formulate a general criterion for when \(\mathbb{L}_G\) is an equivalence.
\begin{lem}\label{lem: general equivalence criterion}
    Let \(\Lambda\) be a \(\bbZ_{\ell}[\sqrt{q},\mu_{p^\infty}]\)-algebra.
    Let \(\psi\from U(E)\to \Lambda^\times\) be a non-degenerate character, and \(\mathcal{W}=i_{1!}\cind_{U(E)}^{G(E)}\psi\) the Gelfrand-Graev representation constructed from the Whittaker datum \((U,\psi)\).
    Let \(p\from\locsys_{\widehat{G}}\to\locsys_{\widehat{G}}^\coarse\) be the canonical map.
    Let \(V\subset\locsys_{\widehat{G}}^\coarse\) be a be an open subscheme, such that \(\nilp|_{p^{-1}(V)}=\{0\}\) and for any algebraically closed field \(k\) over \(\Lambda\) and any semisimple \(L\)-parameter \(\varphi\in V(k)\) the stabilizer \(S_{\varphi}\) is a torus.
    Assume for any algebraically closed field \(k\) over \(\Lambda\) and any semisimple \(L\)-parameter \(\varphi\from W_E\to\widehat{G}(k)\) giving rise to a point in \(V\), we have a direct sum decomposition 
    \begin{equation*}
        (\D(\bun_G,k)^\wedge_{\varphi})^\omega\simeq\bigoplus_{\chi\in X^*(S_{\varphi})}(\D(\bun_G,k)^\wedge_{\varphi})^\omega_\chi
    \end{equation*}
    such that \((\D(\bun_G,k)^\wedge_{\varphi})^\omega_0\subset\D(\bun_G^1,k)^\wedge_{\varphi}\).
    Let us write \(V_k\defined k\times_{\Lambda}V\).
    Assume for each such \(\varphi\) we have 
    \begin{equation*}
        p^{-1}(V_k)\times_{V_k}(V_k)^{\wedge}_{\varphi}\cong N^\wedge_{0}/S_{\varphi},
    \end{equation*}
    where \(N\) is an affine scheme with \(0\in N\) and it admits a trivial \(S_{\varphi}\)-action.
    Fixing such an isomorphism we can construct line bundles \(\oo(\chi,\varphi)\) by pulling back \(\chi\in\perf(BS_{\varphi})\) via \(N^\wedge_{0}/S_{\varphi}\to BS_{\varphi}\).
    Assume that \(\oo(\chi,\varphi)*-\) restricts to an equivalence 
    \[\oo(\chi,\varphi)*-\from(\D(\bun_G,k)^\wedge_{\varphi})^\omega_0\to(\D(\bun_G,k)^\wedge_{\varphi})^\omega_\chi.\]
    Then \(\bbL_{G}^\varphi\) is an equivalence if and only if every Schur-irreducible \(\pi\in\D(G(E))\) with \(L\)-parameter \(\varphi\) is generic.
    In this case, we obtain a \(\Ind\perf(p^{-1}(V))\)-linear equivalence 
    \begin{equation*}
        \Ind\coh_{\nilp}^\qc(p^{-1}(V))\simeq\D^{V}(\bun_G,\Lambda)
    \end{equation*}
    sending \(\oo\) to \(\mathcal{W}^V\).
\end{lem}
\begin{proof}
    By \cite[Proposition 5.2.]{locallanglandsinfamilies2} the \(\Lambda[G(E)]\)-module $\mathcal{W}$ descends to a noetherian $\bbZ_\ell[\sqrt{q}]$-algebra (by case changing the ring \(\oo_K[1/p]\) they consider there along \(\bbZ[1/p]\to\bbZ_\ell[\sqrt{q}]\)) and we can replace $\Lambda$ with this ring and \(\mathcal{W}\) with the representation considered there base changed to \(\Lambda\).
    We can therefore assume that all occurring stacks are noetherian.
    It suffices to show the claim for small enough \(V\), so we may assume that \(V\) is quasi-compact.
    Since \(\nilp|_{p^{-1}(V)}=\{0\}\), it suffices to check that 
    \begin{equation}\label{eq: equivalence to show}
        \perf(p^{-1}((k\times_{\Lambda}V)^\wedge_{\varphi}))_0\simeq(\D(\bun_G,k)^\wedge_{\varphi})^\omega_0
    \end{equation}
    is an equivalence, as in the proof of \cite[Lemma 6.2.5.]{categorical-fargues-tori}.
    Note that \(N^\wedge_0\cong (V_k)^{\wedge}_{\varphi}\) and the isomorphism fits into the following commutative diagram 
    \begin{equation*}
        \begin{tikzcd}
            p^{-1}((V_k)^{\wedge}_\varphi) \arrow[r] \arrow[d] & N^\wedge_0/S_\varphi \arrow[d] \\
            (V_k)^{\wedge}_\varphi \arrow[r]                              & N^\wedge_0  .                  
            \end{tikzcd}
    \end{equation*}
    In particular \(\perf(p^{-1}((V_k)^\wedge_\varphi))_0\simeq\perf((V_k)^{\wedge}_{\varphi})\).
    Since we assumed \(V\) to be quasicompact we can find a pro-\(p\) open compact subgroup \(P^e\subset W_E\) such that \(V\subset\locsys_{\widehat{G},P^e}^\coarse\).
    Since \(\mathcal{W}^{P^e}_k\defined\mathcal{W}^{P^e}\otimes_{\Lambda}k\) is compact projective by \cref{lem: whittaker compact projective}, we have a fully faithful functor \(\perf_{\End(\mathcal{W}^{P^e}_k)}\injto\D^{P^e}(\bun_G^1,k)^\omega\) which is the left adjoint of \(\Hom(\mathcal{W}^{P^e}_k,-)\).
    Note that \(\locsys_{\widehat{G},P^e,k}^\coarse\defined k\times_{\Lambda}\locsys_{\widehat{G},P^e}^\coarse\) is an affine scheme, so \(\perf(\locsys_{\widehat{G},P^e,k}^\coarse)\) is generated by the structure sheaf under finite colimits and retracts.
    It follows that the composite 
    \begin{equation*}
        \perf(\locsys_{\widehat{G},P^e,k}^\coarse)\xto{\pi^*}\perf(\locsys_{\widehat{G},P^e,k})\to\D^{P^e}(\bun_G^1,k)^\omega
    \end{equation*}
    lands in the full subcategory \(\perf_{\End(\mathcal{W}^{P^e})}\), thus we obtained the dashed arrow in the following commutative diagram:
    \begin{equation*}
        \begin{tikzcd}
            {\perf(\locsys_{\widehat{G},P^e,k})} \arrow[r]                                    & {\D^{P^e}(\bun_G,k)^\omega}            \\
            {\perf(\locsys_{\widehat{G},P^e,k}^\coarse)} \arrow[u, "\pi^*"] \arrow[r, dotted] & \perf_{\End(\mathcal{W}^{P^e}_k)} \arrow[u, hook].
            \end{tikzcd}
    \end{equation*}
    In fact, the dashed arrow arises from restricting the action of \(\perf(\locsys_{\widehat{G},P^e,k})\) on \(\D^{P^e}(\bun_G^1,k)^\omega\) via the symmetric monoidal functor \(\pi^*\) and observing that the full subcategory \(\perf_{\End(\mathcal{W}^{P^e})}\) gets preserved, then acting on \(\End(\mathcal{W}^{P^e})\).
    Thus by localizing local Langlands in families in the form of \cref{lem: generic local langlands in families} along open immersion \(\locsys_{\widehat{G},P^e,k}^\coarse\subset\locsys_{\widehat{G},k}^\coarse\)  and \cite[Corollary 6.1.2.]{categorical-fargues-tori} the dashed arrow is an equivalence.
    Everything in the diagram carries an action of \(\perf(\locsys_{\widehat{G},P^e,k}^\coarse)\) and the diagram enhances to a square of \(\perf(\locsys_{\widehat{G},P^e,k}^\coarse)\)-linear categories, so we can apply \(\perf((V_k)^{\wedge}_{\varphi})\otimes_{\perf(\locsys_{\widehat{G},P^e,k}^\coarse)}-\) to obtain 
    \begin{equation*}
        \begin{tikzcd}
            {\perf(p^{-1}((V_k)^{\wedge}_{\varphi}))} \arrow[r]                                      & {(\D(\bun_G,k)^\wedge_{\varphi})^\omega_0}                \\
            {\perf((V_k)^{\wedge}_{\varphi})} \arrow[u, "\pi^*"] \arrow[r, dotted] & (\perf_{\End(\mathcal{W}^{P^e}_k)})^\wedge_{\varphi} \arrow[u, hook] .
            \end{tikzcd}
    \end{equation*}
    The dashed arrow remains an equivalence, being the base change of an equivalence.

    We are reduced to show that 
    \begin{equation*}
        (\D(\bun_G^1,k)^\wedge_\varphi)^\omega\simeq(\perf_{\End(\mathcal{W}^{P^e}_k)})^\wedge_{\varphi}.
    \end{equation*}
    It would then follow that the equivalence \(\perf((V_k)^\wedge_\varphi)\simeq(\perf_{\End(\mathcal{W}^{P^e}_k)})^\wedge_{\varphi}\simeq(\D(\bun_G,k)^\wedge_\varphi)^\omega_0\) factors over the essential image of \(\pi^*\) which is \(\perf(p^{-1}((V_k)^\wedge_\varphi))\), giving rise to the equivalence required for \cref{eq: equivalence to show}.

    For this it suffices to show that the right adjoint to \((\perf_{\End(\mathcal{W}^{P^e}_k)})^\wedge_{\varphi}\injto(\D(\bun_G,k)^\wedge_\varphi)^\omega_0\) exits and is conservative.
    The right adjoint is given by \(\Hom(\mathcal{W}^{P^e}_k,-)\).
    Let \(0\neq A\in(\D(\bun_G,k)^\wedge_{\varphi})^\omega_0\).
    We want to show that \(\Hom(\mathcal{W}^{P^e}_k,A)\) does not vanish.
    The top cohomology \(H^i(A)\) has a subquotient that is an irreducible representation \(\pi\), whose \(L\)-parameter is \(\varphi\) and by \(t\)-exactness of \(\Hom(\mathcal{W}^{P^e}_k,-)\) it suffices to show that \(\Hom(\mathcal{W}^{P^e}_k,\pi)\neq 0\), which holds true by assumption, as \(\Hom(\mathcal{W}^{P^e}_k,\pi)=\Hom(\mathcal{W}\otimes_{\Lambda}k,\pi)\), since the \(L\)-parameter of \(\pi\) is trivial on \(P^e\).
    To get the final conclusion, note that \(\coh_{\nilp}^\qc(V)=\perf^\qc(V)\), so the spectral action gives us a \(\perf(V)\)-linear equivalence 
    \begin{equation*}
        \coh_{\nilp}^\qc(V)\simeq\D(\bun_G,\Lambda)^\omega\otimes_{\perf(\locsys_{\widehat{G}})}\perf(V)
    \end{equation*}
    sending \(\oo\) to \(\mathcal{W}^V\).
    It is an equivalence by the previous discussion and \cref{lem: reduction to individual parameter}.
\end{proof}
\begin{lem}\label{lem: whittaker compact projective}
    Consider the representation \(\mathcal{W}\) as constructed in \cite[Proposition 5.2.]{locallanglandsinfamilies2} base changed to \(\Lambda=\oo_K[1/p]\otimes_{\bbZ[1/p]}\bbZ_\ell[\sqrt{q}]\) as discussed in the beginning of the proof of the previous theorem.   
    For any choice of open pro-\(p\) subgroup \(P^e\) of the wild inertia the representation \(\mathcal{W}^{P^e}\) is compact projective.
\end{lem}
\begin{proof}
    The choice of such a \(P^e\) gives rise to a direct summand \(\speccenter(G)_{\leq e}\) of \(\speccenter(G)\).
    Observe that \(\speccenter(G)_{\leq e}\) only has finitely many minimal idempotents and the map \(\speccenter(G)\to\berncenter(G)\) is defined over \(\bbZ_\ell\) induced by the spectral action\footnote{A priori this is only defined over \(\bbZ_\ell[\sqrt{q}]\), however we can remove this by \cref{rem: remove sqrt q}.} is an isomorphism by \cite[Corollary 7.7]{local-langlands-in-families}(while the map in \cite[Conjecture 7.2]{local-langlands-in-families} goes in the other direction, it agrees with the inverse of the map induced by the spectral action by uniqueness of such a map that is compatible with local Langlands, as discussed there).\footnote{The isomorphism there is proven over \(W(\overline{\bbF_\ell})\), however the base change of the map \(\speccenter(G)\to\berncenter(G)\) along \(\bbZ_\ell\to W(\overline{\bbF_\ell})\) is the inverse of their map and one can then deduce the result for \(\bbZ_\ell\) by flat descent.} 
    It follows that \(\mathcal{W}^{P^e}\) lives in only finitely many blocks.
    In particular, there exists some \(r\) such that \(\mathcal{W}^{P^e}\) lies only in blocks corresponding to depth \(\leq r\) so that it is a direct summand of \(\mathcal{W}^{K_r}\), which is compact projective by \cite[Corollary 5.12.]{locallanglandsinfamilies2}.
\end{proof}
\begin{lem}\label{lem: generic local langlands in families}
    The spectral action induces an isomorphism \(\speccenter(G)\cong\End(\mathcal{W})\) with \(\mathcal{W}\) as in the previous lemma.
\end{lem}
\begin{proof}
    By \cref{lem: whittaker compact projective} we have \(\End(\mathcal{W})\) is concentrated in degree 0, so it suffices to check an isomorphism on classical rings.
    By \cite[Corollary 7.7]{local-langlands-in-families} we have that the spectral action induces an isomorphism \(\speccenter(G)\to\berncenter(G)\) over \(\bbZ_\ell\) as discussed in the proof of the previous lemma.
    \cite[Theorem 5.2.]{integral-whittaker-model} tells us that \(\berncenter(G)\to\End(\mathcal{W})\) is an isomorphism\footnote{Again, this isomorphism is proven over \(W(\overline{\bbF_\ell})\), so we have the isomorphism in the proof over \(\Lambda\otimes_{\bbZ_\ell[\sqrt{q}]}W(\overline{\bbF_\ell})[\sqrt{q}]\) via base change and deduce that the isomorphism exists over \(\Lambda\) via descent.}, so in summary the composite \(\speccenter(G)\to\End(\mathcal{W})\) is an isomorphism, and this composite is the map induced by the spectral action.
\end{proof}

%% file: irreducible.tex
\section{Irreducible case of the conjecture}
Let \(k\) be an algebraically closed \(\bbZ_\ell\)-field.

\begin{lem}\label{lem: decomposition D(Bun_G) at irreducible parameters}
    Let \(\varphi\) be a irreducible parameter.
    We have a direct sum decomposition
    \begin{equation*}
        (\D(\bun_G,k)^\wedge_{\varphi})^\omega\simeq\bigoplus_{\chi\in X^*(S_{\varphi})}(\D(\bun_G,k)^\wedge_{\varphi})^\omega_\chi.
    \end{equation*}
    Here \((\D(\bun_G,k)^\wedge_{\varphi})^\omega_\chi=(\D(\bun_G^{b_{\chi,G}},k)^\wedge_{\varphi})^\omega\) (embedded via \(i_{b!}\)).
\end{lem}
\begin{proof}
    This follows from the fact that for irreducible \(\varphi\) we have \((\D(\bun_G,k)^\wedge_{\varphi})^\omega\subset\D(\bun_G^{\mathrm{ss}},k)\) as argued in \cite[Section X.2.]{geometrization}.
    Then the direct sum decomposition follows from the fact that we have an isomorphism \(X^*(S_{\varphi})\cong|\bun_G^{\mathrm{ss}}|\) sending \(\chi\mapsto b_{\chi,G}\) and that \(|\bun_G^{\mathrm{ss}}|\) is a discrete set of points.
\end{proof}
\begin{lem}\label{lem: hecke stalk irreducible case}
    The functor $\oo(\chi,\varphi)*-$ sends $\D(\bun_G,k)^\wedge_{\varphi})^\omega_0$ to $\D(\bun_G,k)^\wedge_{\varphi})^\omega_\chi$ and vanishes elsewhere.
\end{lem}
\begin{proof}
    This is immediate from \cite[Proposition 5.3.3.]{categorical-fargues-tori}.
\end{proof}
\begin{thm}\label{thm: categorical equivalence irreducible}
    Let \(\Lambda\) be a \(\bbZ_\ell[\sqrt{q},\mu_{p^\infty}]\)-algebra.
    Let \(\locsys_{\widehat{G}}^\irred\) denote the open substack consisting of parameters that are irreducible.
    We have a \(\Ind\perf(\locsys_{\widehat{G}}^\irred)\)-linear equivalence 
    \begin{equation*}
        \Ind\coh_{\nilp}^\qc(\locsys_{\widehat{G}}^\irred)\simeq \D(\bun_G,\Lambda)\otimes_{\Ind\perf(\locsys_{\widehat{G}})}\Ind\perf(\locsys_{\widehat{G}}^\irred)
    \end{equation*}
    sending \(\oo\) to \(\mathcal{W}^{\locsys_{\widehat{G}}^\irred}\) preserving compact objects.
\end{thm}
\begin{proof}
    We only need to check the conditions of \cref{lem: general equivalence criterion}.
    That \(\nilp|_{\locsys_{\widehat{G}}^\irred}=\{0\}\) is \cref{lem: nilpotent singular support irreducible}.
    That \(\locsys_{\widehat{G}}^\irred\times_{\locsys_{\widehat{G}}^{\irred,\coarse}}(\locsys_{\widehat{G}}^{\coarse})^{\wedge}_{\varphi}\simeq N^\wedge_{0}/S_{\varphi}\), where \(N\) is an affine scheme with \(0\in N\) and it admits a trivial \(S_{\varphi}\)-action for any \(\varphi\in U(k)\) for \(k\) an algebraically closed \(\Lambda\)-field is \cref{cor: localize at langlands-shahidi fiber product}.
    The direct sum decomposition 
    \begin{equation*}
        (\D(\bun_G,k)^\wedge_{\varphi})^\omega\simeq\bigoplus_{\chi\in X^*(S_{\varphi})}(\D(\bun_G,k)^\wedge_{\varphi})^\omega_\chi
    \end{equation*}
    with the required properties is \cref{lem: hecke stalk irreducible case} and \cref{lem: decomposition D(Bun_G) at irreducible parameters}.
    The genericity statement is \cref{lem: Langlands-Shahidi type generic}.
\end{proof}
\begin{rem}
    The previous theorem already holds for any \(\bbZ_\ell\)-algebra, see also the discussion in \cref{rem: main theorem works over zl}.
\end{rem}
Let \(k\) be an algebraically closed field.
\begin{notn}
    Let \(\varphi\) be a semisimple \(L\)-parameter for valued in \(k\).
    Write \(\mathcal{W}_{\varphi}\) for the image of \(i_{\varphi}\oo*\mathcal{W}\), where \(i\from{\varphi}\times_{\locsys_{\widehat{G}}^\coarse}\locsys_{G}\to\locsys_{\widehat{G}}\) is the closed embedding induced by \(\varphi\).
    This canonically lives in \(\D(\bun_G,k)_{\varphi}\) and not just the formal fiber.
\end{notn}
\begin{lem}\label{lem: whittaker localized at irreducible parameter}
    If \(\varphi\) is an irreducible parameter valued in \(k\), then \(\mathcal{W}_{\varphi}=i_{1!}\pi\), where \(\pi\) is the irreducible supercuspidal representation attached to \(\varphi\) under the local Langlands correspondence, concentrated in degree 0.
\end{lem}
\begin{proof}
    Let \(P^e\) be an compact open pro-\(p\) subgroup of \(W_E\), giving rise to \(\speccenter(G)_{\leq e}\), so that \(\varphi\in\spec(\speccenter(G)_{\leq e}\otimes_{\bbZ_\ell[\sqrt{q}]}k)\).
    Let \(x_{\varphi}\from\spec(k)\to\spec(\speccenter(G)_{\leq e}\otimes_{\bbZ_\ell[\sqrt{q}]}k)\) be the point corresponding to \(\varphi\).
    Then \(\mathcal{W}_{\varphi}\) is the image of \(x_{\varphi}\oo\) under the functors 
    \[\Mod_{\speccenter(G)_{\leq e}\otimes_{\bbZ_\ell[\sqrt{q}]}k}\simeq\Mod_{\End(\mathcal{W}^{P^e}\otimes_{\bbZ_\ell[\sqrt{q}]}k)}\subset\D(G(E),k)\xto{i_{1!}}\D(\bun_G,k),\]
    the claim follows easily.
\end{proof}
\begin{lem}\label{lem: lift to fiber}
    Let \(\pi\) be a Schur-irreducible object in \(\D(\bun_G,k)\) with parameter \(\varphi\) and assume that \(\End(\pi)\) is coconnective.
    Then \(\pi\) lifts to an object \(\pi^{\lift}\in\D(\bun_G,k)_{\varphi}\).
\end{lem}
\begin{proof}
    Let \(m_{\varphi}\) be the maximal ideal of \(\speccenter(G)\) attached to \(\varphi\).
    Then we have a morphism \(\speccenter(G)\to\End(\pi)\) and a lift exists if we can find a dashed arrow making the following diagram commutative: 
    \begin{equation*}
        \begin{tikzcd}
            \speccenter(G) \arrow[rr] \arrow[rd] &                                               & \End(\pi) \\
                                                 & \speccenter(G)/m_{\varphi} \arrow[ru, dashed] &          
            \end{tikzcd}
    \end{equation*}
    Since \(\speccenter(G)\) and \(\speccenter(G)/m_{\varphi}\) are connective as ring spectra, it suffices to find a dashed arrow making the following diagram commutative:
    \begin{equation*}
        \begin{tikzcd}
            \speccenter(G) \arrow[rr] \arrow[rd] &                                               & \tau^{\leq 0}\End(\pi)=\pi_0\End(\pi)\cong k \\
                                                 & \speccenter(G)/m_{\varphi} \arrow[ru, dashed] &          
            \end{tikzcd}
    \end{equation*}
    In this case, the dashed arrow exists and is induced by \(\varphi\).
\end{proof}
\begin{rem}
    The hypothesis that \(\End(\pi)\) is coconnective holds if \(\pi\) is of the form \(i_{b!}A\) with \(A\) concentrated in a single degree or if \(\pi\) is in the heart of a(ny) \(t\)-structure.
    The first case is the only one we need.
\end{rem}
\begin{lem}\label{lem: hecke operator on whittaker}
    Let \(\varphi\) be an irreducible parameter valued in \(k\).
    Then \(\oo(\chi,\varphi)*\mathcal{W}_{\varphi}\) is concentrated in degree 0.
\end{lem}
\begin{proof}
    If \(k\) is of characteristic 0, this follows from the fact that \(\oo(\chi,\varphi)\) is a direct summand of an Hecke operator \(T_V\) and \cite[Theorem 1.8.]{t-exact} and \cref{lem: whittaker localized at irreducible parameter}.
    Now let \(k\) be of characteristic \(\ell\).
    Note that \(\oo(\chi,\varphi)*\mathcal{W}_{\varphi}\in\D(\bun_G^{b_{\chi,G}},k)_{\varphi}\) and by \cref{thm: categorical equivalence irreducible} we know that \(\D(\bun_G^{b_{\chi,G}},k)_{\varphi}=\Mod_k\).
    In particular, there is a \(\pi\in\D(\bun_G^{b_{\chi,G}},k)_{\varphi}\) such that any object is of the form \(\bigoplus_i \pi^{(J_i)}[i]\) (for some potentially infinite \(J_i\)) and up to shift this object is uniquely determined by the property that \(\pi_0\End(\pi)=k\).
    Let \(\rho\) be a non-zero irreducible subquotient of \(H^i(\oo(\chi,\varphi)*\mathcal{W}_{\varphi})\) for some \(i\).
    Note that \(\pi_0\End_{\D(\bun_G^{b_{\chi,G}})_{\varphi}^\wedge}(\oo(\chi,\varphi)*\mathcal{W}_{\varphi})=k\) again by \cref{thm: categorical equivalence irreducible}.
    We claim that \(\pi_0\End_{\D(\bun_G^{b_{\chi,G}},k)_{\varphi}}(\rho)=k\), then \(\oo(\chi,\varphi)*\mathcal{W}_{\varphi}\cong\rho[i]\) for some \(i\).
    By \cref{lem: lift to fiber} we can lift \(\rho\) to \(\rho^\lift\in\D(\bun_G^{b_{\chi,G}})_\varphi\).
    There is a canonical morphism \(\pi_0\End_{\D(\bun_G^{b_{\chi,G}},k)_\varphi}(\rho^\lift)\to\pi_0\End_{\D(\bun_G^{b_{\chi,G}},k)}(\rho)\), which we claim is an isomorphism.
    Note that \cref{thm: categorical equivalence irreducible} gives us the following compatible equivalences: 
    \begin{equation*}
        \begin{tikzcd}
            \D(\bun_G^{b_{\chi,G}},k)_\varphi \arrow[r] \arrow[d, "{\oo(-\chi,\varphi)*-}"'] & \D(\bun_G^{b_{\chi,G}},k)^\wedge_\varphi \arrow[d, "{\oo(-\chi,\varphi)*-}"] \\
            \D(\bun_G^{1},k)_\varphi \arrow[r]                                         & \D(\bun_G^{1},k)^\wedge_{\varphi} .                                      
            \end{tikzcd}
    \end{equation*}
    In particular, it suffices to check that 
    \[
        \pi_0\End_{\D(\bun_G^{1},k)_\varphi}(\oo(-\chi,\varphi)*\rho^\lift)\to\pi_0\End_{\D(\bun_G^{1},k)^\wedge_{\varphi}}(\oo(-\chi,\varphi)*\rho)
    \]
    is an isomorphism, where \(\oo(-\chi,\varphi)*\rho^\lift\) is indeed a lift of \(\oo(-\chi,\varphi)*\rho\) to \(\D(\bun_G^{1},k)_\varphi\),
    since we already know \(\pi_0\End_{\D(\bun_G^{1})^\wedge_{\varphi}}(\oo(-\chi,\varphi)*\rho)=k\). 
    Under \cref{thm: categorical equivalence irreducible} this translates to the statement that if \(\mathcal{F}\in\Mod_k\), \(y\from\spec(k)\to \locsys_{\widehat{G}}^\coarse\) is a closed point, and we have \(\pi_0\End_{\qcoh((\locsys_{\widehat{G}}^\coarse)^\wedge_{\varphi})}(y_*\mathcal{F},y_*\mathcal{F})=k\), then 
    \[\pi_0\End_{\qcoh(k)}(\mathcal{F})\to\pi_0\End_{\qcoh((\locsys_{\widehat{G}}^\coarse)^\wedge_{\varphi})}(y_*\mathcal{F},y_*\mathcal{F})\] 
    is an isomorphism.
    However, the only way we can have \(\pi_0\End_{\qcoh((\locsys_{\widehat{G}}^\coarse)^\wedge_{\varphi})}(y_*\mathcal{F},y_*\mathcal{F})=k\) if \(\mathcal{F}=k[i]\) for some \(i\), in this case we easily obtain the claimed isomorphism.
    To conclude, we know that \(\oo(\chi,\varphi)*\mathcal{W}_{\varphi}=\rho[i]\) for some \(i\).

    We can find a discrete valuation ring \(\oo\) such that \(\oo/\ell=k\), \(\oo[1/\ell]\) is a field of characteristic 0 and \(\varphi\) lifts to \(\oo\).
    We fix such a lift and call it \(\varphi'\) and call \(\varphi'[1/\ell]\) the corresponding parameter for \(\oo[1/\ell]\).
    We can lift \(\mathcal{W}\) to \(\oo\) and obtain \(\mathcal{W}_{\varphi'}\), which satisfies \(\mathcal{W}_{\varphi'}/\ell=\mathcal{W}_{\varphi}\) and \(\mathcal{W}_{\varphi'}[1/\ell]=\mathcal{W}_{\varphi'[1/\ell]}\).
    Note that \(\spec(\oo)\times_{\locsys_{\widehat{G}}^\coarse}\locsys_{\widehat{G}}\) is a \(B\bbG_m\)-gerbe over \(\oo\), so it splits.
    Observe that \(\qcoh(\spec(\oo)\times_{\locsys_{\widehat{G}}^\coarse}\locsys_{\widehat{G}})\) canonically acts on \(\Mod_{\oo}\otimes_{\qcoh(\locsys_{\widehat{G}}^\coarse)}\D(\bun_G,\oo)\), so we can define \(\oo(\chi,\varphi')*\mathcal{W}_{\varphi'}\) such that \(\oo(\chi,\varphi')*\mathcal{W}_{\varphi'}/\ell=\oo(\chi,\varphi)*\mathcal{W}_{\varphi}\) and that \(\oo(\chi,\varphi')*\mathcal{W}_{\varphi'}[1/\ell]=\oo(\chi,\varphi'[1/\ell])*\mathcal{W}_{\varphi'[1/\ell]}\).
    Since \(\oo(\chi,\varphi')*\mathcal{W}_{\varphi'}/\ell\) is concentrated in degree \(i\), so is \(\oo(\chi,\varphi')*\mathcal{W}_{\varphi'}\).
    Thus, to determine \(i\) it suffices to determine 
    \[\oo(\chi,\varphi')*\mathcal{W}_{\varphi'}[1/\ell]=\oo(\chi,\varphi'[1/\ell])*\mathcal{W}_{\varphi'[1/\ell]},\]
    which is concentrated in degree 0 as discussed before, since we are now in characteristic 0.
\end{proof}
\begin{cor}\label{cor: t-exact irreducible case}
    Let \(\varphi\) be an irreducible parameter valued in \(k\).
    Then \(\oo(\chi,\varphi)*-\) is \(t\)-exact on \(\D(\bun_G,k)^\wedge_{\varphi}\).
\end{cor}
\begin{proof}
    This follows from \cref{lem: hecke operator on whittaker}, noting that \(\oo(\chi,\varphi)*\mathcal{W}_{\varphi}\) generates \(\D(\bun_G^{c_1=\chi},k)^\wedge_{\varphi}\), see also the discussion in the proof of \cite[Theorem 1.8.]{t-exact}.
\end{proof}

%% file: explicit-hecke-operators.tex
\section{Explicit computation of Hecke operators at parameters of Langlands-Shahidi type}\label{sec: computation Hecke operators}

\begin{notn}
    \phantom{.}
    \begin{itemize}
        \item We write \(W_G\) for the Weyl group of \(\GL_n\), we identify this with automorphisms of the set \(\{1,\dots,n\}\).
        \item We define \(\oo(a|b)\defined\oo(a/b)^{\oplus\gcd(a,b)}\) for vector bundles on the Fargues-Fontaine curve.
        
        \item Given \(b\in B(\mathrm{GL}_n)\) we write \(\numin(b)\) for the minimal slope of \(b\) and \(\numax(b)\) for the maximal slope of \(b\).
        We use a similar notation for vector bundles, we define \(\numin(0)=\infty\).
        Any vector bundle \(\mathcal{E}=\mathcal{E}'\oplus\mathcal{E}^{\min}\) such that 
        \(\mathcal{E}^{\min}\) is semistable and \(\numin(\mathcal{E}')>\numax(\mathcal{E}^{\min})\).
        We write \(\rkmin(\mathcal{E})\defined\rk(\mathcal{E}^{\min})\) and \(\degmin(\mathcal{E})\defined\deg(\mathcal{E}^{\min})\).
        Analogously we will write \(\degmin(b)=\degmin(\mathcal{E}_b)\) and \(\rkmin(b)=\rkmin(\mathcal{E}_b)\). 
        
        \item We write \(\chi\succeq\chi'\) for \(\chi,\chi'\in\bbZ^r\) if we have \(\chi(i)\geq\chi'(i)\) for each \(i\).
        \item We write \(|\chi|\defined\chi(1)+\dots+\chi(r)\).
        
        \item We have a modulus character \(\delta_P\) and for \(\theta\in B(M)\) we have \(\delta_{P,\theta}\) as in \cite[Section 3.1]{imaihamann}.
        
        \item We will be careful distinguishing the symbols ``\(\oplus\)'' and ``\(\times\)''.
        When we write \(\mathcal{E}\oplus\mathcal{F}\) we mean this as a \(\GL_{\rk(\mathcal{E})+\rk(\mathcal{F})}\)-bundle, whereas when we write \(\mathcal{E}\times\mathcal{F}\) we mean this as a \(\GL_{\rk(\mathcal{E})}\times\GL_{\rk(\mathcal{F})}\) bundle.

        \item We write \(\ind_{P(E)}^{G(E)}\) for the induction and \(\res_{P(E)}^{G(E)}\) for the restriction, and write
        \begin{equation*}
        \nind_{P(E)}^G(E)\cong\ind_{P(E)}^{G(E)}(-\otimes\delta_P^{1/2}) \text{ and } \nres_{P(E)}^{G(E)}\cong\delta^{-1/2}_P\res_{P(E)}^{G(E)}
        \end{equation*}
        for the normalized induction respectively normalized restriction.
        
        \item We let \(\eta_{M,b}\) denote the natural embedding of a parameter in \(\widehat{L}\) to a parameter in \(\widehat{M_b}\).
        
        \item Given \(\mu\) a dominant cocharacter of \(G\), we write \(T_{\mu}\defined T_{\Delta(\mu)}\), where \(\Delta(\mu)\) is the standard representation of \(\widehat{G}\) attached to \(\mu\).
        
        \item We write \(\std=(1,0^{(n-1)})\) for the cocharacter giving rise to the standard representation of \(\GL_n\) and \(\stdd=(0^{(n-1)},-1)\) for the cocharacter giving rise to the dual of the standard representation for \(\GL_n\).
        
        \item Given a split torus \(T\) with isomorphism \(T\cong\bbG_m^r\) we identify \(X^*(T)\) and \(X_*(T)\) with \(\bbZ^r\) and also functions \(\{1,\dots,r\}\to\bbZ\) and write \(\chi_c\) for the function that is \(\chi_c(c)=1\) and \(\chi_c(i)=0\) otherwise.
        
        \item We fix will always use the isomorphism \(\bbG_m\cong Z(\GL_n)\) sending \(x\mapsto\mathrm{diag}(x,\dots,x)\).
        This gives rise to an isomorphism \(\bbG_m^{\rk(Z(\widehat{L}))}\cong Z(\widehat{L})\). 
        
        \item For \(\chi\in X^*(Z(\widehat{L}))\) we write \(b_{\chi,L}\) for the image of \(\chi\) under the composite
        \begin{equation*}
            X^*(Z(\widehat{L}))\cong B(L)_{\mathrm{basic}}\to B(G).
        \end{equation*}
        \item We write $B(G)_L\defined\im(B(L)_{\mathrm{basic}}\to B(G))$.
        
        \item Given a semi-standard parabolic group \(P\) we write \(\overline{P}\) for the corresponding opposite parabolic.
        
        \item We write \(P^b\) for the semi-standard parabolic subgroup of \(G\) attached to \(\nu_b\) such that the action of \(\nu_b\) on \(\mathrm{Lie}(P^b)\) is non-negative.
    \end{itemize}
    We will omit the coefficients, they will be in a fixed algebraically closed \(\bbZ_{\ell}[\sqrt{q}]\)-field \(k\) of arbitrary characteristic. Let \(b,b'\in B(G)\).
    For a cocharacter \(\mu\) we write \(b\xto{\mu}b'\) if there is a modification of meromorphy \(\mu\) between \(\mathcal{E}_b\dasharrow\mathcal{E}_{b'}\).
    We use a similar notation for \(G\)-bundles.
\end{notn}
Let us start right off with the statement of the main theorem:
\begin{thm}\label{thm: explict hecke ! stalk}
    Let \(\phi\) be an parameter of Langlands-Shahidi type with cuspidal support \((\widehat{L},\varphi)\) and write \(\varphi=\varphi_1\times\dots\times\varphi_r\). This allows us to define \(\oo(\chi,\varphi)\). Let \(n_i\defined\dim(\varphi_i)\) and let \(b,b'\in B(G)_L\). 
    Let \(b=\oo(a_1|n_{1})\oplus\dots\oplus\oo(a_r|n_{r})\) such that \(a_i/n_{i}\geq a_j/n_{j}\) for \(i\geq j\) and let \(w\in\Sigma_r\) be a permutation such that \(a'_{w(i)}/n_{w(i)}\geq a'_{w(j)}/n_{w(j)}\) for \(i\geq j\) where \(a'_j=a_j+\chi(j)\).
    
    We restrict \(i_{b'}^{\ren!}\oo(\chi,\varphi)*i_{b*}^{\ren}\) to a functor
    \begin{equation*}
        i_{b'}^{\ren!}\oo(\chi,\varphi)*i_{b*}^{\ren}\from\D(G_b(E))^\wedge_{\eta_{G,b}\circ\varphi}\to\D(G_{b'}(E))
    \end{equation*}
    then we have the following properties:
    \begin{enumerate}
        \item[(1)] The functor's essential image is contained in \(\D(G_{b'}(E))^\wedge_{\eta_{G,b'}\circ\varphi^{w}}\) if 
            \[b'=\oo(a'_{w(1)}|n_{w(1)})\oplus\dots\oplus\oo(a'_{w(r)}|n_{w(r)}).\]
        \item[(2)] If the condition on \(b'\) from \((1)\) is not met then the functor is \(0\).
        
        \item[(3)] In the situation of \((1)\) the functor \(i_{b'}^{\ren!}\oo(\chi,\varphi)*i_{b*}^{\ren}\) to a functor
    \begin{equation*}
        i_{b'}^{\ren!}\oo(\chi,\varphi)*i_{b*}^{\ren}\from\D(G_b(E))^\wedge_{\eta_{G,b}\circ\varphi}\to\D(G_{b'}(E))^\wedge_{\eta_{G,b'}\circ\varphi^w}
    \end{equation*}
    is an equivalence of categories.
    \end{enumerate}
\end{thm}
\begin{rem}
    For $b=1$ this $w$ is precisely the one considered in \cite[Proposition 4.23.]{imaihamann}, where we identify $W_{L,b}$ with a subset of $W_G$ as explained before \cite[Lemma 4.20.]{imaihamann}.
    Compare also with the element \(w\) constructed in \cite[Section 4.1.]{bm-o-parametrization}.
\end{rem}
\begin{rem}
    The reader might have expected a formula talking about \(i_{b'}^{\ren*}\oo(\chi,\varphi)*i_{b!}^\ren\).
    Indeed, the same holds true for this functor, see \cref{cor: explict hecke * stalk}.
    The reason for considering \(\func{b}{b'}{\chi}{\varphi}\) comes from the fact that for miniscule \(\mu\) we have 
    \begin{equation*}
        s_{b'}^*i_{b'}^!T_\mu i_{b*}s_{b*}\pi\cong\delta_{b'}^{-1}\otimes\rhom_{G_b(E)}(R\Gamma_c(G,b,b',\mu),\pi)[-2d_{b'}],
    \end{equation*}
    where
    \begin{equation*}
        R\Gamma_c(G,b,b',\mu)\defined \colim_{\substack{K\subset G_{b'}(E)\\ \text{compact open}}}R\Gamma_c(\sht(G,b,b',\mu)/K,k[d_{\mu}])(d_{\mu}/2)
    \end{equation*}
    with \(\sht(G,b,b',\mu)\) being the moduli of modifications \(\mathcal{E}_b\dashrightarrow\mathcal{E}_{b'}\) bounded by \(\mu\), \(d_\mu=\langle2\rho_G,\mu\rangle\) and \(d_{b'}=\langle2\rho_G,\nu_{b'}\rangle\)
    This fact is used in the proof of \cref{cor: hodge-newton renormalized reducible reduction}.
    If we work with \(i_{b'}^{\ren*}\oo(\chi,\varphi)*i_{b!}^\ren\), then we need to work with the formula
    \begin{equation*}
        s_{b'}^*i_{b'}^*T_\mu i_{b!}s_{b*}\pi\cong R\Gamma_c(G,b,b',\mu)\otimes_{\mathcal{H}(G_b(E))}^L(\pi\otimes\delta_b)[2d_b].
    \end{equation*}
    This possible, and the proof of \cref{cor: hodge-newton renormalized reducible reduction} and therefore the rest of this section goes through with little modification.
    Hoever the literature prefers working with \(\rhom_{G_b}(-,-)\) instead of \(-\otimes_{\mathcal{H}(G_b(E))}^L-\), which is ultimately we decided to consider \(\func{b}{b'}{\chi}{\varphi}\).
\end{rem}
We want to prove this theorem by induction on \(|\chi|\), but for this we need to show that it is enough to prove it for \(\chi\succeq 0\), this is the statement of the following lemma:
\begin{lem}\label{lem: non-negative slopes suffice}
    Assume that \(\func{1}{b'}{\chi}{\varphi}\) satisfies \cref{thm: explict hecke ! stalk} for all \(\chi\succeq 0\) and all \(b'\) without negative slopes.
    Then \cref{thm: explict hecke ! stalk} holds in general.
\end{lem}
\begin{proof}
    We want to compute \(\func{b}{b'}{\chi}{\varphi}\) for general \(b,b'\in B(G)\) and \(\chi\in X^*(Z(\widehat{L}))\).
    We have \(\oo(\chi,\varphi)*T_{\det^{\otimes n}}-\cong\oo(\chi',\varphi)*-\), for some \(\chi'\) such that \(\chi'(i)>\chi(i)\) for all \(i\).
    In particular choosing \(n\) big enough we obtain that \(\chi'\succeq 0\).
    Note, \(T_{\det}\) identifies with \(q^*\) where \(p\from\bun_G\to\bun_{G_b}\cong\bun_{G}\) with \(b=\oo(1)^n\). The first map is the isomorphism \cite[Corollary III.4.3.]{geometrization} and the second map comes from the isomorphism \(G\cong G_b\).
    This is because if \(\mathcal{E}\) has a modification to \(\mathcal{E}'\) bounded by \(\mu_{\det}\), then \(\mathcal{E}'\cong\mathcal{E}\otimes\oo(1)\) (see also \cite[Théorème 6.3.]{nguyen-central-twist}).
    Thus we only need to obtain \cref{thm: explict hecke ! stalk} for \(\func{b}{b'}{\chi}{\varphi}\) with \(\chi\succeq 0\) and \(b,b'\in B(G)\) such that \(\numin(b),\numin(b')\geq 0\).
    This means we need to understand \(i_{b'}^{\ren!}\oo(\chi,\varphi)*i_{b*}^\ren\).
    By the assumption we can find a \(\xi\) such that \(i_b^{\ren!}\oo(\xi,\varphi)*i_{1*}^\ren\from\D(G(E))^\wedge_{\eta_{G,1}\circ\varphi}\to\D(G_{b}(E))^\wedge_{\eta_{G,b}\circ\varphi}\) is an equivalence of categories and \(i_{b''}^{\ren!}\oo(\xi,\varphi)*i_{1*}^\ren=0\) for \(b''\neq b'\).
    Thus we have  \(i_{b*}^\ren i_b^{\ren!}\oo(\xi,\varphi)*i_{1*}\cong\oo(\xi,\varphi)*i_{1*}^\ren\).
    Therefore, to compute the claim about \(i_{b'}^{\ren!}\oo(\chi,\varphi)*i_{b*}^\ren\) it suffices to check it for \(i_{b'}^{\ren!}\oo(\chi\xi,\varphi)*i_{1*}^\ren\), but we know this by the assumption.
\end{proof}
\begin{custom}{Inductive Hypothesis}\label{inductive hypothesis}
    From now on, assume the following inductive hypothesis:
    \begin{enumerate}
        \item Assume that \cref{thm: explict hecke ! stalk} holds for all \(\chi\succeq 0\) and \(|\chi|<s\) for a natural number \(s\).
        \item Assume that \cref{thm: explict hecke ! stalk} holds for all parameters of Langlands-Shahidi type for \(\GL_{n'}\) with \(n'<n\).
    \end{enumerate}
\end{custom}
\begin{rem}\label{rem: main thm trivial cases}
    Before proceeding, let us include some simple remarks that one should always keep in mind.
    \begin{itemize}
        \item If \(i_{b'}^*T_\mu i_{b!}\pi\neq 0\), then \(b\xto{\mu}b'\).
        \item By \cref{lem: support sheaf with L-parameter of Langlands-Shahidi type}, \(i_{b'}^{\ren!}\oo(\chi,\varphi)*i_{b*}^{\ren}=0\) if at least one of \(b\) or \(b'\) are not in \(B(G)_L\).
        \item If \(\chi\succeq 0\), then \(\func{1}{b}{\chi}{\varphi}=0\) for \(b\) with \(\numin(b)<0\). Writing \(\func{1}{b}{\chi}{\varphi}\) as a direct summand of \(i_{b}^{\ren!}T_{\std}^{|\chi|}i_{1*}^\ren\) it suffices by induction on \(|\chi|\) to check that if \(b\xto{\std}b'\) with \(\numin(b)\geq 0\), then \(\numin(b')\geq 0\).
        This can be easily seen by the Harder-Narasimhan formalism.
    \end{itemize}
    In particular, if there is no modification \(b\xto{\std}b'\) or one of \(b\) or \(b'\) are not in \(B(G)_L\), then \(\func{b}{b'}{\chi_c}{\varphi}=0\) and \cref{thm: explict hecke ! stalk} is trivially true.
    Finally if \(\chi\succeq 0\) and \(\numin(b')<0\), then \(\func{1}{b'}{\chi}{\varphi=0}\) and \cref{thm: explict hecke ! stalk} is trivially true.
\end{rem}
Let us continue with the inductive step.
Given \(\chi\) with \(|\chi|=s\) we can write
\begin{align*}
    i_{b}^{\ren!}\oo(\chi,\varphi)*i_{1!}^\ren&\cong i_{b}^{\ren!}\oo(\chi_c,\varphi)*\oo(\chi\chi_c^{-1},\varphi)*i_{1!}^\ren\\
    &\cong i_{b}^{\ren!}\oo(\chi_c,\varphi)*i_{b_{\chi\chi_c^{-1},L}*}^\ren i_{b_{\chi\chi_c^{-1},L}}^{\ren!}\oo(\chi\chi_c^{-1},\varphi)*i_{1*}^\ren
\end{align*}
where we used Inductive Hypothesis \cref{inductive hypothesis}(1) in the last isomorphism.
By induction we can fully compute \(i_{b_{\chi\chi_c^{-1},L}}^{\ren!}\oo(\chi\chi_c^{-1},\varphi)*i_{1*}^\ren\).
Thus we can prove \cref{thm: explict hecke ! stalk} if we know that it behaves well under composition of Hecke operators and if we can compute \(i_{b}^{\ren!}\oo(\chi_c,\varphi)*i_{b_{\chi\chi_c^{-1},L}*}^\ren\).
For the first claim, we have the following lemma:
\begin{lem}\label{lem: combine chi}
    Fix \(b,b',b''\in B(G)_L\) and \(\chi,\chi'\).
    Assume that \cref{thm: explict hecke ! stalk} holds for the functor \(i_{\beta}^{\ren!}\oo(\chi,\varphi)*i_{b*}^{\ren}\) for any \(\beta\in B(G)_L\) and the functor \(i_{b''}^{\ren!}\oo((\chi')^\sigma,\varphi)*i_{b'*}^{\ren}\) for any \(\sigma\in\Sigma_r\) such that \(L^\sigma\subset\widehat{G_{b'}}\).
    Then we have \cref{thm: explict hecke ! stalk} for \(i_{b''}^{\ren!}\oo(\chi\chi',\varphi)*i_{b*}^{\ren}\).
\end{lem}
\begin{proof}
    By definition $\oo(\chi,\varphi)=\oo(\chi^\sigma,\varphi^\sigma)$.
    We also have $\eta_{G,b}\circ\varphi=\eta_{G,b}\circ\varphi^\sigma$.
    Twisting by \(\sigma\) we see that we have \cref{thm: explict hecke ! stalk} for any \(i_{b''}^{\ren!}\oo(\chi',\varphi^\sigma)*i_{b'*}^{\ren}\) for any \(\sigma\in\Sigma_r\) such that \(L^\sigma\subset\widehat{G_{b'}}\).
    Let \(b=\oo(a_1|n_{1})\oplus\dots\oplus\oo(a_r|n_{r})\) such that \(a_i/n_{i}\geq a_j/n_{j}\) for \(i\geq j\) and let \(w\in\Sigma_r\) be a permutation such that \(a'_{w(i)}/n_{w(i)}\geq a'_{w(j)}/n_{w(j)}\) for \(i\geq j\) where \(a'_j=a_j+\chi(j)\).
    \cref{thm: explict hecke ! stalk}(2) with \(i_{\beta}^{\ren!}\oo(\chi,\varphi_{1}\times\dots\times\varphi_r)*i_{b*}^{\ren}\) for any \(\beta\) implies that we have an isomorphism
    \begin{equation*}
        \oo(\chi,\varphi)*i_{b*}^\ren\cong i_{b'*}^\ren i_{b'}^{\ren!}\oo(\chi,\varphi)*i_{b*}^\ren
    \end{equation*}
    for \(b'=\oo(a'_{w(1)}|n_{w(1)})\oplus\dots\oplus\oo(a'_{w(r)}|n_{w(r)})\) and \cref{thm: explict hecke ! stalk}(1) tells us that \(i_{b'}^{\ren!}\oo(\chi)*i_{b*}^{\ren}\) lands in \(\D(G_{b'}(E))^\wedge_{\eta_{M,b'}\circ\varphi^{w}}\).
    We have that \(L^w\subset\widehat{G_{b'}}\).
    Then we have
    \begin{equation*}
        i_{b''}^{\ren!}\oo(\chi\chi',\varphi)*i_{b*}^\ren=i_{b''}^{\ren!}\oo(\chi,\varphi)*\oo(\chi',\varphi)*i_{b*}^\ren=i_{b''}^{\ren!}\oo(\chi,\varphi)*i_{b'*}^\ren i_{b'}^{\ren!}\oo(\chi',\varphi)i_{b*}^\ren .
    \end{equation*}
\end{proof}
We are left with computing \(i_{b}^{\ren!}\oo(\chi_c,\varphi)*i_{b_{\chi\chi_c^{-1},L}*}^\ren\).
Recall that by \cref{cor: hom galois representation}, we can write \(\rhom_{W_E}(\varphi_c,\varphi_c)\otimes\oo(\chi_c)*-\cong\rhom_{W_E}(\varphi_c,T_{\std}(-))\).
As a first step we therefore want to understand \(i_{b}^{\ren!}T_{\std}*i_{b_{\chi\chi_c^{-1},L}*}^\ren\).
This is particularly simple in the Hodge-Newton reducible case, let us recall the relevant definitions.

\begin{defn}
    A triple \((b,b',\mu)\) is called \defword{Hodge-Newton reducible for \(\GL_{m_1}\times\GL_{m_2}\)} if we have \(\theta,\theta'\in B(M)\) whose image in \(B(G)\) are \(b\) and \(b'\) respectively, \(\mu\) has a \(G\)-dominant reduction \(\mu_M\), we have \(\kappa(\theta')-\kappa(\theta)=\mu_M^{\sharp_M}\), where $\mu_M^{\sharp_M}$ is the image of $\mu_M$ under the map $X^*(\widehat{T})\to X^*(Z(\widehat{G}))=\pi_1(G)$, and when writing \(\theta'=(\theta'_1,\theta'_2)\) and \(\theta=(\theta_1,\theta_2)\) with \(\theta_i,\theta'_i\in B(\GL_{m_i})\) we have \(\numax(\theta'_2)<\numin(\theta'_1)\) and \(\numax(\theta_2)\leq\numin(\theta_1)\).
    We call such a triple \((\theta,\theta',\mu_M)\) a reduction to \(M\).

    We say it is \defword{\(\omega\)-Hodge-Newton reducible for  \(\GL_{m_1}\times\GL_{m_2}\)} if we have \(\theta,\theta'\in B(M)\) whose image in \(B(G)\) are \(b\) and \(b'\) respectively, \(\mu\) has a \(M\)-dominant reduction \(\mu_M\) such that \(\omega\mu_M\) is \(G\)-dominant where \(\omega\in W_G\) is the permutation which switches the first \(m_1\) numbers with the last \(m_2\) numbers, we have \(\kappa(\theta')-\kappa(\theta)=\mu_M^{\sharp_M}\) and when writing \(\theta=(\theta_1,\theta_2)\) and \(\theta'=(\theta'_1,\theta'_2)\) with \(\theta_i,\theta'_i\in B(\GL_{m_i})\) we have \(\numax(\theta_2)<\numin(\theta_1)\) and \(\numax(\theta'_2)\leq\numin(\theta'_1)\).
    Similarly, we will call such a triple \((\theta,\theta',\mu_M)\) a reduction to \(M\).
\end{defn}
\begin{rem}\label{rem: reducible definition}
    We will only use this condition for \(\mu=\std\) and one time for \(\mu=\stdd\) in the proof of \cref{lem: hecke functor is equivalence of categories}.
    For \(\std\) note that \((\std,0)\) is the unique \(G\)-dominant reduction to \(M\).
    In this case, writing \(\theta=(\theta_1,\theta_2)\) and \(\theta'=(\theta'_1,\theta'_2)\) in the definition, we can explicitly unwind the condition of Hodge-Newton reduciblity.
    Then Hodge-Newton reducible translates to \(\theta_2=\theta'_2\) and \(\kappa(\theta'_1)-\kappa(\theta_1)=1\) as well as \(\numax(\theta'_2)>\numin(\theta'_1)\) and \(\numax(\theta_2)\geq\numin(\theta_1)\).
    
    In the case where we are \(\omega\)-Hodge-Newton reducible, \((0,\std)\) is the unique reduction to \(M\) such that \(\omega(0,\std)=\std\).
    This is the \(M\)-dominant reduction of \(\std\) such that applying \(\omega\) to it is \(G\)-dominant.
    The condition on \(\theta\) and \(\theta'\) translates to \(\theta_1=\theta'_1\) and \(\kappa(\theta'_2)-\kappa(\theta_2)=1\) as well as \(\numax(\theta_2)>\numin(\theta_1)\) and \(\numax(\theta'_2)\geq\numin(\theta'_1)\).
    To shorten the notation we say a pair \((b,b')\) is \defword{reducible} if \((b,b',\std)\) is Hodge-Newton or \(\omega\)-Hodge-Newton reducible.
    Most often this is used in the situation where there is a modification \(b\xto{\std}b'\), in this case we call the modification reducible.
\end{rem}
In this situation, the Hecke operators are parabolically induced, more precisely we have:
\begin{thm}\label{cor: hodge-newton renormalized reducible reduction}
    Let \(M\defined\GL_{m}\times\GL_{m'}\) with \(m+m'=n\) and \(\mu\) miniscule and \(b,b'\in B(G)\).
    Let \(P_{b}\) the parabolic in $G_b$ corresponding to the parabolic \(M\overline{P^b}\subset G\),
    define \(P_{b'}\) similarly.
    \begin{enumerate}
        \item Assume that \((b,b',\mu)\) is Hodge-Newton reducible with reduction \((\theta,\theta',\mu_M)\). Let \(\pi\in\D(G_{b}(E))\). Then we have 
        \begin{equation*}
            i_{b'}^{\ren!}T_\mu i_{b*}^\ren \pi\cong i_{\theta'}^{\ren!}T_{\mu_M} i_{\theta*}^\ren\nres^{G_{b}(E)}_{\overline{P_{b}}(E)}\pi[\langle2\rho_G,\nu_{b'}-\mu-\nu_b\rangle-\langle2\rho_M,\nu_{\theta'}-\mu_M-\nu_{\theta}\rangle].
        \end{equation*}
        \item Assume that \((b,b',\mu)\) is \(\omega\)-Hodge-Newton reducible with reduction \((\theta,\theta',\mu_M)\).  Let \(\pi\in\D(G_{b}(E))\). Then we have     \begin{equation*}
            i_{b'}^{\ren!}T_\mu  i_{b*}^\ren\pi\cong\nind_{P_{b'}(E)}^{G_{b'}(E)}i_{\theta'}^{\ren!}T_{\mu_M}i_{\theta*}^\ren\pi[\langle2\rho_G,\nu_{b'}-\mu-\nu_b\rangle-\langle2\rho_M,\nu_{\theta'}-\mu_M-\nu_{\theta}\rangle].
        \end{equation*}
    \end{enumerate}
\end{thm}
\begin{proof}
    This follows as in the proof of \cite[Theorem 4.12.]{imaihamann} using \cite[Proposition 5.1.]{nguyen} and \cite[Corollary 5.2.]{nguyen} and a careful analysis of the twists appearing.
    Let us do this for (1), (2) is similar.
    Let \(\sht(G,b,b',\mu)\to\spd(\breve{E})\) the moduli space of modifications \(\mathcal{E}_b\dashrightarrow\mathcal{E}_{b'}\) bounded by \(\mu\), where \(\breve{E}\) is the maximal unramified extension of \(E\).
    This carries a Weil-descent datum, so it obtains a \(W_E\times G_b(E)\times G_{b'}(E)\) action by the action of \(G_b(E)\) on \(\mathcal{E}_b\).
    We write 
    \[
        R\Gamma_c(G,b,b',\mu)\defined \colim_{\substack{K\subset G_{b'}(E)\\ \text{compact open}}}R\Gamma_c(\sht(G,b,b',\mu)/K,k[d_{\mu}])(d_{\mu}/2)
    \]
    where \(d_\mu=\langle 2\rho_G,\mu\rangle\).
    We have 
    \begin{equation*}
        s_{b'}^*i_{b'}^!T_\mu i_{b*}s_{b*}\pi\cong\delta_{b'}^{-1}\otimes\rhom_{G_b(E)}(R\Gamma_c(G,b,b',\mu),\pi)[-2d_{b'}].
    \end{equation*}
    Write \(d_{\mu}\defined\langle2\rho,\mu\rangle\), \(d_{\mu_{M}}\defined\langle2\rho_M,\mu_M\rangle\) \(d_b\defined\langle2\rho_G,\nu_b\rangle\), \(d_\theta\defined\langle2\rho_M,\nu_\theta\rangle\) and \(\delta\defined\delta_{\theta'}\delta_{b'}^{-1}\) (here we use the isomorphism \(M_{\theta'}\cong G_{b'}\)).
    Now using \cite[Proposition 5.1.]{nguyen} and \cite[Corollary 5.2.]{nguyen} we deduce that we have
    \begin{equation*}
        R\Gamma_c(G,b,b',\mu)\cong\delta\otimes\ind_{P_b(E)}^{G_b(E)}R\Gamma_c(M,\theta,\theta',\mu_M)(-\frac{d_{\mu_M}-d_\mu}{2})[-(d_{\mu_M}-d_\mu)-(2d_{b'}-2d_{\theta'})].
    \end{equation*}
    The Tate twist and additional shift comes from the different normalization of Tate twists and shifts in the definition of \(R\Gamma_c(G,b,b',\mu)\) for differing \(G\).
    We obtain 
    \begin{align*}
        &s_{b'}^*i_{b'}^!T_\mu i_{b*}s_{b*}\pi\\
        \cong&\delta_{b'}^{-1}\rhom_{G_b(E)}(R\Gamma_c(G,b,b',\mu),\pi)[-2d_{b'}]\\
        \cong&\delta_{b'}^{-1}\rhom_{G_b(E)}(\delta\otimes\ind_{P_b(E)}^{G_b(E)}R\Gamma_c(M,\theta,\theta',\mu_M),\pi)(\frac{d_{\mu_M}-d_\mu}{2})[d_{\mu_M}-d_\mu-2d_{\theta'}]\\
        \cong&\delta_{b'}^{-1}\delta^{-1}\rhom_{M_{\theta}(E)}(R\Gamma_c(M,\theta,\theta',\mu_M),\delta_{P_b}^{1/2}\nres_{\overline{P_b}(E)}^{G_b(E)}\pi)(\frac{d_{\mu_M}-d_\mu}{2})[d_{\mu_M}-d_\mu-2d_{\theta'}]\\
        \cong&\delta_{\theta'}^{-1}\delta_{P,\theta'}^{1/2}\rhom_{M_{\theta}(E)}(R\Gamma_c(M,\theta,\theta',\mu_M),\delta_{P,\theta}^{-1/2}\delta_{P_b}^{1/2}\nres_{\overline{P_b}(E)}^{G_b(E)}\pi)[d_{\mu_M}-d_\mu-2d_{\theta'}]\\
        \cong&\delta_{P,\theta'}^{1/2}\otimes s_{\theta'}^*i_{\theta'}^!T_{\mu_M} i_{\theta*}s_{\theta*}\delta_{P,\theta}^{-1/2}\delta_{P_{b}}^{1/2}\nres^{G_{b}(E)}_{\overline{P_{b}}(E)}\pi[d_{\mu_M}-d_\mu],
    \end{align*}
    where we used \cite[Proposition 4.10.]{imaihamann} for the forth isomorphism and \cite[Corollary 1.3.]{finiteness-p-adic} for the third isomorphism.
    We have
    \begin{align*}
        &i_{b'}^{\ren!}T_{\mu_M} i_{b*}^\ren\pi\\
        \cong &\delta_{b'}^{1/2}\otimes s_{b'}^*i_{b'}^!T_{\mu_M} i_{b*}s_{b*}\pi\otimes\delta_{b}^{-1/2}[\langle 2\rho_G,\nu_{b'}-\nu_b\rangle]\\
        \cong&\delta_{b'}^{1/2}\delta_{P,{\theta'}}^{-1/2}\otimes s_{\theta'}^*i_{\theta'}^!T_{\mu_M} i_{\theta*}s_{\theta*}\delta_{\theta}^{-1/2}\delta_{\theta}^{1/2}\delta_{P,\theta}^{-1/2}\delta_{P_{b}}^{1/2}\nres_{\overline{P_{b}}(E)}^{G_b(E)}\pi\otimes\delta_{b}^{-1/2}[d_{\mu_M}-d_\mu+d_{b'}-d_b]\\
        \cong&\delta_{b'}^{1/2}\delta_{P,{\theta'}}^{-1/2}\otimes s_{\theta'}^*i_{\theta'}^!T_{\mu_M} i_{\theta*}s_{\theta*}\delta_{\theta}^{-1/2}\delta_{\theta}^{1/2}\delta_{P,\theta}^{-1/2}\delta_{P_{b}}^{1/2}\delta_{b}^{-1/2}|_{M_b}\nres_{\overline{P_{b}}(E)}^{G_b(E)}\pi[d_{\mu_M}-d_\mu+d_{b'}-d_b].
    \end{align*}
    In the last line we used that \(\delta_{b}^{-1/2}\) is trivial on the unipotent radical of \(\overline{P_b}\).
    Let us compute \(\delta_{\theta}^{1/2}\delta_{P,\theta}^{-1/2}\delta_{P_{b}}^{1/2}\delta_{b}^{-1/2}|_{M_b}\).
    Since this representation factors through \(M^{\mathrm{ab}}\), we can compute this via the Langlands correspondence for tori.
    We see that the \(L\)-parameter \(\varphi'\) of this is
    \begin{equation*}
        \varphi'(w)=(2\rho_{\widehat{M}}-2\rho_{\widehat{M_{\theta}}}+2\rho_{\widehat{G}}-2\rho_{\widehat{M}}+2\rho_{\widehat{M_{\theta}}}-2\rho_{\widehat{G_{b}}}+2\rho_{\widehat{G_{b}}}-2\rho_{\widehat{G}})(\sqrt{q})^{|w|}=1
    \end{equation*}
    Thus we obtain that \(\delta_{\theta}^{1/2}\delta_{P,\theta}^{-1/2}\delta_{P_{b}}^{1/2}\delta_{b}^{-1/2}|_{M_b}=1\) and observe \(\delta_{b'}^{1/2}\delta_{P,{\theta'}}^{-1/2}\otimes s_{\theta'}^*i_{\theta'}^![2d_{\theta'}]=i_{\theta'}^{\ren!}\).
    Accounting for the shifts we get that 
    \begin{align*}
        &\delta_{b'}^{1/2}\delta_{P,{\theta'}}^{-1/2}\otimes s_{\theta'}^*i_{\theta'}^!T_{\mu_M} i_{\theta*}s_{\theta*}\delta_{\theta}^{-1/2}\delta_{\theta}^{1/2}\delta_{P,\theta}^{-1/2}\delta_{P_{b}}^{1/2}\delta_{b}^{-1/2}|_{M_b}\nres_{\overline{P_{b}}(E)}^{G_b(E)}\pi[d_{\mu_M}-d_\mu+d_{b'}-d_b]\\
        \cong&i_{\theta'}^{\ren!}T_{\mu_M} i_{\theta*}^\ren\nres^{G_{b}(E)}_{\overline{P_{b}}(E)}\pi[d_{\mu_M}-d_\mu+d_{b'}-d_b-(d_{\theta'}-d_\theta)]
    \end{align*}
    as claimed.
\end{proof}
We now have a way to compute \(i_{b'}^{\ren!}T_{\std}i_{b*}^\ren\) at least when the \((b,b')\) are reducible.
This is enough to see that one only needs to check \cref{thm: explict hecke ! stalk}(1) and (2):
\begin{lem}\label{lem: hecke functor is equivalence of categories}
    Let \(\chi\succeq 0\).
    Let \(b\defined b_{\chi,L}\).
    If \cref{thm: explict hecke ! stalk}(1) and (2) hold, then the functor 
    \begin{equation*}
        i_{b}^{\ren!}\oo(\chi,\varphi)*i_{1*}^{\ren}\from\D(G(E))^\wedge_{\eta_{G,1}\circ\varphi}\to\D(G_{b}(E))^\wedge_{\eta_{G,b}\circ\varphi^{w}}
    \end{equation*}
    is an equivalence of categories, i.e. we obtain \cref{thm: explict hecke ! stalk}(3).
\end{lem}
\begin{proof}
    From what we have proven so far, we know that \(i_{b'}^{\ren!}\oo(\chi,\varphi)*i_{1*}^{\ren}=0\) for \(b'\neq b\), thus \(i_b^{\ren!}\) is fully faithful on the essential image of \(\oo(\chi,\varphi)*i_{1*}^{\ren}\), which is a composition of fully faithful functors, thus \(i_{b}^{\ren!}\oo(\chi,\varphi)*i_{1*}^{\ren}\) is fully faithful.
    It has a right adjoint functor given by \(i_{1}^{\ren*}\oo(\chi^{-1},\varphi)i_{b!}^\ren\).
    It suffices to show that this functor is conservative.
    We now induct on \(|\chi|\), the base case is clear.

    For the inductive step, by adjunction we know that \(i_{1}^{\ren*}\oo(\chi^{-1},\varphi)i_{b!}^\ren\) vanishes unless \(b=b_{\chi,L}\).
    Note that \(i_{1}^{\ren!}\oo((\chi\chi_c)^{-1},\varphi)*i_{b_{\chi\chi_c^{-1}}*}^\ren\cong i_{1}^{\ren!}\oo(\chi^{-1},\varphi)*\oo(\chi_c^{-1},\varphi)*i_{b_{\chi\chi_c}*}^\ren\).
    Let \(b'\defined b_{\chi\chi_c,L}\) and we see that \(\oo(\chi_c^{-1},\varphi)*i_{b'*}^\ren\) admits a filtration with filtered pieces given by \(i_{\beta!}^\ren i_{\beta}^{\ren!}\oo(\chi_c^{-1},\varphi)*i_{b'*}^\ren\) with \(\beta\in B(G)\).
    Then \(i_{1}^{\ren!}\oo(\chi^{-1},\varphi)*-\cong i_{1}^{\ren*}\oo(\chi^{-1},\varphi)*-\) kills every filtered piece except
    \(i_{b_{\chi,L}!}^\ren i_{b_{\chi,L}}^{\ren!}\oo(\chi_c^{-1},\varphi)*i_{b'*}^\ren\).
    Thus we need to check that \(i_{b_{\chi,L}!}^\ren i_{b_{\chi,L}}^{\ren!}\oo(\chi_c^{-1},\varphi)*i_{b'*}^\ren\) is conservative.
    Note, we can always choose \(c\) such that \((b_{\chi,L},b',\stdd)\) is Hodge-Newton reducible, namely take \(c\) such that \(\chi(c)>0\) and \(\chi(c)/n_c\) is minimal among the positive slopes of \(b_{\chi,L}\).
    So, assume that \((b_{\chi,L},b',\stdd)\) is Hodge-Newton reducible.
    Then by \cref{cor: hodge-newton renormalized reducible reduction}, we need to check that \(\nres_{\overline{P_b}}^{G_b}\) is conservative on \(\D(G_{b}(E))^\wedge_{\eta_{M,b}\circ\varphi^{w}}\).
    For this, it suffices to check that for any \(\pi\in\D(G_{b}(E))^\wedge_{\eta_{M,b}\circ\varphi^{w}}\), there is a \(\rho\in\D(M_{\theta}(E))\) such that \(\rhom(\nind_{P_b}^{G_b}\rho,\pi)\neq 0\).
    Now, let \(Q\) be the parabolic attached to \(L\), consider the transfer \(Q^b\) to \(G_b\) with Levi quotient \(L^b\).
    We know these exist since \(b\in\im(B(L)_{\mathrm{basic}}\to B(G))\).
    We now let \(\rho=\nind_{Q^b}^{M_{\theta}}\rho_0\), where \(\rho_0\) is an irreducible supercuspidal representation whose \(L\)-parameter is \(\varphi^w\). 
    Such a representation exists, this follows from \cref{cor: t-exact irreducible case}.
    Then \(\nind_{P_b}^{G_b}\rho\cong\nind_{Q^b}^{G_b}\rho_0\) is an irreducible representation by \cite[THÉORÈME 7.24]{lift-irreducible} whose parameter is \(\eta_{M,b}\circ\varphi^{w}\).
    Also note, any non-zero \(\pi\in\D(G_{b}(E))^\wedge_{\eta_{M,b}\circ\varphi^{w}}\) admits a non-zero map from a irreducible subquotient of \(\nind_{Q^b}^{G_b}\rho_0\) by considering an irreducible subrepresentation in a non-zero cohomology group of \(\pi\), but there is only one, which is itself.
    In particular, we get that \(\rhom(\nind_{P_b}^{G_b}\rho,\pi)\neq 0\).
\end{proof}
To compute \(\func{b}{b'}{\chi_c}{\varphi}\), we need to pass to \(\varphi_c\)-isotypic components and check this results in \cref{thm: explict hecke ! stalk}.
The following lemma shows that this indeed works as expected.
\begin{lem}\label{lem: reduce r}
    Fix \(\chi\succeq 0\) with \(|\chi|=s\).
    Then \cref{thm: explict hecke ! stalk} with \(\func{1}{b'}{\chi}{\varphi}\) holds for those \(b'\) such that there is a \(\xi\succeq 0\) with \(|\xi|=|\chi|-1\), 
    and \((b_{\xi,L},b')\) is reducible.
\end{lem}
\begin{proof}
    Choose a \(\xi\) as in the assumptions.
    Then by assumption we may assume that there is a reducible modification \(b_{\xi,L}\xto{\std}b'\), so \((b_{\xi,L},b',\std)\) is Hodge-Newton reducible or \(\omega\)-Hodge-Newton reducible and \(\numin(b')\geq 0\).
    Assume that \((b_{\xi,L},b',\std)\) is Hodge-Newton reducible with reduction \((\theta,\theta',\mu)\). 
    By \cref{lem: combine chi} we are reduced to check \cref{thm: explict hecke ! stalk} with \(\func{b_{\xi,L}}{b'}{\chi_c}{\varphi^\sigma}\) for any \(\sigma\in\Sigma_r\) such that \(L^\sigma\subset\widehat{G_{b_{\xi,L}}}\), where \(\chi_c=\chi\xi^{-1}\).
    By \cref{lem: combine chi} and Inductive Hypothesis \cref{inductive hypothesis}(1) it is enough to check \cref{thm: explict hecke ! stalk} for \(\func{b_{\xi,L}}{b'}{\chi_c}{\varphi}\) for any \(c\). We want to compute \(\func{b_{\xi,L}}{b'}{\chi_c}{\varphi}\).
    Let \(\pi\in\D(G_b(E))^\wedge_{\eta_{G,b}\circ\varphi}\).
    Using \cref{cor: hodge-newton renormalized reducible reduction}(1) and \cref{cor: hom galois representation} we see that 
    \begin{align*}
        &i_{b'}^{\ren!}\oo(\chi_c,\varphi) i_{b_{\xi,L}*}^\ren\pi\otimes K\\
        \cong &i_{b'}^{\ren!}\rhom_{W_E}(\varphi_c,T_\std i_{b_{\xi,L}*}^\ren\pi)\\
        \cong&\rhom_{W_E}(\varphi_c,i_{\theta'}^{\ren!}T_\mu i_{\theta*}^\ren\nres_{\overline{P_{b}}(E)}^{G_b(E)}\pi)[\langle2\rho_M,\nu_{\theta'}+\mu_M-\nu_{\theta}\rangle-\langle2\rho_G,\nu_{b'}+\mu-\nu_{b_{\xi,L}}\rangle]
    \end{align*}
    for \(K=\rhom_{W_E}(\varphi_c,\varphi_c)\).
    By consideration of parameters we have a product decomposition
    \begin{equation*}
        \nres_{\overline{P_{b}}(E)}^{G_b(E)}\pi=\prod_{\sigma\in\Sigma_r(b,\varphi,M)}\pi_\sigma,
    \end{equation*}
    where \(\sigma\in \Sigma_r(b,\varphi,M)\) if \(\widehat{L}^\sigma\subset \widehat{M_b}\), \(\eta_{G,b}\circ\varphi\cong\eta_{G,b}\circ\varphi^\sigma\) and we have
    \begin{equation*}
        b=\oo(a_1|n_{\sigma(1)})\oplus\dots\oplus\oo(a_r|n_{\sigma(r)})
    \end{equation*}
    with \(a_i/n_{\sigma(i)}\geq a_j/n_{\sigma(j)}\) for \(i\geq j\) and \(\pi_\sigma\in\D(M_\theta(E))^\wedge_{\eta_{M,\theta}\circ\varphi^\sigma}\).

    By compatibility of the spectral action with products of groups, Inductive Hypothesis \cref{inductive hypothesis}(2), \cref{thm: explict hecke ! stalk}(2) and \cref{thm: explict hecke ! stalk}(1) for \(\func{\theta}{\theta'}{\chi}{\varphi'}\) for all \(\varphi'\) a direct summand of \(\varphi\) and all \(\chi\succeq 0\) we compute that 
    \begin{align*}
        \rhom_{W_E}(\varphi_c,i_{\theta'}^{\ren!}T_\mu i_{\theta*}^\ren\pi_\sigma)&\cong i_{\theta'}^{\ren!}\oo(\chi_{\sigma(c)},\varphi^\sigma)* i_{\theta*}^\ren\pi_\sigma\otimes K
    \end{align*}
    vanishes for \(\theta'\neq\theta_{\xi^\sigma\chi_{\sigma(c)},L^\sigma}\) \footnote{The additional twists by \(\sigma\) exist since \(\varphi_{\pi_\sigma}=\eta_{M,\theta}\circ\varphi^\sigma\) and the supercuspidal support of \(\varphi^\sigma\) is \(L^\sigma\).} and lands in \(\D(M_{\theta'}(E))^\wedge_{\eta_{M,\theta'}\circ\varphi^{w\sigma}}\) with \(w\) satisfying 
    \begin{equation*}
        \xi^\sigma\chi_{\sigma(c)}(i)/n_{w\sigma(i)}\geq \xi^\sigma\chi_{\sigma(c)}(j)/n_{w\sigma(j)}
    \end{equation*}
    for \(i\geq j\).
    Thus we obtain \cref{thm: explict hecke ! stalk}(2) and \cref{thm: explict hecke ! stalk}(1) for \(\func{b_{\xi,L}}{b'}{\chi_c}{\varphi}\).
    Using \cref{lem: hecke functor is equivalence of categories} it follows that \cref{thm: explict hecke ! stalk} holds.
\end{proof}
In light of the previous lemma it is clear that we should analyse when there is a reducible modification \(b\xto{\std}b'\).
\begin{lem}\label{lem: modifications increase min and max slopes}
    If we have a modification \(b\xto{\std}b'\), then we have \(\numin(b')\geq\numin(b)\) and \(\numax(b')\geq b\).
    Moreover, in this situation \((b,b',\std)\) is Hodge-Newton reducible if and only if \(\numin(b)=\numin(b')\) and it is \(\omega\)-Hodge-Newton reducible if and only if \(\numax(b)=\numax(b')\).
\end{lem}
\begin{proof}
    If we have a modification \(b\xto{\std}b'\), we obtain an injection \(\mathcal{E}_b\injto\mathcal{E}_{b'}\) that is generically an isomorphism.
    Then we obtain an injection \(\oo(\numax(b))\injto\mathcal{E}_b\injto\mathcal{E}_{b'}\) showing that \(\numax(b')\geq\numax(b)\).

    We similarly have a non-zero map \(\mathcal{E}_b\injto\mathcal{E}_{b'}\surjto\oo(\numin(b'))\), this shows that \(\numin(b)\leq\numin(b')\).
    The claim about reducibility immediately follow from \cref{rem: reducible definition}.
\end{proof}
\begin{rem}\label{rem: modifications from basic b}
    Assume \(b\) is basic and consider \(b'\) such that we have a modification \(b\xto{\std}b'\).
    We know that \(\nu(b')-\nu(b)\leq\mu^{\sharp_G}\) and any \(b'\) satisfying this equality is realized.
    This shows that \(\mathcal{E}_{b'}=\mathcal{F}\oplus\mathcal{E}_{b'}^{\min}\) with \(\mathcal{E}_{b'}^{\min}\) semistable such that \(\numin(\mathcal{F})>\numin(\mathcal{E}_{b'}^{\min})\).
    We know that \(\numin(\mathcal{F})> \frac{\deg(\oo(\numin(b))^k)+1}{\rk(\oo(\numin(b))^k)}\) and \(\frac{\deg(\oo(\numin(b))^k)+1}{\rk(\oo(\numin(b))^k)}\geq \numin(\mathcal{E}_{b'}^{\min})\) with equality if and only if \(\mathcal{F}=0\).
    If \(\mathcal{F}\neq 0\), then \(\rkmin(b')<\rkmin(b)\).
\end{rem}
In the situation where \(b\) is not basic, this allows us to find situations where all modifications reducible.
\begin{lem}\label{lem: all modification reducible condition}
    Assume \(b\) is not basic.
    Write \(\mathcal{E}_b=\mathcal{E}\oplus\oo(\numin(b))^k\) such that \(\numin(\mathcal{E})>\numin(b)\).
    Assume that \(\numin(\mathcal{E})\geq\frac{\deg(\oo(\numin(b))^k)+1}{\rk(\oo(\numin(b))^k)}\).
    If we have a modification \(b\xto{\std}b'\) then it is reducible or \(\rkmin(b')<\rkmin(b)\).
\end{lem}
\begin{proof}
    Since \(b\) was not basic, \(\mathcal{E}_b=\mathcal{E}\oplus\oo(\numin(b))^k\) for some \(k>0\) such that \(\numin(\mathcal{E})>\numin(b)\) and \(\mathcal{E}
    \neq 0\).
    This induces a \(P\)-reduction on \(\mathcal{E}_b\) where \(P\) is the standard parabolic with Levi \(M=\GL_{\rk(\mathcal{E})}\times\GL_{\rk(\oo(\numin(b))^k)}\).
    Using \cite[Lemma 2.4]{P-reduction-modification}, the bundle \(\mathcal{E}_{b'}\) admits a \(P\)-reduction \(\mathcal{E}_{b',P}\) such that \(\mathcal{E}\times\oo(\numin(b))^k\xto{\mu}\mathcal{E}_{b',P}\times^PM\), where \(\mu\) is a reduction of \(\std\) to \(M\).
    We find a \(\GL_{\rk(\mathcal{E})}\)-bundle \(\mathcal{E}'\) and a \(\GL_{\rk(\oo(\numin(b))^k)}\)-bundle \(\mathcal{E}''\) such that \(\mathcal{E}'\times\mathcal{E}''\cong\mathcal{E}_{b',P}\times^PM\).
    We now have two cases.

    \begin{enumerate}
        \item Assume \(\mu=(\std,0)\).
        In this case \(\oo(\numin(b))^k\cong\mathcal{E}''\) and we have a modification \(\mathcal{E}\xto{\std}\mathcal{E}'\).
        By \cref{lem: modifications increase min and max slopes} we have \(\numin(\mathcal{E}')\geq\numin(\mathcal{E})\).
        Therefore the extension problem splits and we obtain \(\mathcal{E}_{b'}\cong\mathcal{E}'\times\oo(\numin(b))^k\) and \(\numin(b')=\numin(b)\) so the modification is reducible.
        \item Assume \(\mu=(0,\std)\).
        In this case \(\mathcal{E}\cong\mathcal{E}'\). We have a modification \(\oo(\numin(b))^k\xto{\std}\mathcal{E}''\).
        As discussed in \cref{rem: modifications from basic b} we have \(\mathcal{E}''=\mathcal{F}\oplus\mathcal{E}_{b'}^{\min}\) and \(\numin(\mathcal{F})> \frac{\deg(\oo(\numin(b))^k)+1}{\rk(\oo(\numin(b))^k)}\).
        In the case where \(\mathcal{F}\neq 0\), by consideration of slopes \(\mathcal{E}_{b'}^{\min}\) splits off as a direct summand and we have \(\numin(\mathcal{E}''/\mathcal{E}_{b'}^{\min})>\numin(\mathcal{E}_{b'}^{\min})\).
        We see that \(\rkmin(b')<\rkmin(b)\).
        In case where \(\mathcal{F}=0\), then \(\mathcal{E}_{b'}^{\min}=\mathcal{E}''\) and by our assumption on the slopes we see that \(\mathcal{E}_{b'}=\mathcal{E}'\oplus\mathcal{E}_{b'}^{\min}\).
        We conclude that \(\numax(b')=\numax(b)\) so the modification is reducible.
    \end{enumerate}
\end{proof}
Summarizing the discussion so far we see that by induction we can prove \cref{thm: explict hecke ! stalk} at least for those \(\chi\succeq 0\) such that for any \(b\in B(G)\) there is a \(c\in\{1,\dots,r\}\) such that \(\chi(c)>0\) and \((b_{\chi\chi_c,L},b)\) are reducible or we are in one of the cases discussed in \cref{rem: main thm trivial cases}.
Sadly, this will not be the case in general, and there will be some \(b\in B(G)\) where we cannot find such a \(c\) and that are not covered by \cref{rem: main thm trivial cases}.
However, for \(b=b_{\chi,L}\) one can find such a \(c\) as discussed above, so we are left with checking the vanishing of certain stalks \(\func{1}{b}{\chi}{\varphi}\).
This can be checked by probing them with previously computed Hecke operators, more precisely we have the following lemma:
\begin{lem}\label{lem: P true up to finite set}
    Fix a \(\chi\succeq 0\) and fix a finite subset \(S\subset |\bun_G|\) with closure \(\overline{S}\subset|\bun_G|\).
    Assume the following:
    \begin{enumerate}
        \item we have \(\numin(b)\geq 0\) for all \(b\in S\);
        \item we have \(\func{1}{b}{\chi}{\varphi}=0\) for all \(b\in\overline{S}\setminus S\);
        \item\label{item: xi property} the statement of \cref{thm: explict hecke ! stalk} holds for \(\func{1}{b'}{\xi}{\varphi}\) for all \(b'\) and all \(\xi\) such that \(b_{\xi,L}\in S\);
        \item for any \(\xi\) as in \cref{item: xi property} there exists a \(1\leq c\leq r\) (which might depend on \(\xi\)), satisfying \(\chi(c),\xi(c)>0\) and \(b_{\chi\chi_{c}^{-1},L}\notin\overline{\{b_{\xi\chi_{c}^{-1},L}\}}\setminus b_{\xi\chi_c^{-1},L}\).
    \end{enumerate}
    Then \(\func{1}{b}{\chi}{\varphi}=0\) for all \(b\in S\).
    The same statement holds true if instead we change the assumptions (2) and (4) into 
    \begin{enumerate}
        \item[(2')] we have \(\func{1}{b}{\chi}{\varphi}=0\) for all \(b\) such that \(\overline{\{b\}}\cap S=\emptyset\);
        \item[(4')] for any \(\xi\) as in \cref{item: xi property} there exists a \(1\leq c\leq r\) (which might depend on \(\xi\)), satisfying \(\chi(c),\xi(c)>0\) and \(b_{\xi\chi_{c}^{-1},L}\notin\overline{\{b_{\chi\chi_{c}^{-1},L}\}}\setminus b_{\chi\chi_c^{-1},L}\).
    \end{enumerate}
\end{lem}
\begin{rem}
    Before embarking on the proof, let us give the intuition for what the meaning of this lemma is.
    First of all, \(S\) is some set of exceptional \(b\in B(G)\) where we do not know \(\func{1}{b}{\chi}{\varphi}=0\) yet.
    Now towards the reason why each individual item is listed.
    \begin{enumerate}
        \item[(1)] This will always be obviously true in the cases we consider.
        \item[(2)] This ensures that the semi-orthogonal decomposition forces us to have a non-zero morphism
        \begin{equation*}
            i_{b*}^\ren  i_{b}^{\ren!}\oo(\chi,\varphi)*i_{1*}^{\ren}\pi\to\oo(\chi,\varphi)*i_{1*}^\ren\pi
        \end{equation*}
        for any \(b\in S\) that are closed in \(S\).
        Observe that in applications of this lemma we have \cref{thm: explict hecke ! stalk} for all \(b\notin S\), in this case we also need to check that \(b\nleq b_{\chi,L}\) for \(b\in S\), as we expect the unique non-vanishing stalk of \(\oo(\chi,\varphi)*i_{1*}^\ren\) to be \(b_{\chi,L}\), so this condition also ensures that the non-vanishing stalk does not interfere with the other \(b\in S\).
        \item[(3)] This ensures that for each \(b\in S\) and any \(\xi\) such that \(b=b_{\xi,L}\) we know
            \begin{equation*}
                \oo(\xi,\varphi)*i_{1*}^\ren\cong i^\ren_{b*}i_{b}^{\ren!}\oo(\xi,\varphi)*i_{1*}^\ren.
            \end{equation*}
            In addition, varying over such \(\xi\) this ensures that \(i_{b}^{\ren*}\oo(\xi,\varphi)*i_{1*}^\ren\rho\) for \(\rho\in\D(G(E))\) exhaust the category \(\D(G_b(E))^\wedge_{\eta_{G,b}\circ\varphi^w}\).
            Intuitively this category is exhausted by stalks of well-understood Hecke operators, and in addition the Hecke operators are indeed skyscraper sheaves.
        \item[(4)] This guarantees that for some \(1\leq c\leq r\) any morphism
           \[
            \oo(\xi\chi_c^{-1},\varphi)*i_{1*}^{\ren}\rho\to\oo(\chi\chi_c^{-1},\varphi)*i_{1*}^{\ren}\pi
        \]
        must vanish by considering the semi-orthogonal decomposition or by considering \(L\)-parameters.
    \end{enumerate}
    
    The motivation for (2') and (4') are analogous to (2) and (4), except that we consider \(b\in S\) that is open in \(S\).
    Since \(|\bun_G|\) is a poset, we can consider the opposite topological space \(|\bun_G|^\op\) determined by the fact they have identical points and opposite closure relations.
    Then Item (2') is really the condition \(b\notin\overline{S}\setminus S\) when we consider \(S\subset|\bun_G|^\op\).
    In the same vain Item (4') is dual to Item (4).
    Item (4') excludes nontrivial morphisms            
    \[
        \oo(\chi\chi_c^{-1},\varphi)*i_{1*}^{\ren}\pi\to\oo(\xi\chi_c^{-1},\varphi)*i_{1*}^{\ren}\rho.
    \]
    Observe that the arrow is reversed.
    Similar to the situation in (2), we will check condition (2') by checking that \cref{thm: explict hecke ! stalk} holds for all \(b\notin S\) and that \(b\ngeq b_{\chi,L}\) for \(b\in S\).
\end{rem}
\begin{proof}[Proof of \cref{lem: P true up to finite set}]
    We want to show that \(i_{b}^{\ren!}\oo(\chi,\varphi)*i_{1*}^{\ren}=0\). Assume we have \(i_{b}^{\ren!}\oo(\chi,\varphi)*i_{1*}^{\ren}\pi\neq 0\) for some \(\pi\in \D(G(E))\) and assume \(b\) is most special with this property in \(S\).
    We will prove the claim via induction on \(|\chi|\).
    The case \(|\chi|=0\) is clear, as then \(\chi=0\).
    Note that (2) implies that for \(b'\in\overline{\{b\}}\) we have \(i_{b'}^{\ren!}\oo(\chi,\varphi)*i_{1*}^{\ren}=0\).
    In this case, we obtain a non-trivial morphism 
    \begin{equation*}
        i_{b*}^\ren  i_{b}^{\ren!}\oo(\chi,\varphi)*i_{1*}^{\ren}\pi\to\oo(\chi,\varphi)*i_{1*}^\ren\pi.
    \end{equation*}
    We have \(\oo(\xi,\varphi)*i_{1*}^\ren\cong i^\ren_{b*}i_{b}^{\ren!}\oo(\xi,\varphi)*i_{1*}^\ren\) for all \(\xi\) such that \(b_{\xi,L}\in S\) by (3).
    Addtionally, using \cref{thm: explict hecke ! stalk}(3), we find a \(\rho\) and a \(\xi\) such that \(\oo(\xi,\varphi)*\rho\cong i_{b*}^\ren  i_{b}^{\ren!}\oo(\chi,\varphi)*i_{1*}^{\ren}\pi\) and thus we obtain a non-trivial morphism 
    \begin{equation*}
        \oo(\xi,\varphi)*i_{1*}^{\ren}\rho\to\oo(\chi,\varphi)*i_{1*}^{\ren}\pi .
    \end{equation*}
    Now choose \(c\) as in the assumptions, so we obtain a non-trivial morphism 
    \[
        \oo(\xi\chi_c^{-1},\varphi)*i_{1*}^{\ren}\rho\to\oo(\chi\chi_c^{-1},\varphi)*i_{1*}^{\ren}\pi.
    \]
    By Inductive Hypothesis \cref{inductive hypothesis}(1) we see that this gives a non-trivial morphism 
    \begin{equation*}
        i_{b_{\xi\chi_c^{-1}}*}^\ren i_{b_{\xi\chi_c^{-1}}}^{\ren!}\oo(\xi\chi_c^{-1},\varphi)*i_{1*}^{\ren}\rho\to i_{b_{\chi\chi_c^{-1}}*}^\ren i_{b_{\chi\chi_c^{-1}}}^{\ren!}\oo(\chi\chi_c^{-1},\varphi)*i_{1*}^{\ren}\pi .
    \end{equation*}
    However, using \((4)\) we see that this is not possible.
    In case that we have the stronger condition \(b_{\xi\chi_c^{-1},L}\notin \overline{\{b_{\chi\chi_{c}^{-1},L}\}}\), this follows from the semi-orthogonal decomposition on \(\D(\bun_G)\).
    If \(b_{\xi\chi_c^{-1},L}=b_{\chi\chi_{c}^{-1},L}\), then \(\xi\chi_c^{-1}\neq\chi\chi_c^{-1}\) as \(b_{\xi,L}\neq b_{\chi,L}\) so that \(\xi\neq\chi\) and such a map cannot exist by consideration of parameters. This means such a \(b\) does not exist. 
    The statement with assumptions (2') and (4') are proved analoguously, choosing \(b\) to be a most generic element of \(S\) such that \(\func{1}{b}{\chi}{\varphi}\neq 0\) instead to obtain a non-trivial morphism 
    \begin{equation*}
        \oo(\chi,\varphi)*i_{1*}^\ren\pi\to i_{b*}^\ren  i_{b}^{\ren*}\oo(\chi,\varphi)*i_{1*}^{\ren}\pi\cong i_{b!}^\ren i_{b}^{\ren!}\oo(\chi,\varphi)*i_{1*}^{\ren}\pi .
    \end{equation*}
\end{proof}
\begin{rem}\label{rem: condition 4''}
    It is clear from the proof that one can replace condition (4') with the more precise
    \begin{center}
        (4'') There is no non-trivial map     \(\oo(\chi\chi_c^{-1},\varphi)*i_{1*}^{\ren}\pi\to\oo(\xi\chi_c^{-1},\varphi)*i_{1*}^{\ren}\rho.\)
    \end{center}
    Observe that (4') implies (4'').
\end{rem}
Let us now start with the base case for Inductive Hypothesis \cref{inductive hypothesis}(1).
\begin{lem}\label{lem: base case chi}
    We have \cref{thm: explict hecke ! stalk} for \(\func{1}{b}{\chi_c}{\varphi}\) for all \(1\leq c\leq r\) and all \(b\).
\end{lem}
\begin{proof}
    In this case, we only need to consider the set of \(b\) such that \(1\xto{\std}b\).
    Since \(1\) is semistable this happens if \(\nu_b\leq\std^\sharp\).
    In this case the modification is always reducible. 
    So, by \cref{lem: reduce r} we are done.
\end{proof}
Now let us check the base case for Inductive Hypothesis \cref{inductive hypothesis}(2).
\begin{lem}
    If \(\varphi\) is a parameter for \(\GL_1\), then we have \cref{thm: explict hecke ! stalk} for \(\func{1}{b}{\chi}{\varphi}\) for all \(\chi\) and \(b\)
\end{lem}
\begin{proof}
    In this case \(\varphi\) is irreducible and then it is immediate from \cref{lem: hecke stalk irreducible case}.
    Alternatively, this also follows from \cite[Theorem 6.4.1.]{categorical-fargues-tori}.
\end{proof}
We continue with the inductive step in \cref{inductive hypothesis}(1)
The simplest case is \(r=1\), which in fact does not require any induction.
\begin{lem}\label{lem: r=1}
    For all \(\chi\), \(b\) and all irreducible \(\varphi\) we have \cref{thm: explict hecke ! stalk} for \(\func{1}{b}{\chi}{\varphi}\).
\end{lem}
\begin{proof}
    This is immediate from \cref{lem: hecke stalk irreducible case}.
\end{proof}
We continue with the discussion of the inductive step in \cref{inductive hypothesis}(1) in the case where \(\chi(i)>0\) for all \(i\in\{1,\dots,r\}\).
\begin{lem}\label{lem: chi greater 0}
    We have \cref{thm: explict hecke ! stalk} for \(\func{1}{b}{\chi}{\varphi}\) for all \(\chi\succeq 0\) such that \(|\chi|=s\) and \(\chi(i)>0\) for all \(i\).
\end{lem}
\begin{proof}
    We may assume \(r\defined\rk(Z(\widehat{L}))\geq 2\) since \(r=1\) is covered by \cref{lem: r=1}.

    Let us define \(\nmin(\chi)\defined \min\{n_c\in \{n_1,\dots,n_r\}|\chi(c)/n_c=\numin(b_{\chi,L})\}\).
    We need to compute \(\func{1}{b}{\chi}{\varphi}\).
    Take an index \(1\leq c\leq r\) minimizing \(\chi_c/n_c\).
    If there are many such indices choose one which minimizes \(n_c\).
    By assumption \(\numin(b_{\chi\chi_c^{-1},L})\geq 0\).
    Also observe that \(\rkmin(b_{\chi\chi_c^{-1},L})=\nmin(\chi)\).
    One can check that \(b_{\chi\chi_c^{-1},L}\) satisfies the condition of \cref{lem: all modification reducible condition} using \(r\geq 2\) and our choice of \(c\), so any modification \(b_{\chi\chi_c^{-1},L}\xto{\std}b\) is reducible or satisfies \(\rkmin(b)<\rkmin(b_{\chi\chi_c^{-1},L})\).

    We now do an induction on \(N\) in the following statement:
    \begin{center}
        We have \cref{thm: explict hecke ! stalk} for \(\func{1}{b}{\chi}{\varphi}\) for all \(\chi\succeq 0\) such \(\nmin(\chi)\leq N\), \(\chi(i)>0\) and \(|\chi|=s\).
    \end{center}
    The base case happens for those \(\chi\) such that \(\nmin(\chi)=\min\{n_1,\dots,n_r\}\).
    In this case all modifications \(b_{\chi\chi_c^{-1},L}\xto{\std}b\) such that \(b\in B(G)_L\) are reducible by \cref{lem: all modification reducible condition} and we conclude via \cref{lem: reduce r}.
    
    Now for the inductive step on \(N\) we can assume the following:
    \begin{center}
        (\(*\)): \cref{thm: explict hecke ! stalk} holds for \(\func{1}{b}{\xi}{\varphi}\) for any \(\xi\succeq 0\) such that \(\xi(i)>0\), \(|\xi|=s\) and \(\nmin(\xi)<N\).
    \end{center}
    For a \(\chi\) such that \(\nmin(\chi)=N\) we apply \cref{lem: P true up to finite set} to the set 
    \begin{equation*}
    S=\{b\in B(G)_L| \text{there is no reducible modification }b_{\chi\chi_c^{-1},L}\xto{\std}b\}.
    \end{equation*}
    \begin{enumerate}
        \item[(1)] This holds by assumption.
        \item[(2)] This holds by the fact that we have \cref{thm: explict hecke ! stalk}(2) for \(\func{1}{b}{\chi}{\varphi}\) for all \(b\notin S\).
        Observe that we implicitly use that \(b_{\chi,L}\notin \overline{S}\setminus S\) since we have \(\numax(b_{\chi\chi_c^{-1},L})=\numax(b_{\chi,L})\) (using \(r\geq 2\)), in particular we have \(b\nleq b_{\chi,L}\) for \(b\in S\).
        \item[(3)] This can be assumed to hold by (\(*\))\footnote{Recall that we cannot conclude by using Inductive Hypothesis \cref{inductive hypothesis}(1) since \(|\xi|=s\).}, as any \(b'\in S\) satisfies \(\rkmin(b')<\rkmin(b_{\chi\chi_c^{-1},L})\).
        This means if \(b'=b_{\xi,L}\), then \(|\xi|=|\chi|\), \(\xi(i)>0\) for all \(i\in\{1,\dots,r\}\) and we have \(\nmin(\xi)<\rkmin(b_{\chi\chi_c^{-1},L})=\nmin(\chi)=N\).
        \item[(4)] This holds by choosing \(1\leq d\leq r\) to be an index such that \(\xi(d)/n_d=\numin(b_{\xi,L})\).
        Then, since \(r\geq 2\) here, we have 
        \begin{equation*}
            \numax(b_{\xi\chi_d^{-1},L})=\numax(b)>\numax(b_{\chi,L})\geq\numax(b_{\chi\chi_d^{-1},L}),
        \end{equation*}
        in particular \(b_{\xi\chi_d^{-1},L}\nleq b_{\chi\chi_d^{-1},L}\).
    \end{enumerate}
\end{proof}
We are left with the cases where \(\chi(i)=0\) for some \(i\in\{1,\dots,r\}\).
Here we need to treat the case \(r=2\) separately.
\begin{lem}\label{lem: r=2}
    Assume \(r=2\).
    Then we obtain \cref{thm: explict hecke ! stalk} for \(\func{1}{b}{\chi}{\varphi}\) for all \(\chi\succeq 0\) such that \(\chi(i)=0\) for some index \(i\) and \(|\chi|=s\).
\end{lem}
\begin{proof}

    We need to compute \(\func{1}{b}{\chi}{\varphi}\).
    Let \(c\in\{1,2\}\) be an index such that \(\chi(c)>0\).
    We write \(c'\defined 3-c\), so that \(\{c,c'\}=\{1,2\}\).
    In particular \(\chi(c')=0\). 
    In this case \(\mathcal{E}_{\chi\chi_c^{-1},L}\cong\oo(\chi(c)-1|n_c)\oplus\oo^{n_{c'}}\).
    This defines a \(P\)-reduction (for \(P\) the standard parabolic for the Levi \(M\defined\GL_{n_{c}}\times\GL_{n_{c'}}\)) on \(\mathcal{E}_{b_{\chi\chi_c^{-1}}}\), so using \cite[Lemma 2.4]{P-reduction-modification} \(\mathcal{E}_b\) admits a \(P\)-reduction \(\mathcal{E}_{b,P}\) such that \(\oo(\chi(c)-1|n_c)\times\oo^{n_{c'}}\xto{\mu}\mathcal{E}_{b,P}\times^P M\), where \(\mu\) is a reduction of \(\std\) to \(M\). 
    
    There are two cases to consider, \(\mu=(\std,0)\) and \(\mu=(0,\std)\).
    
    \begin{enumerate}
        \item If \(\mu=(\std,0)\), then  we obtain that \(\mathcal{E}_{b,P}\times^P M\cong \mathcal{E}'\times\oo^{n_{c'}}\) with \(\oo(\chi(c)-1|n_c)\xto{\std}\mathcal{E}'\).
        Thus we have \(\mathcal{E}'\cong\oo(\lambda)^m\) for some \(m>0\) with \(\lambda=\chi(c)/n_c\).
        In this case \(b=b_{\chi,L}\), then \(b_{\chi\chi_c^{-1},L}\xto{\std}b_{\chi,L}\) is a reducible modification, so we may apply \cref{lem: reduce r} to reduce to \cref{lem: r=1}.
        
    \item If \(\mu=(0,\std)\), then the \(P\)-reduction \(\mathcal{E}_{b,P}\) satisfies 
        \begin{equation*}
            \mathcal{E}_{b,P}\times^P M=\oo(\chi(c)-1|n_c)\times\oo(1/n_{c'}).
        \end{equation*}
        If \(1/n_{c'}\leq(\chi(c)-1)/n_c\) we have \(\mathcal{E}_{b}\cong\oo(\chi(c)-1|n_c)\times\oo(1/n_{c'})\).
        In this case, \(b_{\chi\chi_c^{-1},L}\xto{\std}b_{\chi,L},\) is a reducible modification, so we reduce to \cref{lem: r=1} by \cref{lem: reduce r}.
        We are left with the case \(1/n_{c'}>(\chi(c)-1)/n_c\).
        It follows that \(\numax(b)\leq 1/n_{c'}\).

        We have now exhausted the analysis of cases where \(b_{\chi\chi_c^{-1},L}\xto{\std}b\) and \((b_{\chi\chi_c^{-1},L},b,\std)\) is reducible.
        We are left in the situation where \(\mathcal{E}_b=\oo(\xi(c')|n_{c'})\oplus\oo(\xi(c)|n_c)\) with \(\xi(c')/n_{c'}\geq \xi(c)/n_c>0\) and \(\xi(1)+\xi(2)=\chi(c)\) and the modification is not reducible.
        Since \(\numax(b)\leq 1/n_{c'}\), we see that \(\xi(c')=1\), then necessarily \(\xi(c)=\chi(c)-1\).
        We want to apply \cref{lem: P true up to finite set} to the singleton set \(S=\{b_{\xi,L}\}\).
        We now have three cases:
    \end{enumerate}
        \begin{enumerate}
            \item[(2.1)] First, consider when \(\xi(1)=\xi(2)=1\). 
            We want to check the conditions of \cref{lem: P true up to finite set} using condition (4'') discussed in \cref{rem: condition 4''}.
            \begin{enumerate}
                \item[(1)] This holds by assumption.
                \item[(2')] By the disussion above we have \cref{thm: explict hecke ! stalk}(2) for any \(b\notin S\) and we use that \(b_{\xi,L}\ngeq b_{\chi,L}\) since \(\numin(b_{\xi,L})>0=\numin(b_{\chi,L})\).
                \item[(3)] This is \cref{lem: chi greater 0}.
                \item[(4'')] This is \cref{lem: case of (1 -1)}.
            \end{enumerate}

            \item[(2.2)] Now assume \(\xi\neq(1,1)\).
            We want to check the conditions of \cref{lem: P true up to finite set}.
            Conditions (1), (2') and (3) hold for the same reasons as in the case when \(\xi=(1,1)\), and we can use the geometric condition (4').
            Condition (4') follows from 
            \begin{equation*}
                \numin(b_{\xi\chi_c^{-1}})=\frac{\chi(c)-2}{n_c}>0=\numin(b_{\chi\chi_c^{-1},L}),
            \end{equation*}
            using the fact that \(\xi\neq (1,1)\).
            
    \end{enumerate}
\end{proof}
In the proof of the preceeding theorem we needed the following lemma.
\begin{lem}\label{lem: case of (1 -1)}
    Assume that \(L=\GL_{n_1}\times\GL_{n_2}\), then 
    \begin{equation*}
        \rhom(\oo((1,0),\varphi)*i_{1*}^\ren\pi,\oo((0,1),\varphi)*i_{1*}^\ren\rho)=0,
    \end{equation*}
    similarly \(\rhom(\oo((0,1),\varphi)*i_{1*}^\ren\pi,\oo((1,0),\varphi)*i_{1*}^\ren\rho)=0\) for \(\pi,\rho\in\D(G(E))^\wedge_{\eta_{G,1}\circ\varphi}\).
\end{lem}
\begin{proof}
    It is easy to see that this is equivalent to the following claim:
    \begin{center}
        (\(*\)) Let \(\chi=(1,-1)\),
        then \(i_1^{\ren!}\oo(\chi,\varphi)*i_{1*}^\ren=0\).
        The same holds for \(\chi=(-1,1)\).
    \end{center}
    Let \(b\) correspond to \(\oo(1/n_1)\oplus\oo^{n_2}\).
    We have \cref{thm: explict hecke ! stalk} for \(\func{b}{1}{(0,-1)}{\varphi}\) by applying \cref{cor: hodge-newton renormalized reducible reduction}, as in the proof of \cref{lem: reduce r}.
    We also have \cref{thm: explict hecke ! stalk} \(\func{1}{b}{(1,0)}{\varphi}\) for any \(b\) by \cref{lem: base case chi}.
    In particular it follows that
    \begin{equation*}
        \oo((1,0),\varphi)*i_{1*}^\ren=i_{1*}^\ren\func{1}{b}{(1,0)}{\varphi}.
    \end{equation*}
    Writing
    \begin{align*}
        \func{1}{1}{(1,-1)}{\varphi}&=i_{1}^{\ren!}\oo((0,-1),\varphi)*\oo((1,0),\varphi)*i_{1*}^\ren\\
        &=\func{b}{1}{(0,-1)}{\varphi}\func{b}{1}{(1,0)}{\varphi}
    \end{align*}
    one easily concludes (\(*\)).
    We obtain the case for \(\chi=(-1,1)\) by switching the roles of \(n_1\) and \(n_2\).
\end{proof}
We can finally discuss the case \(r\geq 3\).
\begin{lem}\label{lem: r=3}
    Assume \(r\geq 3\).
    Then we obtain \cref{thm: explict hecke ! stalk} for \(\func{1}{b}{\chi}{\varphi}\) for all \(\chi\succeq 0\) such that \(\chi(i)=0\) for some index \(i\) and \(|\chi|=s\).
\end{lem}
\begin{proof}

    We need to compute \(\func{1}{b}{\chi}{\varphi}\).
    Let \(1\leq c\leq r\) be an index such that \(\chi(c)>0\) and \(\chi(c)/n_c\) is minimal among the positive slopes of \(b_{\chi,L}\).
    Consider the \(P\)-reduction coming from \(\mathcal{E}_{b_{\chi\chi_c^{-1},L}}\cong\mathcal{E}\oplus\oo^m\) where \(m=\sum_{\chi\chi_c^{-1}(i)=0}n_i\) and \(\numin(\mathcal{E})>0\).
    Again, we will have to consider \(\mathcal{E}_{b,P}\) such  that \(\mathcal{E}\times\oo^m\xto{\mu}\mathcal{E}_{b,P}\times^PM\), where \(\mu\) is a reduction of \(\std\) to \(M\).
    
    The two cases to consider are \(\mu=(\std,0)\) and \(\mu=(0,\std)\).
    \begin{enumerate}
        \item  If \(\mu=(\std,0)\) we obtain that \(\mathcal{E}_b=\mathcal{E}'\oplus\oo^m\) for some \(\mathcal{E}'\) such that \(\mathcal{E}\xto{\std}\mathcal{E}'\).
        This is because in this situation \(\numin(\mathcal{E})\geq\numin(\mathcal{E})\), and \(\numin(\mathcal{E}')\geq0\), so the extension problem splits.
        In this case, \(b_{\chi\chi_c^{-1},L}\xto{\std}b\) is a reducible modification, so we may apply \cref{lem: reduce r}.
        
        \item If \(\mu=(0,\std)\), we obtain that \(\mathcal{E}_{b,P}\times^P M\cong \mathcal{E}\times\mathcal{E}'\) with \(\oo^m\xto{\std}\mathcal{E}'\).
        Thus we have \(\mathcal{E}'\cong\mathcal{E}''\oplus\oo(\lambda)^l\) for some \(l>0\) with \(\lambda=0\) and \(\numin(\mathcal{E}'')>0\) or \(\mathcal{E}''=0\).
    \end{enumerate}
        \begin{enumerate}
            \item[(2.1)] The case \(\mathcal{E}''=0\) will be discussed further below.
            
            \item[(2.2)] If \(\lambda=0\), then \(b_{\chi\chi_c^{-1},L}\xto{\std}b\) is a reducible modification
            and we may apply \cref{lem: reduce r}.
    \end{enumerate}
    Continuing from the above one obtains that \(\lambda=1/m\) (and \(l=1\)).  
    We obtain that \cref{thm: explict hecke ! stalk} for \(\func{1}{b}{\chi}{\varphi}\) holds for all \(b\) except possibly for those \(b\) such that \(b_{\chi\chi_c^{-1},L}\xto{\mu}b\), and additionally \(\mathcal{E}_b\) admits a \(P\)-reduction \(\mathcal{E}_{b,P}\) such that \(\mathcal{E}_{b,P}\times^P M=\mathcal{E}\times\oo(\lambda)\), this implies that \(\numin(b)>0\).

    \begin{itemize}
            \item We first treat the case where \(\chi=\chi_c^a\) for some \(a\geq 1\).
            In this case, \(\mathcal{E}=\oo(a-1|n_c)\).
            If \(\frac{a-1}{n_c}\geq\lambda\), then \(\mathcal{E}_b=\oo(a-1|n_c)\oplus\oo(\lambda)\), which is reducible so we may apply \cref{cor: hodge-newton renormalized reducible reduction}.
            Assume that \(\frac{a-1}{n_c}<\lambda\), then \(\lambda\geq\numax(b)\) and we see that \(\lambda=\frac{1}{n-n_c}\). Since \(b\in \im(B(L)_{\mathrm{basic}}\to B(G))\) and \(r\geq 3\), there are at least two different indices \(i_1\neq i_2\) such that \(\chi(i_1)=\chi(i_2)=0\).
            Observe that \(n_{i_1}<n-n_c\). Then \(\lambda\geq\numax(b)\) implies that \(\numin(b)=0\), since otherwise we would have a slope of at least \(\frac{1}{n_{i_1}}>\lambda\). 
            Then the modification is reducible, so we conclude by \cref{lem: reduce r}
            \item We are left with the case that there are two different indices \(i_1\neq i_2\) such that \(\chi(i_1),\chi(i_2)>0\).
            In this case, we see that \(\numax(\mathcal{E})=\numax(\mathcal{E}_{b_{\chi\chi_c^{-1},L}})\) and also that \(\numax(b)\geq\numax(b_{\chi\chi_c^{-1},L})\).
            If the modification is reducible, then we conclude by \cref{lem: reduce r}.
            We have now exhausted the simple analysis of cases where \(b_{\chi\chi_c^{-1},L}\xto{\std}b\) is reducible, and we are left with the non-reducible modifications.
                
            Consider the set \(S\) of non-reducible modifications, explicitly this is
            \begin{equation*}
                S=\{b\in B(G)|\numin(b)>0,\numax(b)>\numax(b_{\chi,L}),b_{\chi\chi_c^{-1},L}\xto{\std}b\}.
            \end{equation*}
            This set is finite.
            We want to apply \cref{lem: P true up to finite set}.
            Let us check the conditions.
            \begin{enumerate}
                \item[(1)] This holds by assumption.
                \item[(2)] We have \cref{thm: explict hecke ! stalk}(2) for \(\func{1}{b}{\chi}{\varphi}\) for all \(b\notin S\).
                We also use that \(b\nleq b_{\chi,L}\) for \(b\in S\) since we have \(\numax(b)>\numax(b_{\chi,L})\) for \(b\in S\).
                \item[(3)] This is \cref{lem: chi greater 0}.
                \item[(4)] We have \(\numax(b_{\xi\chi_{i_1}^{-1},L})=\numax(b_{\xi,L})\) or \(\numax(b_{\xi\chi_{i_2}^{-1},L})=\numax(b_{\xi,L})\).
                Up to renaming \(i_1\) and \(i_2\) we can assume the first one is the case so that we have 
                \begin{equation*}
                    \numax(b_{\xi\chi_{i_1}^{-1},L})=\numax(b)>\numax(b_{\chi\chi_c^{-1},L})=\numax(b_{\chi,L})\geq\numax(b_{\chi\chi_{i_1}^{-1},L}),
                \end{equation*}
                using \(r\geq 3\) and \(\chi\neq\chi_c^a\) to conclude that \(\numax(b_{\chi\chi_c^{-1},L})=\numax(b_{\chi,L})\).            
                In particular \(b_{\xi\chi_{i_1}^{-1},L}\nleq b_{\chi\chi_{i_1}^{-1},L}\).
            \end{enumerate}
        \end{itemize}

\end{proof}
\begin{proof}[Proof of \cref{thm: explict hecke ! stalk}]
    This is now immediate from \cref{lem: non-negative slopes suffice}, \cref{lem: r=1}, \cref{lem: chi greater 0} and \cref{lem: r=3}.

\end{proof}
\begin{cor}
    The functor \(\func{b}{b'}{\chi}{\varphi}\) is \(t\)-exact.
\end{cor}
\begin{proof}
    As in the proof of \cref{lem: non-negative slopes suffice}, it suffices to check it for \(\chi\succeq 0\) and \(b'\) without negative slopes and \(b=1\), using the identification \(T_{\det}\cong q^*\) where \(p\from \bun_G\to\bun_{G_b}\cong\bun_G\) with \(b=\oo(1)^n\). 
    The first map is the isomorphism \cite[Corollary III.4.3.]{geometrization} and the second map comes from the isomorphism \(G\cong G_b\).
    This functor is \(t\)-exact.

    The only non-trivial case to consider is the \(t\)-exactness of \(\func{1}{b_{\chi,L}}{\chi}{\varphi}\), when \(\chi\succeq 0\).
    We now induct on \(|\chi|\).
    Choose a \(c\in\{1,\dots,r\}\) such that \(\chi(c)>0\).
    Using \cref{thm: explict hecke ! stalk} may rewrite this as 
    \begin{equation*}
        \func{1}{b_{\chi,L}}{\chi}{\varphi}=\func{b_{\chi\chi_c^{-1},L}}{b_{\chi,L}}{\chi_c}{\varphi}\func{1}{b_{\chi\chi_c^{-1},L}}{\chi\chi_c^{-1}}{\varphi}.
    \end{equation*}
    By induction \(\func{1}{b_{\chi\chi_c^{-1},L}}{\chi\chi_c^{-1}}{\varphi}\) is a \(t\)-exact equivalence of categories, so it suffices to check that \(\func{b_{\chi\chi_c^{-1},L}}{b_{\chi,L}}{\chi_c}{\varphi}\) is \(t\)-exact.
    By \cref{lem: isotypic decomposition Langlands-Shahidi type}, this is a direct summand of \(i_{b_{\chi}}^{\ren!}T_{\std}*i_{b_{\chi\chi_c^{-1},L}*}^\ren\).
    If we choose \(c\) to be an index such that \(\chi(c)/n_c\) is a minimal positive slope of \(b_{\chi,L}\), then \((b_{\chi\chi_c^{-1},L},b_{\chi,L},\std)\) is reducible.
    Recall that both normalized restriction and induction are \(t\)-exact, so using \cref{cor: hodge-newton renormalized reducible reduction}, and an induction on \(r\defined\rk(Z(\widehat{L}))\) we can conclude.
    Here the base case \(r=1\) is given by \cref{cor: t-exact irreducible case}.
\end{proof}
\begin{cor}\label{cor: explict hecke * stalk}
    \cref{thm: explict hecke ! stalk} also holds for \(i_{b'}^{\ren*}\oo(\chi,\varphi)*i_{b!}^{\ren}\).
    In addition this functor is \(t\)-exact.
\end{cor}
\begin{proof}
    Since \(i_{b'}^{\ren*}\oo(\chi,\varphi)*i_{b!}^{\ren}\) preserves compact objects and the inclusion 
    \[\D(G_b(E))^\wedge_{\eta_{G,b}\circ\varphi}\subset\D(G_b(E))\] preserves compact objects and sends compact objects to compact and ULA objects by \cref{lem: compact objects in localization compact and ULA}, we can deduce the claim using Verdier duality.
    For \(t\)-exactness, note filtered colimits are \(t\)-exact, so it suffices to check \(t\)-exactness on compact objects.
\end{proof}
\begin{rem}
    Using \cref{thm: splitting semi-orthogonal} one can show that in fact we have a natural isomorphism \(i_{b'}^{\ren*}\oo(\chi,\varphi)*i_{b!}^{\ren}\cong\func{b}{b'}{\chi}{\varphi}\) when restricted to \(\D(G_b(E))^\wedge_\varphi\).
\end{rem}
\begin{rem}
    Let us see how this arguments works out in practice.
    We take as our Levi the group \(L=\GL_2\times\GL_3\times\GL_3\).
    We will fix some parameter \(\varphi\) of Langlands-Shahidi type with cuspidal support \(\widehat{L}\) and write \(\varphi=\varphi_1\times\varphi_2\times\varphi_3\) for the reduction through \(\widehat{L}\).
    We want to unterstand \(\oo((2,2,2))*i_{1*}^\ren\).
    \begin{figure}[h]
        \begin{tikzpicture}
            \draw (0,0) -- (2,2) -- (8,6);
            \draw[orange] (5,4) -- (8,5);
            \draw[red] (5,4) -- (6,5);
            \draw[blue] (6,5) -- (8,6); 
        \end{tikzpicture}
        \caption{\label{fig 1}\(b_{(2,2,1),L}\) with the smallest slope marked orange and a modification of the smallest slope}
    \end{figure}
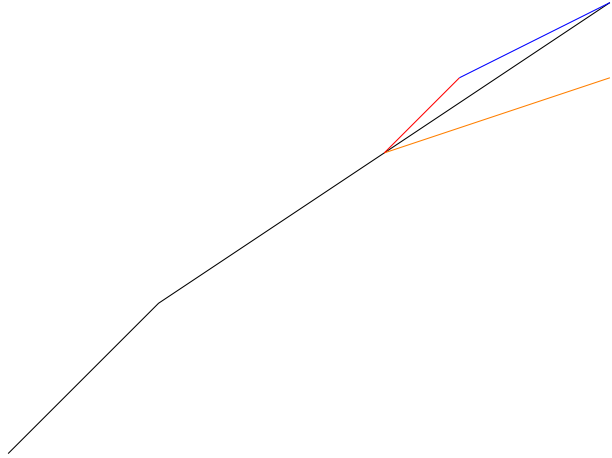
    Here the black line is the Newton polygon for \(b_{(2,2,2),L}\).
    We want to take some index \(c\) responsible for a smallest slope and decrase the value by \(1\), here we can take \(c=3\).
    Then we write \(\oo((2,2,2),\varphi)*i_{1*}^\ren=\oo((0,0,1),\varphi)*\oo((2,2,1),\varphi)*i_{1*}^\ren\).
    By induction we know that \(\oo((2,2,1),\varphi)*i_{1*}^\ren\) is a \(*\)-extension from \(b_{(2,2,1),L}\).
    By construction this has a Newton polygon whose smallest slope comes from the index \(c\) we have chosen, this see \cref{fig 1}.
    Since we fully understand \(\oo((2,2,1),\varphi)*i_{1*}^\ren\), we are left to compute \(\oo((0,0,1),\varphi)*i_{b_{(2,2,1),L}*}^\ren\).
    Since \(\oo((0,0,1),\varphi)\) splits off from \(T_{\std}\) when restricted to \(\D(\bun_G)^\wedge_\varphi\), we want to understand the set of modifications \(b_{(2,2,1),L}\xto{\std}b\).
    We filter \(\mathcal{E}_{b_{(2,2,1),L}}\) by the black part and the orange part.
    Then \(\mathcal{E}_b\) admits a 2-step filtration such that the first graded piece is a modification of the black part bouned by \(\std\) and the second graded piece is given by the orange part or the black parts stay the same, and the orange part gets modified.
    The latter case is shown in the graphic.
    We have marked the biggest slope of the vector bundle modifying the orange part with red and the smallest slope with blue.

    When the black parts get modified, we are in the reducible situation, which is understood by \cref{lem: reduce r}, the same thing happens if either the red slope is not bigger than the biggest black slope, or the blue slope is equal to the orange slope.

    In this example we see that \(b=b_{(1,2,3),L}\), and the situation looks like \cref{fig 2}.
    \begin{figure}[h]
        \begin{tikzpicture}
            \draw (0,0) -- (2,2) -- (5,4) -- (8,5);
            \draw (0,0) -- (3,3) -- (6,5) -- (8,6);
            \draw[dotted] (2,2) -- (2,0);
        \end{tikzpicture}
        \caption{\label{fig 2}\(b_{(1,2,3),L}\) and \(b_{(2,2,1),L}\), the dotted line shows that the modification is \(\omega\)-Hodge-Newton reducible for \(\GL_2\times\GL_6\)}
    \end{figure}
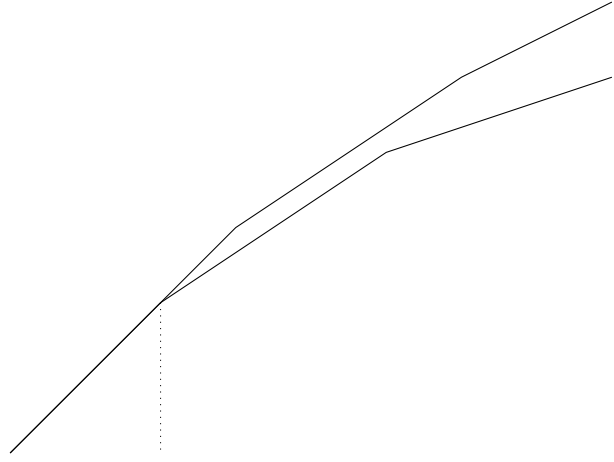
    The application of \cref{lem: reduce r} now works like this.
    We want to show that 
    \(i_{b_{(1,2,3),L}}^{\ren!}\oo((2,2,2),\varphi)*i_{1*}^\ren\pi=0\) 
    for any \(\pi\in\D(G(E))^\wedge_{\varphi}\).
    Then by writing \(\oo((2,2,2),\varphi)*i_{1*}^\ren=\oo((0,0,1),\varphi)*\oo((2,2,1),\varphi)*i_{1*}^\ren\) and induction on \(|\chi|\) it suffices to show that \(i_{b_{(1,2,3),L}}^{\ren!}\oo((0,0,1),\varphi)*i_{b_{(2,2,1),L}*}^\ren\pi=0\) for any \(\pi\in\D(\GL_2(E)\times D^\times_{2/3}\times D^\times_{1/3})^\wedge_{\varphi_1\times\varphi_2\times\varphi_3}\).
    Since the triple \((b_{(2,2,1),L},b_{(1,2,3),L},\std)\) is \((\omega)\)-Hodge-Newton reducible, we can compute 
    \begin{equation*}
        i_{b_{(1,2,3),L}}^{\ren!}T_{\std}i_{b_{(2,2,1),L}*}^\ren\pi=\nind^{\GL_3\times D^\times_{1/3}\times D^\times_{1/2}}_{\GL_2\times\GL_1\times D^\times_{1/3}\times D^\times_{1/2}}i_{((2),(1,1,2),L')}^{\ren!}T_{(0,\std)}i_{b_{((2,2,1)),L}*}^\ren\pi.
    \end{equation*}
    On the right hand-side we are working with \(\bun_M\) with \(M=\GL_2\times\GL_6\), in particular we consider \(b_{((2,2,1)),L}\in B(M)\) by writing \(\mathcal{E}_{b_{((2,2,1)),L}}=\oo(1)^{\oplus 2}\times\oo(2/3)\oplus\oo(1/3)\).
    Here \(L'=\GL_2\times(\GL_1\times\GL_2\times \GL_3)\), this is a Levi of \(M\) and \((0,\std)\) refers to the cocharacter \((0,0,1,0,\dots,0)\).
    We are reduced to checking that the \(\varphi_1\times\varphi_2\times\varphi_3\)-isotypic part of \(i_{b_{((2),(1,1,2),L')}}^{\ren!}T_{(0,\std)}i_{b_{((2,2,1)),L}*}^\ren\pi\) vanishes.\footnote{In this case one can shortcut the argument by observing that anything on \(b_{((2),(1,1,2),L')}\) must have a parameter that reduces to \(\widehat{L'}\), contradicting that \(\varphi\) has cuspidal support \(\widehat{L}\).}
    For our claimes it suffices to compute this for those \(\pi\) of the form \(\pi=\pi_1\boxtimes\pi_2\) with \(\pi_1\in\D(\GL_2(E))^\wedge_{\varphi_1}\) and \(\pi_2\in\D(D^\times_{2/3}\times D^\times_{1/3})^\wedge_{\varphi_2\times\varphi_3}\).
    Since the formation of Hecke operators is compatible with products of groups, we compute
    \begin{equation*}
        i_{b_{((2),(1,1,2),L')}}^{\ren!}T_{(0,\std)}i_{b_{((2,2,1)),L}*}^\ren\pi_1\boxtimes\pi_2=\pi_1\boxtimes i_{\oo(1)\oplus\oo(2/3)\oplus\oo(1/2)}^{\ren!}T_{\std}i_{\oo(2/3)^{\oplus 2}*}^\ren\pi_2.
    \end{equation*}
    Since \(i_{\oo(1)\oplus\oo(2/3)\oplus\oo(1/2)}^{\ren!}T_{\std}i_{\oo(2/3)^{\oplus 2}*}^\ren\pi_2\) is a Hecke operator on \(\bun_{\GL_6}\) and \(\pi_2\) has a Langlands-Shahidi type parameter we can by induction assume that we can compute this completely, in this case we see that the \(\varphi_2\times\varphi_3\) isotypic component vanishes, so the entire term vanishes.

    The more tricky situation happens when the modification is not reducible.
    Let us take \(L=\GL_3\times\GL_4\), \(\varphi=\varphi_1\times\varphi_2\) as above and \(\chi=(2,2)\).
    We write 
    \begin{equation*}
        \oo((2,2))*i_{1*}^\ren=\oo((0,1))*\oo((2,1))*i_{1*}
    \end{equation*}
    and by induction are left with understanding \(\oo((0,1))*i_{b_{(2,1),L}*}^\ren\).
    We do the same analysis as above, however, we now also have a non-reducible modification, see \cref{fig 3}.
    \begin{figure}[h]
        \begin{tikzpicture}
            \draw (0,0) -- (3,2) -- (7,4);
            \draw[orange] (3,2) -- (7,3);
            \draw[red] (3,2) -- (4,3);
            \draw[dashed,blue] (4,3) -- (7,4);
            \draw[dashed] (0,0) -- (4,3);
        \end{tikzpicture}
        \caption{\label{fig 3}Analysis of a modification of \(b_{(2,1),L}\) as above. Observe that the red slope is bigger than the biggest black slope, and the blue dashed slope is bigger than the orange one. The dashed line is \(b_{(3,1),L}\), and this is non-reduciblerelative to \(b_{(2,1),L}\) and we see that it is neither Hodge-Newton nor \(\omega\)-Hodge-Newton reducible.}
    \end{figure}
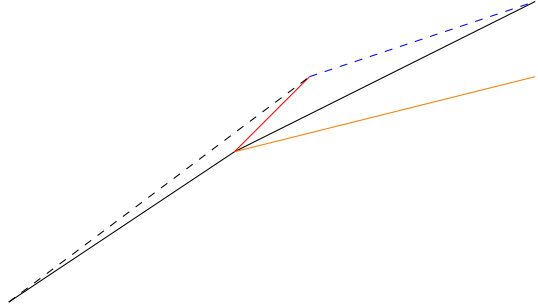
    We want to show that \(\func{b_{(2,1),L}}{b_{(1,3),L}}{(0,1)}{\varphi}\pi\) vanishes for \(\pi\in\D(D_{2/3}^\times\times D_{1/4}^\times)^\wedge_{\varphi_1\times\varphi_2}\).
    We see that no application of (\(\omega\)-)Hodge-Newton reducibilty can help us as above.
    Observe however for such a situation to occur that the blue line is necessarily shorter than the orange line.
    It follows that this situation does not occur for those \(\chi\) such that \(\chi(1)/3\leq\chi(2)/4\).
    In particular we can say this does not occur for \(\chi=(1,3)\).
    In this case, by induction we may compute that \(\oo((1,3),\varphi)*i_{1*}^\ren\) is always a \(*\)-extension from \(b_{(1,3),L}\).
    We want to apply \cref{lem: P true up to finite set} to \(S=\{b_{(1,3),L}\}\).
    Specifically, if \(\func{b_{(2,1),L}}{b_{(1,3),L}}{(0,1)}{\varphi}\pi\) would not vanish we would obtain non-zero representations \(\pi,\rho\in \D(G(E))\) and a non-trivial morphism 
    \begin{equation*}
        \oo((1,3),\varphi)*i_{1*}^\ren\rho\to\oo((2,2),\varphi)*i_{1*}^\ren\pi.
    \end{equation*}
    This would yield a non-trivial morphism 
    \begin{equation*}
        \oo((0,3),\varphi)*i_{1*}^\ren\pi\to\oo((1,2),\varphi)*i_{1*}^\ren\rho.
    \end{equation*}
    By induction the left hand side is a \(*\)-extension from \(b_{(0,3),L}\) and the right hand side is a \(*\)-extension \(b_{(1,2),L}\), however this is impossible by considering the semi-orthogonal decomposition and considering the closure relations between strata.
\end{rem}

%% file: proof-of-categorical-conjecture.tex
\section{Proof of the conjecture}\label{sec: proof}
\begin{lem}\label{lem: decomposition D(Bun_G) at Langlands-Shahidi type parameter}
    Let \(\varphi\) be a parameter of Langlands-Shahidi type with cuspidal support \(\widehat{L}\) with coefficients in an algebraically closed \(\bbZ_\ell\)-field \(k\).
    We have a direct sum decomposition
    \begin{equation*}
        (\D(\bun_G,k)^\wedge_{\varphi})^\omega\simeq\bigoplus_{\chi\in X^*(S_{\varphi})}(\D(\bun_G,k)^\wedge_{\varphi})^\omega_\chi.
    \end{equation*}
    Here \((\D(\bun_G,k)^\wedge_{\varphi})^\omega_\chi=(\D(\bun_G^b,k)^\wedge_{\eta^\alg_{G,b_{\chi,L}}\circ\varphi^{w_\chi}})^\omega_\chi\) (embedded via \(i_{b!}\)), where \((b_{\chi,L},w_{\chi})\) is the image of \(\chi\) under 
    \begin{equation*}
        X^*(S_{\varphi})=X^*(Z(\widehat{L}))\cong B(L)_{\mathrm{basic}}\subset B(G)\times W_G
    \end{equation*}
    constructed in \cite[Lemma 4.23.]{imaihamann}, where we identify $W_{L,b}$ with a subset of $W_G$ as explained before \cite[Lemma 4.20.]{imaihamann}.
\end{lem}
\begin{proof}
    By \cref{lem: support sheaf with L-parameter of Langlands-Shahidi type}, we see that \((\D(\bun_G,k)^\wedge_{\varphi})^\omega\) is generated by \((\D(\bun_G,k)^\wedge_{\varphi})^\omega_\chi\), so we need to show there are no non-trivial morphisms between objects in these categories for different \(\chi\).
    Let \(\pi\in(\D(\bun_G,k)^\wedge_{\varphi})^\omega_\chi\) and \(\pi'\in(\D(\bun_G,k)^\wedge_{\varphi})^\omega_{\chi'}\) for \(\chi\neq\chi'\).
    We want to show that \(\Hom(\pi,\pi')=0\).
    If \(b_{\chi,L}=b_{\chi',L}\) (so then \(w_{\chi}\neq w_{\chi'}\)), this is clear by consideration of \(L\)-parameters.
    Consider now the case \(b_{\chi,L}\neq b_{\chi',L}\).
    Using \cref{cor: explict hecke * stalk}, we can apply \(\oo(\chi'^{-1})*-\) to assume that \(\chi'=0\) and that \(b_{\chi,L}\neq 1\).
    In this case, \(\Hom(\pi,\pi')=0\) due to the semi-orthogonal decomposition on \(\D(\bun_G)\).
\end{proof}
\begin{thm}\label{thm: categorical equivalence}
    Let \(\Lambda\) be a \(\bbZ_\ell[\sqrt{q},\mu_{p^\infty}]\)-algebra.
    Let \(\LSt\) denote the open subscheme consisting of parameters that are of Langlands-Shahidi type with cuspidal support \(\widehat{L}\).
    Let \(p\from\locsys_{\widehat{G}}\to\locsys_{\widehat{G}}^\coarse\) be the natural map.
    We have a \(\Ind\perf(p^{-1}(\LSt))\)-linear equivalence 
    \begin{equation*}
        \Ind\coh_{\nilp}^\qc(p^{-1}(\LSt))\simeq \D(\bun_G)\otimes_{\Ind\perf(\locsys_{\widehat{G}}^\coarse)}\Ind\perf(\LSt)
    \end{equation*}
    sending \(\oo\) to \(\mathcal{W}^\LSt\).
\end{thm}
\begin{proof}
    We only need to check the conditions of \cref{lem: general equivalence criterion}.
    That \(\nilp|_{p^{-1}(\LSt)}=\{0\}\) is \cref{lem: nilpotent singular support Langlands-Shahidi}.
    That \(p^{-1}(\LSt)\times_{\LSt}\LSt^{\wedge}_{\varphi}\simeq N^\wedge_{0}/S_{\varphi}\), where \(N\) is an affine scheme with \(0\in N\) and it admits a trivial \(S_{\varphi}\)-action for any \(\varphi\in U(k)\) for \(k\) an algebraically closed \(\Lambda\)-field is \cref{cor: localize at langlands-shahidi fiber product}.
    The direct sum decomposition 
    \begin{equation*}
        (\D(\bun_G,k)^\wedge_{\varphi})^\omega\simeq\bigoplus_{\chi\in X^*(S_{\varphi})}(\D(\bun_G,k)^\wedge_{\varphi})^\omega_\chi
    \end{equation*}
    with the required properties is \cref{cor: explict hecke * stalk} and \cref{lem: decomposition D(Bun_G) at Langlands-Shahidi type parameter}.
    The genericity statement is \cref{lem: Langlands-Shahidi type generic}.
\end{proof}
\begin{rem}\label{rem: main theorem works over zl}
    Using \cite[Proposition 5.2.]{locallanglandsinfamilies2} we can descend \(\mathcal{W}\) to \(\bbZ_\ell\).
    Together with \cref{rem: remove sqrt q} it follows that one can formulate the categorical equivalence already over \(\bbZ_\ell\).
    One can check that \cref{lem: general equivalence criterion} in fact already works over a \(\bbZ_\ell\)-algebra, so the previous theorem already holds over any \(\bbZ_\ell\)-algebra. 
\end{rem}

%% file: t-exactness.tex
\section{t-exactness}\label{sec: t-exactness}
Let \(k\) be an algebraically closed \(\bbZ_\ell\)-field.
\begin{lem}\label{lem: t-exact Langlands-Shahidi type}
    Let \(\varphi\) be a parameter of Langlands-Shahidi type valued in \(k\).
    Then \(\oo(\chi,\varphi)*-\) is \(t\)-exact on \(\D(\bun_G,k)^\wedge_\varphi\).
\end{lem}
\begin{proof}
    This is immediate from the \(t\)-exactness claim in \cref{thm: explict hecke ! stalk} and \cref{cor: explict hecke * stalk}.
\end{proof}
\begin{defn}
    For any open substack \(U\subset\locsys_{\widehat{G}}\) we define a \(t\)-structure \(\ind\perf^\qc(U)\) by imposing that the connective part be generated by \(f^*V\) with \(f\from\locsys_{\widehat{G}}\to B\widehat{G}\) being the canonical map and \(V\in\perf(B\widehat{G})\cap\qcoh(B\widehat{G})_{\geq 0}\).
\end{defn}
\begin{thm}\label{thm: t-exact}
    Let \(\LSt\) denote the open subscheme consisting of parameters that are of Langlands-Shahidi type with cuspidal support \(\widehat{L}\), and \(p^{-1}(\LSt)\) the corresponding open stack of \(\locsys_{\widehat{G}}\).
    The \(\perf(\LSt)\)-linear equivalence 
    \begin{equation*}
        \bbL_G^\LSt\from\ind\coh_{\nilp}^\qc(p^{-1}(\LSt))\simeq \D^{\LSt}(\bun_G,k)
    \end{equation*}
    obtained by \(\ind\)-extending \cref{thm: categorical equivalence} is \(t\)-exact.
    The right-hand side has the \(t\)-structure given by equipping \(\D(\bun_G,k)\) with the perverse \(t\)-structure and equipping \(\ind\coh_{\nilp}^\qc(p^{-1}(\LSt))\) with the \(t\)-structure coming via transfer from \(\ind\coh_{\nilp}^\qc(p^{-1}(\LSt))\simeq\ind\perf^\qc(p^{-1}(\LSt))\).
\end{thm}
\begin{proof}
    By definition of the \(t\)-structure on \(\ind\coh_{\nilp}^\qc(p^{-1}(\LSt))\) it suffices to check that Hecke operators \(T_V\) are \(t\)-exact on \(\mathcal{C}\defined\D^{\LSt}(\bun_G,k)\).
    In fact, it suffices to check that \(T_V\) are right exact for all \(V\in\rep(\widehat{G})\), since then, the adjoint \(T_{V^\vee}\) is left \(t\)-exact for all \(V\in\rep(\widehat{G})\) and \(V=V^{\vee\vee}\).
    For this, consider the following diagram:
    \begin{equation*}
        \begin{tikzcd}
            \mathcal{C} \arrow[r, "T_V"] \arrow[d]                                      & \mathcal{C} \arrow[d]                                      \\
            \prod_\varphi(\mathcal{C}\otimes_k k_{\varphi})^\wedge_\varphi \arrow[r, "T_V"] & \prod_\varphi(\mathcal{C}\otimes_k k_\varphi)^\wedge_\varphi
            \end{tikzcd}
    \end{equation*}
    where the product runs over some set of geometric points \(\varphi\) of \(\LSt\) covering all points of \(\LSt\) and \(k_\varphi\) is the field of definition of \(\varphi\).
    The vertical functor is the product of the localization functors and are right \(t\)-exact (being left adjoints of the \(t\)-exact inclusion functors) and the functor to the product is conservative (which is an easy consequence of \cref{thm: categorical equivalence}).
    The bottom arrow is \(t\)-exact by \cref{lem: t-exact Langlands-Shahidi type}.
\end{proof}
\begin{rem}\label{rem: t-exact any field and flat zl algebra}
    Arguing as in \cref{lem: reduction to algebraic closure} we see that base changing the coefficients \(k\) to its algebraic closure yields a conservative functor, which one can easily check to be \(t\)-exact.
    In particular one obtains the previous theorem for any \(\bbZ_\ell\)-field.
    Using the same argument as in \cite[Lemma 3.8.]{iwahori-whittaker-realization} we can then extend this to \(\bbZ_\ell\) and finally also extend this to any flat \(\bbZ_\ell\)-algebra \(\Lambda\), since in this case both \(\bbL_G^\LSt\otimes_{\bbZ_\ell}\Lambda\) and \((\bbL_G^\LSt)^{-1}\otimes_{\bbZ_\ell}\Lambda\) are right \(t\)-exact, but \((\bbL_G^\LSt)^{-1}\otimes_{\bbZ_\ell}\Lambda\) is the right adjoint of \(\bbL_G^\LSt\otimes_{\bbZ_\ell}\Lambda\) by rigidity of \(\Mod_\Lambda\), so that \(\bbL_G^\LSt\otimes_{\bbZ_\ell}\Lambda\) is also left \(t\)-exact.
\end{rem}
\begin{rem}
    It was necessary to \(\ind\)-extend, as the \(t\)-structure on \(\ind\coh_{\nilp}^\qc(p^{-1}(\LSt))\) might not restrict to \(\coh_{\nilp}^\qc(p^{-1}(\LSt))\).
\end{rem}
\begin{rem}
    One should observe that since all the stabilizers of \(p^{-1}(\LSt)\) are linearly reductive and \(p^{-1}(\LSt)\) clearly admits an adequate moduli space (which is \(\LSt\)), that \(\ind\perf^\qc(p^{-1}(\LSt))\) in fact agrees with \(\qcoh(p^{-1}(\LSt))\) and under this equivalence the \(t\)-structures agree.
    This follows from \cite[Proposition 6.14.]{etale-local-structure-stacks} and \cite[Corollary 6.11.]{etale-local-structure-stacks}.
\end{rem}

%% file: applications.tex
\section{Applications}\label{sec: applications}
The results of this paper resolve all the conjectures in \cite{beijing-notes} about generous parameters for \(\GL_n\) and also extend them to positive characteristic.
We also resolve \cite[Conjecture 6.4.]{torsionvanishing} for \(\GL_n\) and \cite[Conjecture 6.6.]{torsionvanishing} for PEL type Shimura varieties in type A (but only for Langlands-Shahidi type parameters).
We fix a Whittaker datum \((U,\psi)\) and let \(\mathcal{W}\) be the sheaf on \(\bun_G\) attached to it.
We took the liberty to formulate some conjectures in slightly more generality than originally formulated as we consider these slight generalizations more natural.
\begin{thm}[compare with {\cite[Conjecture 6.4.]{torsionvanishing}}]\label{thm: splitting semi-orthogonal}
    Let \(\Lambda\) be any \(\bbZ_\ell\)-field.
    Let \(\varphi\) be a semisimple parameter of Langlands-Shahidi type with cuspidal support \((\widehat{L},\varphi_{\widehat{L}})\).
    Then 
    \begin{equation*}
        \D(\bun_G,\Lambda)_{\varphi}^\wedge\simeq\prod_{b\in B(G)_L}\D(\bun_G^b,\Lambda)_{\varphi}^\wedge
    \end{equation*}
    via excision and the \(!\) and \(*\) pushforwards agree for any (complex of) representation of \(G_b(\bbQ_p)\) lying in \(\D(\bun_G^b,\Lambda)_{\varphi}\).
    Hecke operators are \(t\)-exact for the perverse \(t\)-structure.
\end{thm}
\begin{proof}
    The product decomposition is immediate from \cref{lem: decomposition D(Bun_G) at Langlands-Shahidi type parameter}.    
    This formally implies that \(!\) and \(*\)-pushforwards agree.
    It is then also formal that the product decomposition is induced via excision.
    The perversity statement is immediate from \cref{lem: t-exact Langlands-Shahidi type} extended to arbitrary \(\bbZ_\ell\)-field as in \cref{rem: t-exact any field and flat zl algebra}.
\end{proof}
We also have an analogue where we localize over the entire locus of Langlands-Shahidi type parameters.
\begin{thm}\label{thm: t-exact LSt locus}
    Let \(\Lambda\) be any \(\bbZ_\ell\)-algebra.
    Let \(p^{-1}(\LSt)\) denote the open substack consisting of parameters that are of Langlands-Shahidi type with cuspidal support \(\widehat{L}\) and let \(\LSt\) denote the coarse moduli.
    Then we have
    \begin{equation*}
        \D^\LSt(\bun_G,\Lambda)\simeq\prod_{b\in B(G)_L}\D^\LSt(\bun_G^b,\Lambda),
    \end{equation*}
    where \(\D^\LSt(\bun_G,\Lambda)\defined\D(\bun_G,\Lambda)\otimes_{\Ind\perf(\locsys_{\widehat{G}}^{\mathrm{coarse}})}\Ind\perf(\LSt)\).
    This decomposition is via excision and the \(!\) and \(*\) pushforwards agree.
    For \(\Lambda\) a flat \(\bbZ_\ell\)-algebra or a \(\bbZ_\ell\)-field Hecke operators are \(t\)-exact
\end{thm}
\begin{proof}
    One can deduce the product decomposition from a careful analysis of the proof of \cref{thm: categorical equivalence}.
    This formally implies that \(!\) and \(*\)-pushforwards agree.
    Let us give another argument.
    If \(\Lambda\) is an algebraically closed field, the product decomposition is immediate from \cref{lem: decomposition D(Bun_G) at Langlands-Shahidi type parameter}.
    From this we immediately obtain the product decomposition for \(\Lambda=\bbZ_\ell\), since we need only need to check certain hom-spaces to vanish, but they are naturally objects of \(\D(\bbZ_\ell)\) so it suffices to check this after base change to the geometric points of \(\spec(\bbZ_\ell)\).
    We also have
    \begin{align*}
        \D^\LSt(\bun_G,\bbZ_\ell)\otimes_{\Mod_{\bbZ_\ell}}\Mod_{\Lambda}&\simeq\D^\LSt(\bun_G,\Lambda)\\
        \shortintertext{and}\\
        \left(\prod_{b\in B(G)_L}\D^\LSt(\bun_G^b,\bbZ_\ell)\right)\otimes_{\Mod_{\bbZ_\ell}}\Mod_{\Lambda}&\simeq\prod_{b\in B(G)_L}\D^\LSt(\bun_G^b,\Lambda)
    \end{align*}
    so we extend the results to arbitrary \(\bbZ_\ell\)-algebras.

    For \(\Lambda\) a flat \(\bbZ_\ell\)-algebra or a \(\bbZ_\ell\)-field Hecke operators are \(t\)-exact using \cref{thm: t-exact} and \cref{rem: t-exact any field and flat zl algebra}.
\end{proof}
Before discussing the application to global Shimura varieties, we will need the following lemma on Hecke operators for Weil restrictions of \(\GL_n\).
\begin{lem}
    Let \(F/E\) be a finite field extension of \(E\) and write \(G'\defined\mathrm{Res}_{F/E}\GL_n\)
    Then the Hecke operators attached to representations \(V'\in\rep(\widehat{G'})\) inflated from \(\widehat{G}\) are perverse \(t\)-exact on \(\D^\LSt(\bun_{G'},k)\) for any \(\bbZ_{\ell}\)-field \(k\) where \(\LSt\) is the locus of Langlands-Shahidi type parameteres.
\end{lem}
\begin{proof}
    We will need to distinguish variants of \(\bun_G\), \(\locsys_{\widehat{G}}\) and \(\locsys_{\widehat{G}}^\coarse\) for \(E\) and \(F\), we will do this with a subscript, so \(\bun_{G,E}\) refers to the version of \(\bun_G\) constructed using the Fargues-Fontaine curve for \(E\) and similar notation for the other objects.

    First observe that \(\bun_{G',E}\cong\bun_{G,F}\) and that under the isomorphism \(\locsys_{\widehat{G'},E}^\coarse\cong\locsys_{\widehat{G},F}^\coarse\) the locus of Langlands-Shahidi type parameters agrees.
    Therefore we have 
    \begin{equation*}
        \D^\LSt(\bun_{G',E},k)\simeq\D^\LSt(\bun_{G,F},k).
    \end{equation*}
    Let \(V\in\rep(\widehat{G})\) and \(V'\in\rep(\widehat{G'})\) be its inflation, which we consider as a representation of \(\widehat{G'}\rtimes W_F\) by letting \(W_F\) act trivially.
    Write \(V''\defined\Ind_{\widehat{G'}\rtimes W_F}^{\widehat{G'}\rtimes W_E}V'\).
    Following the discussion in the proof of \cite[Proposition IX.6.3.]{geometrization}, the Hecke operator \(T_{V''}\) acts as \(\Ind_{W_F}^{W_E}\circ T_V\) under the identification \(\D(\bun_{G',E})\simeq\D(\bun_{G,F})\) and restricting the Galois action to \(W_F\).
    Forgetting the Galois action we see that \(T_{V'}\) occurs as a direct summand of \(\Ind_{W_F}^{W_E}\circ T_V\), which gives us the claim.
\end{proof}
\begin{cor}\label{cor: t-exact weil restriction}
    In the situaton of the previous lemma the Hecke operators attached to Weyl modules for \(\widehat{G'}\) are perverse \(t\)-exact on \(\D^\LSt(\bun_{G'},k)\).
\end{cor}
\begin{proof}
    Any Weyl module for \(\widehat{G'}=\prod_{F\injto\overline{E}}\GL_n\) is obtained as an exterior product of Weyl modules for \(\GL_n\), so it is immediate from \cref{thm: t-exact LSt locus}.
\end{proof}
\begin{thm}[compare with {\cite[Conjecture 6.6.]{torsionvanishing}}, {\cite[Conjecture 1.2.]{koshikawa2021genericcohomologylocalglobal}}]
    Let \(F\) be a CM-field.
    Let \((\mathbf{G},X)\) be a PEL Shimura datum of type A such that \(\mathbf{G}_{\bbQ_p}=\prod_{i}\mathrm{Res}_{L_i/\bbQ_p}\GL_{n_i}\) with \(L_i/\bbQ_p\) unramified and assume we have \cite[Assumption 1.10.]{torsionvanishing}.\footnote{This is automatic if the Shimura variety is compact, see \cite[Remark 1.11.]{torsionvanishing} for further discussion on this point.}
    Let \(C\) be a completed algebraic closure of \(\bbQ_p\).
    We have a \(G(\bbQ_p)\)-representation 
    \begin{equation*}
        R\Gamma_c(\mathcal{S}(\mathbf{G},X)_{K^p,C},\overline{\bbF}_{\ell})
    \end{equation*}
    by passing to infinite level.
    We can localize around a parameter \(\varphi\) of Langlands-Shahidi type to obtain a complex 
    \begin{equation*}
        R\Gamma_c(\mathcal{S}(\mathbf{G},X)_{K^p,C},\overline{\bbF}_{\ell})_{\varphi}.
    \end{equation*}
    Then this is concentrated in degrees \([0,d]\) where \(d=\dim(\mathcal{S}(\mathbf{G},X)_{K^P,C})\)
\end{thm}
\begin{proof}
    As explained in \cite[§6]{torsionvanishing}, this follows from \cref{thm: splitting semi-orthogonal} and \cref{cor: t-exact weil restriction} and the analysis of \cite[§5]{torsionvanishing}.
    Note that the Hecke operator they need to be \(t\)-exact is the one associated with the Hodge cocharacter attached to \((\mathbf{G},X)\), which is miniscule, in particular the represenation attached to it is a Weyl module.
\end{proof}
\begin{thm}[{\cite[Conjecture 2.1.8.]{beijing-notes}}]\label{thm: explicit computation under equivalence}
    Suppose that \((\varphi,\rho)\) and \((b,\pi)\) match under the BM-O bijection \cite[Theorem 1.1.]{bm-o-parametrization} and that \(\varphi\) is of Langlands-Shahidi type.
    \begin{enumerate}
        \item There is an isomorphism \(i_{\varphi*}\rho*\mathcal{W}\cong i_{b!}^\ren\pi\).
        \item The natural maps \(i_{b\sharp}^\ren\pi\xto{\cong}i_{b!}^\ren\pi\xto{\cong}i_{b*}^\ren\pi\) are isomorphisms.
    \end{enumerate}
\end{thm}
\begin{proof}
    For point (1), note that \cref{thm: t-exact}, that \(i_\varphi\rho*\mathcal{W}\) is a Schur-irreducible representation of \(G_{b_{\rho,L}}\) whose \(L\)-parameter is \(\varphi\).
    We have \(b=b_{\rho,L}\) and \(\pi\) is the unique such representation.

    Point (2) to is immediate from \cref{thm: splitting semi-orthogonal}.
\end{proof}
\begin{thm}[{\cite[Conjecture 2.1.9.]{beijing-notes}}]\label{thm: hecke eigensheaf}
    Suppose that \(\varphi\) is a parameter of Langlands-Shahidi type.
    Then 
    \begin{equation*}
        \mathcal{F}_{\varphi}\defined\bigoplus_{b\in B(G),\pi\in\Pi_{\phi}(G_b)}i_{b!}^\ren\pi^{\dim(\iota_{\psi}(b,\pi))}
    \end{equation*}
    is a perverse Hecke eigensheaf with eigenvalue \(\varphi\).
\end{thm}
\begin{proof}
    This is immediate from \cref{thm: explicit computation under equivalence}.
\end{proof}
\begin{rem}
    We also obtain a version of \cref{thm: explicit computation under equivalence} and \cref{thm: hecke eigensheaf} for torsion coefficients, if we extend the BM-O bijection to torsion coefficients by requiring it to be compatible with mod-\(\ell\) reduction.
\end{rem}
\begin{rem}
    Assume \(b\) is basic and \(\pi\in\Pi(G_b)\) has supercuspidal parameter.
    Then the discussion in \cite[IX.2.]{geometrization} show that \(i_{b!}\pi\cong i_{b*}\pi\) and these sheaves are perverse.
    In particular, we see that \cite[Conjecture 3.2.6.]{beijing-notes} also holds for \(\GL_n\), using \cite[Theorem IX.7.4.]{geometrization}.
\end{rem}
\begin{thm}
    Let \(k\) be an algebraically closed field with \(p\in k^\times\).
    Let \((G,b,\mu)\) be a basic local Shimura datum with \(d=\dim\mathrm{Sht}(G,b,\mu)_K\).
    Let \(\pi\) be an irreducible supercuspidal representation in \(k\)-vector spaces.
    For \(\rho\in\rep(G_b(E))\), write 
    \begin{equation*}
        R\Gamma_c(G,b,\mu)[\rho]=\colim_{K\to\{1\}}R\Gamma_c(\mathrm{Sht}(G,b,\mu)_K,k)\otimes_{\mathcal{H}(G_b(E))}\rho.
    \end{equation*}
    Then \(R\Gamma_c(G,b,\mu)[\mathrm{JL}(\pi)]\cong\pi\boxtimes(r_{\mu}\circ\varphi_\pi\otimes|\cdot|^{-d/2})[d]\), where \(r_\mu\) is the standard representation attached to \(\mu\), \(\varphi_{\pi}\) is the semisimplified \(L\)-parameter attached to \(\pi\) and \(\mathrm{JL}(\pi)\) is the unique irreducible representation of \(G_b(E)\) whose semisimplified \(L\)-parameter is \(\varphi_\pi\).
\end{thm}
\begin{proof}
    This is clear from \cref{thm: t-exact}.
\end{proof}